\documentclass[leqno]{article}
\usepackage{amssymb,latexsym,bm,amscd,pst-all,multido,mathrsfs}
\usepackage[all]{xy}

\renewcommand{\subsection}[1]{\stepcounter{subsection} \vskip 4mm\noindent {\bf \S\thesubsection\  #1\nolinebreak.}\addcontentsline{toc}{subsection}{\thesubsection\ #1}}
\newtheorem{theorem}{Theorem}[section]
\newtheorem{lm}{Lemma}[section]
\newtheorem{prop}{Proposition}[section]
\newtheorem{de}{Definition}[section]
\newtheorem{co}{Corollary}[section]

\newtheorem{con}{Conjecture}[section]
\newenvironment{proof}{\noindent\textbf{Proof:} } {\hfill $\Box$\\ }
\newenvironment{remark}{\noindent\textbf{Remark:} } {\hfill $\Box$\\ }
\DeclareMathAlphabet{\mathsc}{OT1}{pplx}{m}{sc}

\def\k{{K\"{a}hler}}

\input epsf
\begin{document}
\hbadness=10000
\title{{\bf Convergence   of Calabi-Yau manifolds  }}
\author{Wei-Dong Ruan \thanks{Partially supported by the SRC Program of Korea Science and Engineering Foundation (KOSEF) gran  funded by the Korea government(MOST) (No. R11-2007-035-02002-0).}\\\small KAIST, Daejeon, Korea\\\small E-mail address: ruan@math.kaist.ac.kr \and Yuguang Zhang \thanks{Partially supported by the SRC program of the Korea Science and Engineering
 Foundation (KOSEF) grant funded by the Korean government (MEST) (No. R11-2007-035-02002-0), and
 by the National Natural Science Foundation of China 10771143.  }\\\small KAIST, Daejeon, Korea\\\small and\\\small Capital Normal University, Beijing, China\\\small E-mail address: yuguangzhang76@yahoo.com}
\maketitle

\begin{abstract}
In this paper, we study the convergence of Calabi-Yau manifolds under K\"{a}hler degeneration to orbifold singularities and complex degeneration to canonical singularities (including the conifold singularities), and the collapsing of a family of Calabi-Yau manifolds.
\end{abstract}

\tableofcontents

\setcounter{equation}{0}\section{Introduction}
A Calabi-Yau $n$-manifold  is a  complex  projective manifold  $M$ of complex  dimension  $n$ with trivial canonical bundle $\mathcal{K}_{M}$.
 The study of  Calabi-Yau manifolds  is  important in both mathematics and physics (c.f. \cite{Y4}).
  On a Calabi-Yau manifold, the set $\mathbb{K}_{M} $ of K\"{a}hler classes forms
  an open cone of $H^{1,1}(M, \mathbb{R})$, which is called  K\"{a}hler cone.  By Yau's  theorem
   on  the  Calabi  conjecture (\cite{Ya1}), for
    any K\"{a}hler class $\alpha\in H^{1,1}(M, \mathbb{R})$, there exists a
     unique Ricci-flat K\"{a}hler metric $g$ on $M$ with  K\"{a}hler form
     $\omega\in\alpha$.
      A natural question is to study  how a family
       of Calabi-Yau manifolds $ (M_k, g_k, \omega_k)$ with  Ricci-flat K\"{a}hler metrics and the same underlying differential manifold $M$ converges.  There are several motivations to study this question:

\begin{itemize}\label{00} \item[(i)]  On a compact Calabi-Yau
   manifold, Yau's
  theorem shows the existence of Ricci-flat K\"{a}hler metrics. However, very few of them can be written down explicitly, except for  some very special  cases,  such as the  flat torus.  It is desirable to improve our knowledge of Ricci-flat K\"{a}hler metrics on a compact Calabi-Yau manifold, for example what the manifold with these metrics looks
   like. Understanding the convergence of Calabi-Yau metrics will help us to achieve this understanding.
 \item[(ii)] In mirror symmetry,  SYZ conjecture [\cite{SYZ}]
 predicts that there is a special Lagrangian fibration on a Calabi-Yau
   manifold if it is close enough to the  large complex limit.   In
   \cite{GW2} and \cite{KS}, this conjecture was  refined by using
   the Gromov-Hausdorff convergence of a family of Ricci-flat K\"{a}hler
   metrics.
    \item[(iii)] The conifold transition (or more general geometric transition) provides  a way to connect Calabi-Yau
     threefolds with different topology in algebraic geometry  (c.f. \cite{Ro}).  Furthermore,
    it was conjectured by physicists  that this process is continuous in the space
    of all  Ricci-flat K\"{a}hler threefolds  in  \cite{Ca}. Therefore it is important and interesting to study how Calabi-Yau metrics change in this process.
 \end{itemize}

  Let $\mathfrak{M}_{M} $ denote  the space of Ricci-flat  Calabi-Yau $n$-manifolds with the same underly differential manifold   $M$. By Yau's theorem, there are two natural parameters on $\mathfrak{M}_{M} $: one is the complex structure,  and  the other is the K\"{a}hler class.  It is studied in algebraic geometry how a family of Calabi-Yau $n$-manifolds degenerates when their complex structures approach the boundary of the space of complex structures (respectively their K\"{a}hler classes approach the boundary
 of K\"{a}hler cone while fixing a complex structure).  Usually, a family of Calabi-Yau  manifolds  degenerate into  a singular projective variety in some suitable sense. In  \cite{CT},  \cite{CCT} and \cite{C}, the convergence  of  Ricci-flat K\"{a}hler manifolds  in the  Gromov-Hausdorff topology was studied without any assumptions on complex structures and K\"{a}hler classes. It is shown that  the limits are path metric spaces
 in this case.
  A natural question is, if  we know how a family of Calabi-Yau
   manifolds degenerates in the algebraic geometry sense, what can we say about their
 convergence    in the  Gromov-Hausdorff topology? Of course,   more knowledge about the limit is  expected. For example, what is the relationship
 between  the singular projective variety obtained
 from the degeneration in algebraic geometry and  the metric space obtained from the  Gromov-Hausdorff convergence?\\

 For K3 surfaces,  this question was studied in \cite{An1},    \cite{Kob1} and \cite{GW2}. If $(N, g)$  is a Ricci-flat
    K3 orbifold, it was shown in \cite{Kob1}
     that   there is a family of Ricci-flat  K\"{a}hler metrics $g_{k} $ on the crepent resolution $M$ of $N$ such that
       $(M, g_{k}) $ converges to $(N, g)$.  Then, by using the  hyper-K\"{a}hler rotation, \cite{Kob1}
        proved  that a family of  Ricci-flat  K\"{a}hler K3 surfaces   $(M_{k}, g_{k}) $ converges to $(N, g)$,
  where $M_{k}$ are obtained by a smoothing of  $N$, i.e. there is a complex $3$-manifold $ \mathcal{M}$, and a holomorphic map $\pi: \mathcal{M} \rightarrow \Delta\subset \mathbb{C}$ such that $N=\pi^{-1}(0)$ and $M_{k}=\pi^{-1}(t_{k}) $ for a family $\{t_{k}\} \subset \Delta $ with $t_{k}\rightarrow 0$. In this paper, we generalize these results to higher dimensional Calabi-Yau manifolds.\\

A Calabi-Yau $n$-variety is a normal Gorenstein projective variety $N$ of dimensional $n$
admitting only  canonical singularities, such that the dualizing sheaf $\mathcal{K}_N$ of $N$ is trivial,
(i.e. $\mathcal{K}_{N} \simeq \mathcal{O}_{N}$,) and $H^{2}(N,
\mathcal{O}_{N})=\{0\}$.  $(M, \pi)$ is called a resolution of $N$,
if $M$ is a compact  complex $n$-manifold, and $\pi : M\rightarrow N$ is a bi-rational proper
 morphism such that $\pi: M\backslash \pi^{-1}(S)\rightarrow N\backslash S$ is bi-holomorphic, where $S$ is the singular set of $N$.
 The resolution is called crepant if $\pi^* \mathcal{K}_N = \mathcal{K}_M$, i.e.  $M$ is a compact  Calabi-Yau $n$-manifold in our case.  There are analogous notions of K\"{a}hler metrics,  K\"{a}hler forms, smooth  K\"{a}hler forms and holomorphic volume forms on $N$ (see Section 2  for details).  If $\mathcal{PH}_{N}$ denotes the sheaf of pluri-harmonic functions on $N$, any K\"{a}hler form $\omega$ represents a class $[\omega]$ in $H^{1}(N, \mathcal{PH}_{N})$ (c.f. Section 5.2 in  \cite{EGZ}).   In  \cite{EGZ}, it is proved that, for any $\alpha\in H^{1}(N, \mathcal{PH}_{N})$ which can be represented by a smooth K\"{a}hler form, there is a unique Ricci-flat K\"{a}hler metric $g$ with K\"{a}hler form $\omega \in \alpha$. If $N$ admits a crepant  resolution $(M, \pi)$,  and $\alpha_{k}\in H^{1,1}(M, \mathbb{R})$  is  a family of K\"{a}hler classes such that $\lim_{k\rightarrow \infty}  \alpha_{k}=\pi^{*}\alpha$, in  \cite{To} it is proved that $g_{k}$ converges to  $\pi^{*}g$ in the $C^{\infty}$-sense on any compact subset of  $M\backslash \pi^{-1}(S)$ when $k\rightarrow \infty$,  where $g_{k}$ is the unique Ricci-flat K\"{a}hler metric with K\"{a}hler form $\omega_{k}\in \alpha_{k}$.  The first goal of the present paper is to  study the convergence of  $(M , g_{k})$  in the  Gromov-Hausdorff topology.\\

 \begin{theorem}\label{t1} Let $N$ be a  Calabi-Yau $n$-variety which admits
 a crepant  resolution $(M, \pi)$,  $\alpha \in H^{1}(N,
 \mathcal{PH}_{N})$ be a class  represented by a smooth K\"{a}hler form on $N$,   and $g$ be  the unique  singular Ricci-flat
 K\"{a}hler metric with K\"{a}hler form $\omega\in \alpha$.  Assume that the path metric structure of  $(N\backslash S, g)$  extends    to a path metric structure $d_{N}$ on
 $N$ such that the Hausdorff dimension of $S$
 satisfies $\dim_{\mathcal{H}}S\leq 2n-4$,  where $S$ is   the singular set  of $N$,  and
 $N\backslash S$ is geodesic convex in $(N, d_{N})$,  i.e.  for any $x, y\in
 N\backslash S$, there is a minimal geodesic $\gamma \subset N\backslash S$ connecting $x$ and $y$
  satisfying ${\rm length}_{g}(\gamma)=d_{N}(x,y)$. If $g_{k}$ is a family
  of Ricci-flat
 K\"{a}hler metrics on $M$ with K\"{a}hler forms $\omega_{k}$ such
 that $[\omega_{k}]\rightarrow \pi^{*}\alpha$ in $H^{1,1}(M, \mathbb{R})$  when
 $k\rightarrow \infty$, then
$$\lim_{k\rightarrow \infty}d_{GH}((M,g_{k}), (N, d_{N}))=0,
$$ where $d_{GH}$ denotes the Gromov-Hausdorff distance.
\end{theorem}

As application we use the  above theorem on calabi-Yau orbifolds. A {\sf\em projective $n$-orbifold} is a normal  projective   $n$-variety
with only quotient singularities, i.e. for any singular point $p$,
there is a neighborhood $U_{p}$ of $p$, a neighborhood $V$ of $0\in
\mathbb{C}^{n}$, and a finite group $\Gamma_{p}\subset GL(n,
\mathbb{C})$ such that $U_{p}$ is bi-holomorphic to $V/\Gamma_{p}$.  A Calabi-Yau $n$-orbifold
is a  projective  orbifold  $N$ of dimension   $n$ with the following properties:
$H^{2}(N, \mathcal{O}_{N})=\{0\}$,  $N$ admits orbifold K\"{a}hler
 metrics, all of the orbifold groups are finite subgroups of $SU(n) $, and   the canonical bundle $\mathcal{K}_{N}$ of $N$
 is  trivial. A Calabi-Yau orbifold $N$  is  a Calabi-Yau variety in the
 above sense (see Section 2  for details).
  By the same arguments   as Yau's
 proof  of the  Calabi  conjecture,  for any K\"{a}hler class $\alpha\in
 H^{1,1}(N, \mathbb{R})$ on a Calabi-Yau orbifold $N$,  there exists a unique orbifold  Ricci-flat K\"{a}hler metric $g$
 on $N$ with   K\"{a}hler form  $\omega\in\alpha$ (\cite{Ya1} and
 \cite{Kob}).
 In  \cite{Lu1}, it is
 proved that there exists a family of  Ricci-flat K\"{a}hler metrics $g_{k}$ on $\bar{M}$
 such that $\{(\bar{M}, g_{k})\}$ converges to $(T^{6}/\mathbb{Z}_{3}, h)$ in the
 Gromov-Hausdorff topology, where
 $T^{6}=\mathbb{C}^{3}/(\mathbb{Z}^{3}+\sqrt{-1}\mathbb{Z}^{3})$,
 $h$ is the flat metric on $T^{6}/\mathbb{Z}_{3}$, and $\bar{M}$ is  a crepant  resolution
 of $T^{6}/\mathbb{Z}_{3}$.  For general case,  as a corollary of
 Theorem 1.1, we obtain:

 \begin{co}\label{t2} Let $N$ be a compact
Calabi-Yau
  $n$-orbifold, which  admits a  crepant resolution $(M,
\pi)$, and   $g$ be  a Ricci-flat K\"{a}hler metric
 on $N$  with K\"{a}hler form $\omega$. If  $g_{k}$ is  a family of Ricci-flat
K\"{a}hler metrics on $M$ with K\"{a}hler forms $\omega_{k}$
such that K\"{a}hler classes  $[\omega_{k}]$ converge  to
$\pi^{*} [\omega]$ in $H^{1,1} (M, \mathbb{R})$ as
$k\rightarrow \infty$, then
$$\lim_{k\rightarrow \infty}d_{GH}((M,g_{k}), (N, g))=0,
$$ where $d_{GH}$ denotes the Gromov-Hausdorff distance.
\end{co}

This  shows that we can find   Ricci-flat K\"{a}hler metrics $g_{k}$ on $M$ such that the shape of these Ricci-flat manifolds $(M, g_{k})$ look like the Ricci-flat orbifold $(N, g)$ as close as we want.\\

 The second goal is to study the convergence of Calabi-Yau manifolds obtained from a smoothing of a Calabi-Yau variety. Let $M_{0}$ be a normal projective Calabi-Yau  $n$-variety. Assume that $M_{0}$ admits a {\sf\em smoothing}  $\pi: \mathcal{M}\rightarrow \Delta$ in
$\mathbb{CP}^{N}$ over the unit disc $\Delta=\{t\in \mathbb{C}||t|<1 \}$, i.e. $\mathcal{M} \subset \mathbb{CP}^{N}\times \Delta$ is an irreducible closed subvariety,  $\pi$ is the restriction of the projection from $\mathbb{CP}^{N}\times \Delta$ to $\Delta$, $M_{0}=\pi^{-1}(0)$, and for $t\neq 0$, $M_{t}=\pi^{-1}(t)$ is a smooth projective $n$-manifold, where $\pi^{-1} (t)$ for $t\in \Delta$ denote the scheme theoretical fibres. We also assume that the dualizing sheaf $\mathcal{K}_{\mathcal{M}} \cong \mathcal{O}_{\mathcal{M}}$. Let $\Omega = \Omega_{\mathcal{M}}$ denote the corresponding trivializing section of $\mathcal{K}_{\mathcal{M}}$. By the adjunction formula (c.f. \cite{GH}), we have $\mathcal{K}_{M_{t}} =\mathcal{K}_{\mathcal{M}} \otimes [M_{t}]|_{M_{t}} \cong \mathcal{O}_{M_{t}}$. The corresponding trivializing section can be expressed locally as $\Omega_t = \Omega_{M_t} = (\imath_{\frac{\partial}{\partial t}} \Omega)|_{M_t}$. For any $t\neq 0$,  $M_{t}$ is a projective  $n$-manifold with trivial canonical bundle $\mathcal{K}_{M_{t}}$. $\Omega$ and $\Omega_t$ define the volume forms
\[
d\mu = d\mu_{\mathcal{M}} = (-1)^{\frac{(n+1)^2}{2}} \Omega \wedge \overline{\Omega} \mbox{ and } d\mu_t = d\mu_{M_t} = (-1)^{\frac{n^2}{2}} \Omega_t \wedge \overline{\Omega_t}
\]
on $\mathcal{M}$ and $M_t$. In particular, we use $\Omega_{\mathbb{C}^n}$ to denote the standard Calabi-Yau form
 on $\mathbb{C}^n$ with the corresponding volume form $d\mu_{\mathbb{C}^n} = (-1)^{\frac{n^2}{2}} \Omega_{\mathbb{C}^n} \wedge \overline{\Omega}_{\mathbb{C}^n}$.

In our discussion, we would need the technical condition that $\mathcal{M}$ is locally homogeneous, which would include the case that $\mathcal{M}$ is smooth or with isolated homogeneous singularities (see \S3.3 for details). We believe, all our results should still be true with this technical condition removed.

Roughly speaking, we say $(\mathcal{M}, \pi)$ is {\sf\em locally quasi-homogeneous}, if for any $p\in M_0$, there exist an open neighborhood $U \subset \mathcal{M}$ with a local embedding $(U, p) \rightarrow (\mathbb{C}^m, 0)$, and a weight vector $w= (w_1, \cdots, w_m)$, where $w_i$ are positive integers, such that $(U, \pi|_U)$ is $w$-homogeneous under the standard $\mathbb{C}^*$-action on $\mathbb{C}^m$ of weight $w$. In particular, $(\mathcal{M}, \pi)$ is locally homogeneous if all $w_i=1$. For technical reason, our precise definition would require slightly stronger condition on $U$ (see \S3.3 for details).

$t= \pi(z)$ can be viewed as a holomorphic function on $\mathcal{M}$. The standard K\"{a}hler metric on $\mathbb{CP}^{N}\times \Delta$ restricts to a K\"{a}hler metric on $\mathcal{M}$. $V= -\frac{\nabla |t|}{|\nabla |t||^2}$ defines a horizontal vector field on $\mathcal{M} \setminus M_0$ such that $\pi_*V$ is the inward radial unit vector field on $\Delta$. $V$ generates a family $\phi_{t,a}: M_t \rightarrow M_{at}$ for $a\in (0, 1]$ of symplectomorphisms. It is straightforward to see that $\phi_{t,a}$ can be extended to $\phi_{t,0}: M_t \rightarrow M_0$ that is symplectomorphism over $M_0\setminus S$. This construction gives us a smooth embedding $F: (M_0\setminus S) \times \Delta \rightarrow \mathcal{M}$, $F(x,t) = F_t(x) := \phi_{t,0}^{-1} (x)$ for $x\in M_0\setminus S$ and $t\in \Delta$. (For our discussion, we would not need the symplectic property of $F$.)

By \cite{EGZ}, for any smooth K\"{a}hler form $\omega_{0}$ on $M_{0}$, there is a unique singular Ricci-flat K\"{a}hler metric $\tilde{g}_{0} $ on $M_{0}$ with
K\"{a}hler form $\tilde{\omega}_{0}$ such that $\tilde{\omega}_{0}\in [\omega_{0}]\in H^{1}(M_{0}, \mathcal{PH}_{M_{0}})$. Furthermore, $\tilde{g}_{0} $ is a smooth Ricci-flat K\"{a}hler metric on $M_{0}\backslash S$.\\

\begin{con}
\label{t0}
Let $M_{0}$ be a projective Calabi-Yau  $n$-variety, and  $S$
 be the  singular points of $M_{0}$. Assume that $M_{0}$ admits a smoothing $\pi: \mathcal{M}\rightarrow \Delta$ in $\mathbb{CP}^{N}$
  over the unit disc $\Delta\subset \mathbb{C}$ such that the dualizing sheaf $\mathcal{K}_{\mathcal{M}} $ of $\mathcal{M}$ is trivial.
   For any smooth  K\"{a}hler form $\omega$ on $\mathcal{M}$ and any $t\in
\Delta\backslash \{0\}$, let $\tilde{g}_{t}$ be the unique Ricci-flat K\"{a}hler metric on $M_{t}=\pi^{-1}(t)$ with its K\"{a}hler form $\tilde{\omega}_{t}\in [\omega|_{M_{t}}]\in
H^{1,1}(M_{t}, \mathbb{R})$.   Then for any sequence $\{t_{k}\}\subset \Delta$ with $t_{k}\rightarrow 0$  such that, for any smooth embedding $F:M_{0}\backslash S\times \Delta \rightarrow  \mathcal{M}$ satisfying  that  $F(M_{0}\backslash S\times
 \{t\})\subset M_{t}$ and $F|_{M_{0}\backslash S\times
 \{0\}}={\rm Id}: M_{0}\backslash S \rightarrow M_{0}\backslash S$ is the identity map,   we have      $$F|_{M_{0}\backslash S\times
 \{t_{k}\}}^{*}\tilde{g}_{t_{k}} \rightarrow \tilde{g}_{0}, \ \ \ \ {\rm and } \ \ \ F|_{M_{0}\backslash S\times
 \{t_{k}\}}^{*}\tilde{\omega}_{t_{k}} \rightarrow \tilde{\omega}_{0}$$
 in the $C^{\infty}$-sense on  any compact subset $K \subset
M_{0}\backslash S$,  where $\tilde{g}_{0}$ is the unique singular
Ricci-flat K\"{a}hler metric on $M_{0}$ with K\"{a}hler form $\tilde{\omega}_{0}\in [\omega|_{M_{0}}] \in H^{1}(M_{0}, \mathcal{PH}_{M_{0}})$. Furthermore, the diameters of $(M_{t_{k}}, \tilde{g}_{t_{k}})$ have a  uniformly upper bound, i.e.
$$ {\rm diam}_{\tilde{g}_{t_{k}}}(M_{t_{k}})\leq \bar{C},   $$ for a constant $\bar{C}>0$  independent of $k$.
\end{con}

We will prove this conjecture under a technical condition (related to the log canonical threshold) on the smoothing that we believe is always satisfied for the smoothing considered in conjecture \ref{t0}. We are able to verify this condition under quite general circumstances, therefore proving the conjecture in these cases. We say a smoothing $\pi: \mathcal{M} \rightarrow \Delta$ satisfies {\sf\em condition (\ref{1.1})} for $\Lambda \subset \Delta$ if for any $x_0 \in M_0$, there exist $r, c_1, C_1>0$ and a holomorphic map $\mathfrak{p}: U= B_r(x_0, \mathcal{M}) \rightarrow B_1(0) \subset \mathbb{C}^n$ that restricts to a finite branched covering $\mathfrak{p}: M_t \cap U \rightarrow B_1(0)$ for all $t\in \Delta$, and
\begin{equation}
\label{1.1}
\int_{U \cap M_t} |f|^{-2c_1} (-1)^{\frac{n^2}{2}}\Omega_t\wedge
\overline{\Omega}_t \leq C_1, \mbox{ where } f \Omega_t = \mathfrak{p}^* \Omega_{\mathbb{C}^n} \mbox{ for } t\in \Lambda.\\
\end{equation}

\begin{theorem}\label{t1.1}
The conjecture \ref{t0} is true if we assume that the smoothing $\pi: \mathcal{M} \rightarrow \Delta$ satisfies condition (\ref{1.1}) for $\Lambda = \Delta$.
\end{theorem}

\begin{remark}
For any specific example, it is usually fairly straightforward to construct $\mathfrak{p}$ and compute the explicit integral in (\ref{1.1}) to verify the condition (\ref{1.1}). (For example, the verification of the condition (\ref{1.1}) is a rather simple exercise in the conifold case.) One may even attempt to use computer to make such verification. Therefore, Theorem \ref{t1.1} can be adequately employed in proving the conjecture \ref{t0} for any specific smoothing. The difficulty lies in the verification of the condition (\ref{1.1}) in full generality, especially when $\mathcal{M}$ is singular.
\end{remark}

In general, we can prove a slightly weaker version of conjecture \ref{t0}.\\

\begin{theorem}\label{t3} The conjecture \ref{t0} is true if we assume that $\mathcal{M}$ is locally homogeneous (including when $\mathcal{M}$ is smooth) and replace ``for any sequence $\{t_{k}\}\subset \mathbb{C}$" by ``there exists a sequence $\{t_{k}\}\subset \mathbb{C}$".
\end{theorem}

If we further assume that $\pi$ possesses some local homogeneous property, the stronger version of the conjecture \ref{t0} can be proved. We say $(\mathcal{M}, \pi)$ satisfies the {\sf\em condition (1.2)} if either (i) $\mathcal{M}$ and $\pi$ are locally homogeneous, or (ii) $\mathcal{M}$ is smooth and $\pi$ is locally quasi-homogeneous.\\

\begin{theorem}
\label{t6}
The conjecture \ref{t0} is true if $(\mathcal{M}, \pi)$ satisfies the condition (1.2).
\end{theorem}

\begin{remark}
It would be clear from our proof that our method also applies to more singular $\mathcal{M}$, (especially when $\mathcal{M}$ is locally quasi-homogeneous, where the condition (1.2) becomes ``$\mathcal{M}$ and $\pi$ are locally quasi-homogeneous"). To demonstrate our method more clearly and avoid unnecessary complications, we would restrict ourself to the case when $\mathcal{M}$ is locally homogeneous (including $\mathcal{M}$ being smooth) in this paper.
\end{remark}

Now, we consider Calabi-Yau varieties   with  ``generic" singularities --- the ordinary double points. Let $M_{0}$ be a projective $n$-variety with only finite many ordinary double points $S=\{p_{\alpha}\} $ as singular points, i.e. for any $ p_{\alpha}\in S$, the singularity of $M_{0}$
is given by
$$ \{z^{2}_{1}+\cdots +z^{2}_{n+1}=0\}\subset \mathbb{C}^{n+1}.$$
Note that ordinary double points  are not orbifold singularities
when $n\geq 3$. We call $M_{0}$ a  Calabi-Yau $n$-conifold, if
$M_{0}$ is a Calabi-Yau $n$-variety.  Assume that the Calabi-Yau $n$-conifold $M_{0}$
 admits a crepant  resolution $(\hat{M},
\hat{\pi})$, and there is a smoothing of $M_{0}$ to a Calabi-Yau manifold $M$.
 The process of going from $\hat{M}$ to $M$ is called conifold transition.
 Conifolds and conifold transition appear in the literature frequently both in  mathematics and in physics
  (c.f. \cite{Ro} \cite{Ti}).
 In  mathematics,  it is related to the famous Reid's fantasy, which
conjectured that all of Calabi-Yau threefolds are connected to  each
other in some sense, and form a huge connected  web (c.f. \cite{Re}
\cite{Ro}). Furthermore, in physics, the conifold transition
provides a way to connect topologically distinct space-times  in
string theory (c.f. \cite{Ca} \cite{AGM}  \cite{CGH} \cite{GH2}
\cite{Ro}). In \cite{Ca}, it is conjectured  that there exists a
family of Ricci-flat K\"{a}hler metrics $\hat{g}_{s}$, $s\in (0,1)$,
on $\hat{M}$, and  a family of Ricci-flat K\"{a}hler metrics
$g_{s}$, $s\in (0,1)$, on $M$, which correspond to different complex
structures, satisfying that $\{( \hat{M}, \hat{g}_{s})\}$ and $\{(M,
g_{s})\}$ converge to the same limit in a suitable sense (for
example,  the Gromov-Hausdorff topology), when $s\rightarrow 0$.
This conjecture was  verified in \cite{Ca} by assuming $M_{0}$
 is the standard non-compact quadric  cone, i.e. $M_{0}=\{(z_{1}, \cdots , z_{4})\in \mathbb{C}^{4}|z^{2}_{1}+\cdots +z^{2}_{4}=0\}$.
In the  compact case, it is implied  by  \cite{To} that there exists a family of Ricci-flat K\"{a}hler metrics   $\hat{g}_{s}$ on $\hat{M}$  converging  to a Ricci-flat K\"{a}hler metric $g$ on any compact subset of the smooth part of $M_{0}$.  The next result will show the convergence of $g_{s}$ on $M$. Actually, since the  conifold singularity  is isolated homogeneous singularity, it is a corollary of theorem \ref{t6}. We will also provide a direct proof of this result in section 5.\\

\begin{co}\label{t4}  Let $M_{0}$ be a projective Calabi-Yau  $n$-conifold, then the conjecture \ref{t0} is true.\\
\end{co}

We have an analogy  of Theorem \ref{t1}.\\

\begin{co}\label{t5}  Let $M_{0}$ be a projective Calabi-Yau  $n$-variety, and  $S $
 be the  singular points of $M_{0}$. Assume that $M_{0}$ admits a
 smoothing
$\pi: \mathcal{M}\rightarrow \Delta$ in $\mathbb{CP}^{N}$ over the
unit disc $\Delta\subset \mathbb{C}$ such that the canonical bundle
$\mathcal{K}_{\mathcal{M}} $ of $\mathcal{M}$ is trivial. For any
smooth  K\"{a}hler form $\omega$ on $\mathcal{M}$ and any $t\in
\Delta\backslash \{0\}$, let $\tilde{g}_{t}$ be the unique
Ricci-flat K\"{a}hler metric on $M_{t}=\pi^{-1}(t)$ with its
K\"{a}hler form $\tilde{\omega}_{t}\in [\omega|_{M_{t}}]\in
H^{1,1}(M_{t}, \mathbb{R})$, and $\tilde{g}_{0}$ is the unique
singular Ricci-flat K\"{a}hler metric on $M_{0}$ with K\"{a}hler
form
 $\tilde{\omega}_{0}\in
[\omega|_{M_{0}}]\in H^{1}(M_{0},
 \mathcal{PH}_{M_{0}})$.
  Assume that the path metric structure of  $(M_{0}\backslash S, \tilde{g}_{0})$  extends    to a path metric structure $d_{M_{0}}$ on
 $M_{0}$ such that the Hausdorff dimension of $S$
 satisfies $\dim_{\mathcal{H}}S\leq 2n-4$,  and
 $M_{0}\backslash S$ is geodesic convex in $(M_{0}, d_{M_{0}})$,  i.e.  for any $x, y\in
 M_{0}\backslash S$, there is a minimal geodesic $\gamma \subset M_{0}\backslash S$ connecting $x$ and $y$
  satisfying ${\rm length}_{\tilde{g}_{0}}(\gamma)=d_{M_{0}}(x,y)$.
  Then there exists a sequence $\{t_{k}\}\subset \mathbb{C}$ with $t_{k}\rightarrow 0$  such that
   $$\lim_{k\rightarrow \infty}d_{GH}((M_{t_{k}},g_{t_{k}}), (M_{0}, d_{M_{0}}))=0.
$$   Furthermore, it holds for any sequence $\{t_{k}\}\subset \mathbb{C}$ with $t_{k}\rightarrow 0$,  if $M_{0}$ is a  Calabi-Yau  conifold.\\
\end{co}

Finally, we apply Corollary  \ref{t2} to study the collapsing of
Calabi-Yau manifolds. For constructing mirror manifolds, the
famous SYZ conjecture says
 that there is a special lagrangian fibration on a Calabi-Yau
   manifold if it closes  to the  large complex limit enough  (c.f.
   \cite{SYZ}).    In   \cite{GW1}, special lagrangian fibrations are
   constructed on some  Calabi-Yau threefolds of Borcea-Voisin type  with
   degenerated Ricci-flat K\"{a}hler metrics.   In
   \cite{GW2} and \cite{KS}, this conjecture was  refined to the
   following form: Let $M_{0}$ be a projective   $n$-variety (actually always  reducible in this case), and
    $\pi: \mathcal{M}\rightarrow \Delta$ be  a
 smoothing in $\mathbb{CP}^{N}$
 over the
unit disc $\Delta\subset \mathbb{C}$ such that the canonical bundle
$\mathcal{K}_{\mathcal{M}} $ of $\mathcal{M}$ is trivial. For any
smooth  K\"{a}hler form $\omega$ on $\mathcal{M}$ and any $t\in
\Delta\backslash \{0\}$, let $\tilde{g}_{t}$ be the unique
Ricci-flat K\"{a}hler metric on $M_{t}=\pi^{-1}(t)$ with its
K\"{a}hler form $\tilde{\omega}_{t}\in [\omega|_{M_{t}}]\in
H^{1,1}(M_{t}, \mathbb{R})$, and  $\bar{g}_{t}={\rm
diam}_{\tilde{g}_{t}}^{-2}(M)\tilde{g}_{t}$.   If $0\in \Delta$ is a
large complex limit point of the deformation moduli of $M_{t}$, then
$(M_{t}, \bar{g}_{t})$
   converges to a compact  metric space $(B, d_{B})$ when $t\rightarrow 0$, where $B$ is
   homeomorphic to $S^{n}$, and   $d_{B}$ is induced by a
   Riemannian metric $g_{B}$ on $B\backslash \Pi$ with a set  $\Pi  \subset
   B$ of codimension $2$. Furthermore, $B\backslash \Pi$ admits an affine manifold structure, and
    $g_{B}$ is a Monge-Amp\`{e}re metric  on $B\backslash \Pi$ (see \cite{KS} for the definitions).  This conjecture was proved for elliptic K3 surface with only $I_{1}$
    singular fibers in \cite{GW2}. It is interesting to construct some examples of Ricci-flat  Calabi-Yau manifolds of higher dimension, which
    collapse  to  metric spaces of half dimension.

   Let $X$ be a  K3 surface, which admits a  holomorphic
involution $ \iota_{1}$ such that  $ \iota_{1}^{*}\Omega=-\Omega$
for any holomorphic 2-form $\Omega$,
 $T^{2}=\mathbb{C}/(\mathbb{Z}+\sqrt{-1}\mathbb{Z})$, and $ \iota_{2}$ be the   holomorphic
involution on $T^{2}$ given by $z\mapsto -z$.  Then
$(\iota_{1},\iota_{2})$ induces a holomorphic
 $\mathbb{Z}_{2} $-action on $X\times  T^{2}$, and $X\times  T^{2}/\langle (\iota_{1},\iota_{2}) \rangle $ is a Calabi-Yau orbifold.
  If $M$ is a crepant resolution of $X\times  T^{2}/\langle (\iota_{1},\iota_{2}) \rangle $,  $M $ is called  a Calabi-Yau
   manifold  of Borcea-Voisin type (cf.  \cite{GW1}). Combining Corollary  \ref{t2} and \cite{GW2}, we obtain:

 \vskip 3mm

    \begin{theorem}\label{t6.1} There is a  family $\{(M_{k}, g_{k})\} $ of   Calabi-Yau
   3-manifolds   with Ricci-flat K\"{a}hler metrics such that $ M_{k}$ are  homeomorphic to  a Calabi-Yau manifold   $M$ of Borcea-Voisin type,  and
    $$ \lim_{k\rightarrow \infty} d_{GH}((M_{k}, g_{k}), (B, d_{B}))= 0, $$
  where    $(B, d_{B})$ is a  compact metric
space, and $B$   is homeomorphic to $S^{3}$. Furthermore,  $d_{B}$
is induced by a Riemannian  metric $g_{B}$ on $B\backslash \Pi$,
where $\Pi \subset B$ is a graph.
\end{theorem}

 The organization of the paper is as follows: In \S2, we  review some notions and results, which will be
used in this paper. In \S3, some  priori  estimates will be
obtained.   In \S4, we prove Theorem \ref{t1} and Corollary
\ref{t2}. In \S5, we prove Theorems \ref{t1.1}, \ref{t3}, \ref{t6}
and Corollaries \ref{t4}, \ref{t5}.
Finally, in \S6, we prove Theorem \ref{t6.1}. \\

\section{Preliminary}
\setcounter{equation}{0}
In this section, we review some notions and results, which will be used in this paper.

\subsection{Gromov-Hausdorff convergence}
In  \cite{G1}, Gromov introduced the notion of  Gromov-Hausdorff convergence,  which
provides a frame to study  families  of Riemannian manifolds.

\begin{de}[\cite{Fu}]
\label{2001}
For two  compact metric spaces  $(X, d_{X})$ and $(Y, d_{Y})$, a map $\psi:
 X\rightarrow Y$ is called an $\epsilon$-approximation if $Y\subset \{y\in Y| d_{Y}(y, \psi(X))<\epsilon\}$, and
 \[
 |d_{X}(x_{1}, x_{2})-d_{Y}(\psi(x_{1}), \psi(x_{2}))|<\epsilon
 \]
 for any $x_{1}$ and $ x_{2}\in X$.  The number
 \[
 d_{GH}((X, d_{X}),(Y,
 d_{Y}))=\inf \left\{\epsilon\left|\begin{array}{c} {\rm There \ are} \  \epsilon-{\rm approximations } \\ \psi:
 X \rightarrow Y,  \  {\rm and} \ \phi:Y \rightarrow  X\end{array} \right.\right\}
 \]
is called Gromov-Hausdorff distance between $(X, d_{X})$ and $(Y, d_{Y})$ (c.f.  \cite{G1}  \cite{Fu}).  The Gromov-Hausdorff distance induces a topology, the so called Gromov-Hausdorff topology,  on the space of all isometric classes of compact metric spaces. We say that  a family of compact  metric spaces  $(X_{k}, d_{X_{k}})$ convergence to a  compact metric space   $(Y, d_{Y})$ in the Gromov-Hausdorff sense, if $$\lim_{k\rightarrow\infty} d_{GH}((X_{k}, d_{X_{k}}), (Y, d_{Y}))=0. $$
\end{de}

 Let $(Y, d_{Y})$ be a compact
 metric space.  If  $\gamma: [0,1]\rightarrow Y$ is  a Lipschitz  curve, define the
   length of $\gamma$ by
\[
   {\rm length}_{d_{Y}}(\gamma)=\sup \left\{\left.\sum_{j=1}^m
   d_{Y}(\gamma(s_{j-1}), \gamma(s_{j})) \right| \mbox{ for any }
   0=s_{0}\leq \cdots \leq s_{m} =
   1\right\},
\]
    (c.f.  Chapter 1 of   \cite{G1}).  A metric space  $(Y, d_{Y})$  is a path metric space if the
 distance between each pair of points equals the infimum of the
 lengths of Lipschitz  curves joining the points (c.f.
 \cite{G1}), i.e. $$d_{Y}(y_{1},y_{2})= \inf \{{\rm length}_{d_{Y}}(\gamma)|\gamma  \  {\rm is \ a \ Lipschitz \ curve \ with} \ y_{1}=\gamma (0), \  y_{2}=\gamma (1) \}.  $$
  Clearly Riemannian manifolds are path metric spaces. In
  \cite{G1}, it is proved that a complete  metric space $(Y, d_{Y})$  is a path metric space
  if there is a family of compact path metric spaces $(X_{k},
  d_{X_{k}})$ converging to $(Y, d_{Y})$  in the Gromov-Hausdorff
  sense.  Hence we obtain a completion of the space of all compact Riemannian
  manifolds in the space of compact path metric spaces.  The following is the  famous
  Gromov pre-compactness theorem:

   \begin{theorem}[\cite{G1}]
   \label{2002}
   Let $(M_{k}, g_{k})$ be a
   family of compact Riemannian
  manifolds such that Ricci curvatures ${\rm Ric}(g_{k})\geq -C $, and
  diameters ${\rm diam}_{g_{k}}(M_{k}) \leq C'$ where $C$ and $C'$ are
  constants in-dependent of $k$. Then, a subsequence of $(M_{k},
  g_{k})$ converges to a compact path metric space $(Y, d_{Y})$  in the Gromov-Hausdorff
  sense.
  \end{theorem}

The Gromov-Hausdorff convergence of compact Riemannian
  manifolds under stronger curvature assumptions  was studied by
  various authors (c.f. \cite{G1} \cite{An2} \cite{Fu} \cite{GW}).
  For example, if  $(M_{k}, g_{k})$ is a family compact Riemannian
  manifolds with uniform  bounded sectional curvatures, uniform lower bound of
  volumes and uniform upper bound of diameters,   the  famous
  Cheeger-Gromov convergence  theorem says  that a subsequence of  $(M_{k}, g_{k})$ converges to a $C^{1,\alpha}$-Riemannian
  manifold in the $C^{1,\alpha}$-sense.
  The analogous convergence  of K\"{a}hler manifolds was studied in
  \cite{Ru1}.

  Let $(Y, d_{Y})$ be a compact path metric space.    For a closed subset $S_{Y}\subset Y$,  an integer
 $l>0$ and a $\eta >0$,  set $$\mathcal{H}^{l}_{\eta}(S_{Y})=\inf_{\{B_{d_{Y}}(p_{i},
 r_{i})\}} \varpi_{l}\sum_{i}  r_{i}^{l}, $$ where $ \{B_{d_{Y}}(p_{i},
 r_{i})\}$ is  a collection of countable metric balls such that $\bigcup_{i}B_{d_{Y}}(p_{i},
 r_{i})\supset S_{Y} $, $r_{i} < \eta$,  and $\varpi_{l}$ is the
 volume of the unit ball in $  \mathbb{R}^{l}$.  Define the
 $l$-dimensional Hausdorff measure of $S_{Y} $ by $$
 \mathcal{H}^{l}(S_{Y})=\lim_{\eta\rightarrow
 0}\mathcal{H}^{l}_{\eta}(S_{Y}).$$ The Hausdorff dimension $\dim_{\mathcal{H}}
 S_{Y}$ of $S_{Y}$ is the non-negative  number such that  $
 \mathcal{H}^{l}(S_{Y})= \infty$  for $l<\dim_{\mathcal{H}}
 S_{Y}$, and  $
 \mathcal{H}^{l}(S_{Y})=0$  for $\dim_{\mathcal{H}}
 S_{Y}<l$ (c.f. \cite{C}).

 Now let's consider compact  Ricci-flat  K\"{a}hler manifolds.
  The Gromov pre-compactness theorem  shows that a   family of
  compact Ricci-flat  K\"{a}hler manifolds with a  uniform upper bound of  diameters  converges to a compact path metric
  space by passing to a subsequence. The structure of the limit space was studied
  in \cite{CC1},   \cite{CT} and
  \cite{CCT}.

 \begin{theorem}[\cite{C}  \cite{CCT}]
 \label{2003}
 Let $(M_{k}, g_{k})$ be a
   family of compact  Ricci-flat  K\"{a}hler $n$-manifolds, and  $(Y,
   d_{Y})$ be a compact path metric space such that $$\lim_{k\rightarrow\infty} d_{GH}((M_{k}, g_{k}), (Y, d_{Y}))=0.
   $$ If $${\rm Vol}_{ g_{k}}( M_{k})\geq C_{1}>0, \ \ {\rm and } \ \ \int_{M_{k}}c_{2}(M_{k})\wedge \omega_{k}^{n-1}
   \leq C_{2},$$ for  constants $C_{1}$ and  $C_{2}$
   independent of $k$, where  $ c_{2}(M_{k})$ is the second Chern-class of $M_{k}$, and
    $\omega_{k}$ is the K\"{a}hler form of  $ g_{k}$,   there is a closed subset $S\subset Y$ with
   Hausdorff dimension $\dim_{\mathcal{H}}S\leq 2n-4 $ such that $Y\backslash S
   $ is a Ricci-flat  K\"{a}hler $n$-manifold. Furthermore, off a subset of $S$ with $(2n-4)$-dimensional Hausdorff
measure zero, $S$ has only  orbifold type singularities
$\mathbb{C}^{n-2}\times \mathbb{C}^{2}/\Gamma$, where $\Gamma$ is a
finite subgroup of $ SU(2)$.
 \end{theorem}

If $M_{k}$ are K3 surfaces in the above theorem,  \cite{An1} shows  that $Y$ is a K3 orbifold.
However, if $\dim_{\mathbb{C}}M_{k}\geq 3$, we do not know
 whether $Y$ is an  analytic  variety or not.

\subsection{Calabi-Yau  variety}
Let $N$ be a normal  projective variety of dimension $n$, which  is Cohen-Macaulay,  and $\mathcal{K}_{N}$ be the canonical sheaf of
 $N$.  All  varieties considered in this paper  are  normal and
 Cohen-Macaulay.
  We call
  $ N$  Gorenstein if   $\mathcal{K}_{N}$ is a rank one locally free
  sheaf.   We say that  $N$ has only canonical
 singularities, if  $ N$ is  Gorenstein, and,  for any resolution $\pi: M\rightarrow N$,
 $$\mathcal{K}_{M}=\pi^{*}\mathcal{K}_{N}+\sum a_{D}D,  \quad a_{D}\geq
 0,$$ where $D$ are exceptional divisors.
   A {\em\sf Calabi-Yau
$n$-variety} is a normal Gorenstein  variety $N$ of dimensional $n $
satisfying that $N$ admits only  canonical singularities, the
dualizing sheaf of $N$ is trivial, i.e.
$\mathcal{K}_{N} \simeq \mathcal{O}_{N}$, and $H^{2}(N,
\mathcal{O}_{N})=\{0\}$.  We call $(M, \pi)$ a crepant  resolution
of $N$, if
 $M$ is a compact  Calabi-Yau $n$-manifold, and $\pi :
 M\rightarrow N$ is a resolution, i.e. a   bi-rational proper  morphism  satisfying that
 $\pi: M\backslash \pi^{-1}(S)\rightarrow N\backslash S$ is
 bi-holomorphic, where $S$ is the singular set of $N$.  From the
 definition, the dualizing sheaf $\mathcal{K}_{N}$ of a Calabi-Yau
$n$-variety $N$ has a global generator $\Omega$, which is a
holomorphic volume form on $N\backslash S$ in the usual sense.  If
$(M, \pi)$ is a resolution of $N$, $\pi^{*}\Omega$ is holomorphic on
$M$. Furthermore, $\pi^{*}\Omega$ is nowhere vanishing,  if $(M,
\pi)$ is crepant. See \cite{MP} for more material  of singularities
and Calabi-Yau varieties.

\begin{prop}
\label{2.3}
Let $N \subset \mathbb{C}^m$ be an irreducible Calabi-Yau $n$-variety with the holomorphic volume form $\Omega$, and $\psi$ is a non-trivial holomorphic function on $N$. Assume $N \cap B_R$ is a closed subvariety in $B_R$. Then for any $R'<R$, there exists $\epsilon,C>0$ such that
\[
\int_{N \cap B_{R'}} \frac{d\mu}{|\psi|^{2\epsilon}} \leq C \mbox{ where } d\mu = (-1)^{\frac{n^2}{2}} \Omega \wedge \bar{\Omega}.
\]
\end{prop}
\begin{proof}
Since $N$ admits only canonical singularities, there exists a resolution $\pi: M \rightarrow N$ with normal crossing exceptional divisors such that $\pi^*\Omega$ is holomorphic and in local coordinate $\pi^* \psi(z) = z_1^{k_1} \cdots z_n^{k_n} g(z)$ with $g(z)$ nowhere zero in the local neighborhood. Then locally, there is a holomorphic function $f(z)$ such that $\displaystyle \frac{\pi^*d\mu}{|\pi^*\psi|^{2\epsilon}} = \frac{|f(z)|^2 |dzd\bar{z}|}{|z_1|^{2\epsilon k_1} \cdots |z_n|^{2\epsilon k_n}}$, whose integral converges in the local neighborhood when $\epsilon>0$ is small. By compactness of $\pi^{-1}(N \cap B_{R'})$, finitely many such local neighborhoods would cover $\pi^{-1}(N \cap B_{R'})$. Hence for $\epsilon>0$ small enough, there exists $C>0$ such that

\[
\int_{N \cap B_{R'}} \frac{d\mu}{|\psi|^{2\epsilon}} = \int_{\pi^{-1}(N \cap B_{R'})} \frac{\pi^*d\mu}{|\pi^*\psi|^{2\epsilon}} \leq C.
\]
\end{proof}

Let $N$ be a normal  projective $n$-variety with singular set $S$.
  For any $p\in S$ and a  small neighborhood $U_{p}\subset N$ of $p$,
a pluri-subharmonic function $v$ (resp. strongly pluri-subharmonic,
and pluri-harmonic) on $U_{p}$ is an upper semi-continuous function
with value in $\mathbb{R}\cup \{-\infty\}$, which is not locally
$-\infty$, and extends to a pluri-subharmonic function $ \tilde{v}$
(resp. strongly pluri-subharmonic, and pluri-harmonic) in some local
embedding $U_{p} \hookrightarrow \mathbb{C}^{m}$.  We call $v$
smooth if and only if $
 \tilde{v}$ is smooth. A continuous function $v$ is pluri-subharmonic if and only if the restriction of  $v$ to $U_{p}\backslash S $ is so  \cite{FN}.
     A  K\"{a}hler form $\omega$ (resp.   its K\"{a}hler
metric $g$) is a smooth K\"{a}hler form $\omega$ in the usual sense
on the smooth part $N\backslash S $ of $N$, and, for any singular
point $p\in S$, there is a neighborhood $U_{p}$, and a continuous
strongly  pluri-subharmonic function $v$ on
 $U_{p}$ such that
 $\omega=\sqrt{-1}\partial\overline{\partial}v$ on $U_{p}\bigcap N\backslash S
 $. We call $\omega$ (resp.    $g$) smooth if $ v$ is
 smooth in the  above sense. Otherwise, we call $\omega$ a singular K\"{a}hler form. The following property of smooth \k~forms on normal analytic variety is standard, although we could not find its precise statement in the literature.\\
 \begin{prop}
 \label{2.2}
 For any two smooth \k~metrics $g_1,g_2$ on a normal analytic variety $M$, and $p\in M$, there exists a neighborhood $U$ of $p$ such that $g_1$ is quasi-isometric to $g_2$ on $U$.
\end{prop}
\begin{proof}
For $k=1,2$, let $\omega_k$ be the \k~form of $g_k$. Since $\omega_k$ is smooth on $M$, there exists local embedding $i_k: (M,p) \hookrightarrow (\mathbb{C}^{m_k},0)$ such that $\omega_k = i_k^* \tilde{\omega}_k$ on $M$, where  $\tilde{\omega}_k$ is a smooth \k~form on $\mathbb{C}^{m_k}$. Since $M$ is normal, by results in \S7 of chapter II of \cite{De}, there exists $B_1 = B_{r_1}(0,\mathbb{C}^{m_1})$ such that the holomorphic map $i_2$ can be extend to a holomorphic map $F: (B_1,0) \rightarrow (\mathbb{C}^{m_2}, 0)$. Namely, $i_2 = F \circ i_1$. Then there exists $C_1>0$ such that $F^* \tilde{\omega}_2 \leq C_1 \tilde{\omega}_1$ on $B_1$, and $\omega_2 = i_2^* \tilde{\omega}_2 = i_1^* \circ F^* \tilde{\omega}_2 \leq C_1 i_1^*\tilde{\omega}_1 = C_1\omega_1$ on $i_1^{-1}(B_1) \subset M$. Similarly, $\omega_1 \leq C_2 \omega_2$ on $i_2^{-1}(B_2) \subset M$. Let $U := i_1^{-1}(B_1)  \cap i_2^{-1}(B_2)$. Then $C_2^{-1} \omega_1 \leq \omega_2 \leq C_1 \omega_1$ on $U$.
\end{proof}

If $\mathcal{PH}_{N}$ denotes the sheaf of pluri-harmonic functions on $N$, any K\"{a}hler form $\omega$ represents a class $[\omega]$ in $H^{1}(N,\mathcal{PH}_{N})$ (c.f. Section 5.2 in  \cite{EGZ}). Note that $H^{1}(N, \mathcal{PH}_{N})\cong H^{1,1}(N,\mathbb{R})$ if $N$ is a smooth variety.  We call a class $\alpha \in H^{1}(N,\mathcal{PH}_{N})$ a K\"{a}hler class if  $\alpha$ can be represented by a K\"{a}hler form.  A  K\"{a}hler form
 $\omega$ on a  Calabi-Yau variety $N$ is called  Ricci-flat if the  restriction of $\omega$ to the smooth part $N\backslash S
 $ is Ricci-flat.

 If $M$ is a compact Calabi-Yau manifold,    Yau's
  theorem   on  the  Calabi  conjecture (\cite{Ya1}) says that, for any K\"{a}hler class $\alpha\in
 H^{1,1}(M, \mathbb{R})$, there exists a unique Ricci-flat K\"{a}hler
 form $\omega$ representing  $\alpha$.  In (\cite{EGZ}),   Yau's
  theorem  was generalized to singular   Calabi-Yau
  varieties.

  \begin{theorem}[\cite{EGZ}]
  \label{201}
  Let $N$ be a  Calabi-Yau
  $n$-variety, $S$ be the singular set of $N$,
   and  $\omega_{0}$ be a smooth  K\"{a}hler form on $N$.  Then there is a unique Ricci-flat  K\"{a}hler form  $\omega$ with continuous potential function such that
  $\omega\in [\omega_{0}] \in H^{1}(N,
 \mathcal{PH}_{N})$, i.e. there is a unique continuous  function
 $\varphi$ on $N$ such that $\omega=\omega_{0}+ \sqrt{-1}\partial
 \overline{\partial} \varphi$ is a  K\"{a}hler form satisfying $$(\omega_{0}+ \sqrt{-1}\partial
 \overline{\partial} \varphi)^{n}=\frac{(-1)^{\frac{n^2}{2}}}{\mathcal{V}}\Omega \wedge
 \overline{\Omega}, \ \ \    \  \sup_{N} \varphi =0 ,
 $$ on the smooth part $N\backslash S$, where $\mathcal{V}=(\int_{N}\omega_{0}^{n})^{-1}\int_{N\backslash S}(-1)^{\frac{n^2}{2}}\Omega \wedge \overline{\Omega} $.
   \end{theorem}

  In \cite{ST},  singular Ricci-flat K\"{a}hler  metrics were constructed  on projective  manifolds of Kodaira
dimension $0$.  If the Calabi-Yau
  variety $N$ admits a crepant resolution $(M, \pi)$,  and $\omega_{0}$ is  a smooth  K\"{a}hler form on
  $N$,  $\pi^{*}\omega_{0} $ is a smooth  semi-positive $(1,1)$-form on
  $M$, and the class  $\pi^{*}[\omega_{0}]\in H^{1,1}(M, \mathbb{R}) $ is
  big and semi-ample.  The following convergence theorem was proved
  in \cite{To}.

 \begin{theorem}[\cite{To}]
 \label{202}
 Let $N$ be a  Calabi-Yau $n$-variety, $S$ be the singular set of $N$, and $ \alpha\in H^{1}(N,\mathcal{PH}_{N})$ be a class represented by a smooth  K\"{a}hler
 form.  Assume that  $N$ admits a crepant resolution $(M, \pi)$, and
 $\alpha_{t}$, $t\in (0,1]$, is  a family of K\"{a}hler classes on
 $M$ such that $\lim_{t\rightarrow 0}\alpha_{t}=\pi^{*}\alpha$ in $ H^{1,1}(M,
 \mathbb{R})$. Then $\omega_{t}$,  $t\in (0,1]$, $C^{\infty}$-converges to
 $\pi^{*}\omega$ on any compact subset of $ M\backslash
 \pi^{-1}(S)$, when $t\rightarrow 0$,   where  $\omega_{t}$ are Ricci-flat K\"{a}hler
 forms with $\omega_{t}\in \alpha_{t} $, and $ \omega$ is the unique Ricci-flat  K\"{a}hler
 form representing $\alpha$.
 \end{theorem}

A projective    $n$-orbifold is a normal  projective   $n$-variety
with only quotient singularities, i.e. for any singular point $p$,
there is a neighborhood $U_{p}$ of $p$, a neighborhood $V$ of $0\in
\mathbb{C}^{n}$, and a finite group $\Gamma_{p}\subset GL(n,
\mathbb{C})$ such that $U_{p}$ is bi-holomorphic to $V/\Gamma_{p}$.
We call $\Gamma_{p}$ the orbifold group of $p$. Projective orbifolds
 are   Cohen-Macaulay (c.f.  \cite{CK}).   An orbifold K\"{a}hler form $\omega$
  (resp.   the corresponding orbifold   K\"{a}hler metric $g$ ) on a projective  orbifold  $N$ is a K\"{a}hler
 form  on the smooth part of $N$, and, on any neighborhood $U_{p}$ of a singularity point
 $p$, $\omega$ is identified with a $\Gamma_{p}$-invariant K\"{a}hler
 form   on $V$ by the quotient map.  Orbifolds share many of the good properties of
manifolds. For example, De Rham cohomology  and Dolbeault cohomology
are well-defined on orbifolds,  and have most of   usual properties
on manifolds (c.f. \cite{Ba}  \cite{CK} \cite{Sa}).  An orbifold
K\"{a}hler form $\omega$ defines a $(1,1)$-class $ [\omega]$ in  $
H^{1,1}(N, \mathbb{R})$.    We call a $(1,1)$-class $\alpha$ a
K\"{a}hler class if it is represented  by  an   orbifold K\"{a}hler
form, and call the set  $\mathbb{K}_{N} $ of such classes  the
K\"{a}hler cone of $N$, which is an open cone in $ H^{1,1}(N,
\mathbb{R})$.
  Another important fact is
that an orbifold K\"{a}hler metric $g$ on an orbifold induces a path
metric space structure $d_{g}$  on $N$ (c.f. \cite{Bo}).  However,
an orbifold K\"{a}hler form  is not smooth in the  sense of smooth
K\"{a}hler
 forms on   projective   varieties.  On the other hand, a smooth K\"{a}hler
 form in the  sense of smooth K\"{a}hler
 forms on   projective   varieties is only a semi-positive
 $(1,1)$-form in the orbifold sense, but not an  orbifold
K\"{a}hler form.

 \begin{lm}\label{203}  Let $N$ be a projective    $n$-orbifold with $H^{2}(N,
 \mathcal{O}_{N}) =\{0\}$,
 and $\alpha\in H^{1,1}(N,
 \mathbb{R})$ be a class represented by an  orbifold
K\"{a}hler form.  Then  $\alpha$ can be represented by a
semi-positive  orbifold  $(1,1)$-form $\omega_{0}$, which is  a
smooth K\"{a}hler
 form in the  sense of smooth K\"{a}hler
 forms on   projective   varieties.
  \end{lm}

 \noindent{\bf Proof:} By the hypothesis, $H^{1,1}(N, \mathbb{C})= H^{2}(N, \mathbb{C}) $, $H^{1,1}(N, \mathbb{R})\cap H^{2}(N, \mathbb{Z})$ is not empty, and $H^{1,1}(N, \mathbb{R})\cap H^{2}(N, \mathbb{Q})$ is dense in $H^{1,1}(N, \mathbb{R})$. Note that, for any orbifold  K\"{a}hler form $\omega$, $[\omega]=\sum_{i=1}^{I}a_{i}\alpha_{i}$ where $\alpha_{i}\in \mathbb{K}_{N} \cap H^{2}(N, \mathbb{Q})$, and $a_{i}\in \mathbb{R}$. For any $i$,  there is an  integer $\nu_{i}>0$ such that  $\nu_{i}\alpha_{i}\in \mathbb{K}_{N}  \cap H^{2}(N, \mathbb{Z})$.   By the orbifold version of Kodaira's embedding theorem (c.f. \cite{Ba}), there is an  integer $\mu_{i}>0$ such that $\mu_{i}\nu_{i}\alpha_{i}$ induces an embedding $\iota_{\alpha_{i}}: N\hookrightarrow \mathbb{CP}^{m_{i}}$, for some  $m_{i}>0$, which satisfies   $\alpha_{i}= \frac{1}{\mu_{i}\nu_{i}} \iota_{\alpha_{i}}^{*}c_{1}(\mathcal{O}(1))$, where
  $\mathcal{O}(1)$ is the hyperplane  line bundle on $\mathbb{CP}^{m_{i}}$.
  If we denote $\omega_{FS,i}$ the  Fubini-Study metric on $\mathbb{CP}^{m_{i}}$, then
    $\omega_{0}=\sum_{i=1}^{I}a_{i}\frac{1}{\mu_{i}\nu_{i}}\iota_{\alpha_{i}}^{*}\omega_{FS,i}\in [\omega]$,  which is a smooth $(1,1)$-form in the sense of orbifold forms,
    and is a smooth  K\"{a}hler form in the  sense of smooth K\"{a}hler
 forms on   projective   varieties.
 \hfill $\Box$\\

  A Calabi-Yau $n$-orbifold
is a  projective  orbifold  $N$ of dimension   $n$ satisfying that
$H^{2}(N, \mathcal{O}_{N})=\{0\}$,  $N$ admits orbifold K\"{a}hler
 metrics, all of  orbifold groups are finite subgroups of $SU(n) $, and   the canonical bundle $\mathcal{K}_{N}$ of $N$
 is  trivial.  Note that a Calabi-Yau orbifold is Gorenstein, and,
 thus,
 has only canonical singularities (c.f. Appendix A in \cite{CK}).
 Hence a Calabi-Yau orbifold $N$  is  a Calabi-Yau variety in the
 above sense.
  By the same arguments   as Yau's
 proof  of the  Calabi  conjecture,  for any K\"{a}hler class $\alpha\in
 H^{1,1}(N, \mathbb{R})$ on a Calabi-Yau orbifold $N$,  there exists a unique orbifold  Ricci-flat K\"{a}hler metric $g$
 on $N$ with   K\"{a}hler form  $\omega\in\alpha$ (\cite{Ya1}  and
 \cite{Kob}). Note that there is a smooth  K\"{a}hler
 form $\omega_{0}$ in the  sense of smooth K\"{a}hler
 forms on   projective   varieties with $\omega_{0}\in \alpha $ by
 Lemma \ref{203}, and  $\omega$ is actually  the solution given in Theorem
 \ref{201} by the uniqueness of that  theorem. However, we know that   $\omega$
 induces a path metric space structure $d_{g}$ on $N$ in the orbifold  case  \cite{Bo}.

Let $T^{2n}=\mathbb{C}^{n}/(\mathbb{Z}^{n}+\sqrt{-1}\mathbb{Z}^{n})$, and
 $\Gamma$ be  a finite group, which has a holomorphic  action  on $T^{2n}$   preserving the flat K\"{a}hler
 form $\omega_{0}=\sqrt{-1}\sum dz_{i}\wedge d\bar{z}_{i}$ and the
 holomorphic volume form $\Omega_{0}=dz_{1}\wedge \cdots \wedge
 dz_{n}$, but not  holomorphic $2$-forms.   Then  $T^{2n}/\Gamma $ is a complex orbifold,  $\omega_{0}$ induces a flat orbifold K\"{a}hler metric on  $T^{2n}/\Gamma $,
   and $\Omega_{0} $ induces a  nowhere vanishing holomorphic $n$-form on $T^{2n}/\Gamma $, which implies the canonical
   bundle $\mathcal{K}_{T^{2n}/\Gamma} $ is trivial.   Since
 $H^{p,q}(T^{2n}/\Gamma)$ is isomorphic to the fixed subspace of $H^{p,q}(T^{2n})$  under
 the natural action $\Gamma$ on $H^{p,q}(T^{2n})$, we have
 $H^{2,0}(T^{2n}/\Gamma)=\{0\}$. Thus  $T^{2n}/\Gamma $ is a projective variety by the orbifold version of Kodaira's embedding theorem (c.f.
  \cite{Ba}), and is a
 Calabi-Yau orbifold.
   Assume that $T^{2n}/\Gamma $ admits a crepant
 resolution $(M, \pi)$.  If $n=3$, $T^{6}/\Gamma $ always  admits a crepant resolution by \cite{Roa}.
 By Yau's theorem, there are Ricci-flat
K\"{a}hler metrics on $M$, but maybe  non of them can be written
down  explicitly. However, from Corollary  \ref{t2},  for any
$\varepsilon
>0$,  we can find a Ricci-flat K\"{a}hler metric $g_{\varepsilon}$
on $M$ such that the Gromov-Hausdorff distance between $(M,
g_{\varepsilon})$ and $T^{2n}/\Gamma $ is less than $\varepsilon$.
This means that we can find Ricci-flat K\"{a}hler metrics
$g_{\varepsilon}$ on $M$ such that the Ricci-flat manifolds $(M,
g_{\varepsilon})$ look like the flat orbifold   $T^{2n}/\Gamma $ as
close as we want.

Now, we consider Calabi-Yau varieties   with   a different   type of
singularity. Let $M_{0}$ be a Calabi-Yau  $n$-variety with only
finite many ordinary double points $S=\{p_{\alpha}\} $ as singular
points, i.e. for any $ p_{\alpha}\in S$, the singularity of $M_{0}$
is given by
$$ \{z^{2}_{1}+\cdots +z^{2}_{n+1}=0\}\subset \mathbb{C}^{n+1}.$$  We call $M_{0}$ a  Calabi-Yau $n$-conifold.
Note that ordinary double points  are not orbifold singularities
when $n\geq 3$.
 Let $M_{t}\subset
\mathbb{CP}^{4}$ be the  hypersurface given by
\[
f_{t}=z_{3}\mathfrak{g}(z_{0}, \cdots, z_{4})+z_{4}\mathfrak{h}(z_{0},
\cdots, z_{4})-t(z_{0}^{5}+ \cdots  +z_{4}^{5})=0, \quad t\in
\Delta \subset \mathbb{C},
\]
where
 $\mathfrak{g}$ and $\mathfrak{h}$ are generic homogeneous
polynomials of degree $4$, and $z_{0}, \cdots, z_{4}$ are
homogeneous coordinates of $\mathbb{CP}^{4}$.     If $t=0$, $M_{0}$ is a
projective Calabi-Yau $3$-conifold with $16$ ordinary double points
as singular set $S=\{z_{3}=z_{3}=\mathfrak{g}=\mathfrak{h}=0\}$
(c.f. \cite{Ro}). If $\mathcal{M}=\{([z_{0}, \cdots,
z_{4}],t)|f_{t}=0\}\subset \mathbb{CP}^{4}\times \Delta$ and $\pi:\mathcal{M}\rightarrow \Delta$ is induced by the projection from $
  \mathbb{CP}^{4}\times \Delta$ to $\Delta$,  it is easy to check that $ (\mathcal{M}, \pi)$ is a smoothing of $M_{0}$, and the canonical bundle
   $\mathcal{K}_{\mathcal{M}}$ is trivial.  Applying Theorem \ref{t4}, we obtain that,  for   any
   $t_{k}\rightarrow 0$, and  any
smooth embedding $F:M_{0}\backslash S\times \Delta
 \rightarrow  \mathcal{M}$ such that  $F(M_{0}\backslash S\times
 \{t\})\subset M_{t}$ and $F|_{M_{0}\backslash S\times
 \{0\}}: M_{0}\backslash S \rightarrow M_{0}\backslash S$ is the identity map,  i.e. $F|_{M_{0}\backslash S\times
 \{0\}}={\rm Id} $, we have
\[
F|_{M_{0}\backslash S\times
 \{t_{k}\}}^{*}\tilde{g}_{t_{k}} \rightarrow \tilde{g}_{0}, \quad \mbox{ and } \quad F|_{M_{0}\backslash S\times \{t_{k}\}}^{*}\tilde{\omega}_{t_{k}} \rightarrow \tilde{\omega}_{0}
 \]
 in the $C^{\infty}$-sense on  any compact subset $K \subset
M_{0}\backslash S$,  where $\tilde{g}_{0}$ is the unique singular
Ricci-flat K\"{a}hler metric on $M_{0}$ with K\"{a}hler form
$\tilde{\omega}_{0}$ such that $\tilde{\omega}_{0}\in
[\omega|_{M_{0}}]\in H^{1}(M_{0},
 \mathcal{PH}_{M_{0}})$.

 Assume that the Calabi-Yau $n$-variety  $M_{0}$
 admits a crepant  resolution $(\hat{M},
\hat{\pi})$, and there is a smoothing of $M_{0}$ to a Calabi-Yau
manifold $M$.
 The process of going from $\hat{M}$ to $M$ (or from $M$ to $\hat{M}$) is called a  geometric  transition.
 The geometric  transition provides a method to  connect  two topologically distinct
 Calabi-Yau manifolds.
 In  mathematics,  it is related to the famous Reid's fantasy (c.f. \cite{Re}), and, in physics,
 this process
 connects  topologically distinct space-times  in
string theory (c.f.  \cite{Ca} \cite{AGM}  \cite{CGH} \cite{GH2}
\cite{Ro}).  In \cite{Ca},  it is conjectured  that there exists a
family of Ricci-flat K\"{a}hler metrics $\hat{g}_{s}$, $s\in (0,1)$,
on $\hat{M}$, and  a family of Ricci-flat K\"{a}hler metrics
$g_{s}$, $s\in (0,1)$, on $M$, which correspond to different complex
structures, satisfying  that $\{( \hat{M}, \hat{g}_{s})\}$ and
$\{(M, g_{s})\}$ converge to the same limit in a suitable sense, for
example in the  Gromov-Hausdorff sense,  when $s\rightarrow 0$. For
the sake of string theory, physicists conjectured that all
Calabi-Yau 3-manifolds are connected each other by preforming
geometric  transitions finite times  (c.f.   \cite{AGM} \cite{CGH}
\cite{GH2} \cite{Ro}), and form a huge  web,  which is called the
connectedness conjecture. Combing these conjectures  from
physicists, it seems that the Gromov-Hausdorff topology is a
suitable frame to present the connectedness conjecture:

\begin{con}[Metric geometry version of the  connectedness conjecture] We denote
  $(\mathcal{MET},d_{GH}) $ the set of all
isometry classes of compact metric spaces with Gromov-Hausdorff
topology, and   $\mathcal{CY}(3)\subset \mathcal{MET}$ the subset
such that each element of $ \mathcal{CY}(3)$ can be represented by a
simply connected Ricci-flat  Calabi-Yau K\"{a}hler  3-manifold $(M,
g)$ with     $Vol_{g}(M)=1 $. Then the closure $ \overline{
\mathcal{CY}(3)}$ of $ \mathcal{CY}(3)$ in $(\mathcal{MET},d_{GH}) $
is connected.

\end{con}

\subsection{Complex Monge-Amp\`{e}re Equation  and Capacities}
Let $X$ be a Stein manifold of dimension $n$, and $U$ be an open
subset of $X$.  We denote ${\rm PSH}(U)$ the space of
pluri-subharmonic functions on $U$.  If $u\in {\rm PSH}(U)$,
$\sqrt{-1}\partial \overline{\partial}u $ is a semi-positive
$(1,1)$-current on $U$. In the  pioneer work \cite{BT}, it is shown
that $(\sqrt{-1}\partial \overline{\partial}u)^{n}=\sqrt{-1}\partial
\overline{\partial}u \wedge \cdots \wedge \sqrt{-1}\partial
\overline{\partial}u $ is a well-defined  semi-positive
$(n,n)$-current on $U$, if $u\in {\rm PSH}(U)\cap L^{\infty}(U) $.
The operator $(\sqrt{-1}\partial \overline{\partial}u)^{n}$ on  the
space of locally bounded pluri-subharmonic functions is called
Monge-Amp\`{e}re operator. The following is the comparison principle
for Monge-Amp\`{e}re operators.

\begin{theorem}[\cite{BT}]\label{t2.3.1}  If $$u, v \in
{\rm PSH}(U)\cap L^{\infty}(U) , \ \ \ {\rm and }  \ \ \
\liminf_{z\rightarrow
\partial U}(u-v )(z)\geq 0,  $$ $${\rm then} \ \ \ \int_{\{u<v\}}(\sqrt{-1}\partial
\overline{\partial}v)^{n}\leq \int_{\{u<v\}}(\sqrt{-1}\partial
\overline{\partial}u)^{n}.$$
\end{theorem}

In \cite{BT}, Bedford and Taylor introduced the notion of relative
capacity, which is very  useful in the  studying of  Monge-Amp\`{e}re operators.
If $K$ is a  compact subset of $U$, the relative capacity of $K$ is
defined by
\begin{equation}\label{2.3.1}{\rm Cap_{BT}}(K,U)=\sup\{\int_{K}(\sqrt{-1}\partial
\overline{\partial}u)^{n}| u\in {\rm PSH}(U),  -1\leq u <0\}.
\end{equation} The relative  capacity has the property of decreasing under
holomorphic mappings (c.f. \cite{BT}), i.e. if $F:
U_{1}\rightarrow U_{2}$ is holomorphic, then

\begin{equation}
\label{2.3.02}
{\rm Cap_{BT}}(K,U_{1})\geq {\rm Cap_{BT}}(F(K),U_{2}).
\end{equation}

By combining Bedford-Taylor's work and Yau's solution of Calabi conjecture, \cite{Ko} solved  the
  Monge-Amp\`{e}re equation on a compact K\"{a}hler manifold under weak assumptions on the right-hand side. Particularly, a  $C^{0}$-estimate for Monge-Amp\`{e}re equations
was obtained under a very weak condition  in  \cite{Ko}.

\begin{theorem}[Lemma 2.3.1 in \cite{Ko}]\label{t2.3.2}  Let $U$ be a strictly
pseudoconvex subset of $\mathbb{C}^{n}$, and $v\in {\rm PSH}(U)$ with
$\|v\|_{L^{\infty}(U)}<C $.  Suppose that $u \in {\rm PSH}(U)\cap
L^{\infty}(U)$ satisfies the following conditions: $u<0$, $u(z)>C'$
($z\in U$), and \begin{equation}\label{2.3.2}
\int_{K}(\sqrt{-1}\partial \overline{\partial}u)^{n}\leq A
{\rm Cap_{BT}}(K,U)[h(({\rm Cap_{BT}}(K,U))^{-\frac{1}{n}})]^{-1},
\end{equation}for any compact subset $K$ of $U$, where $h: (0,
\infty)\rightarrow (1, \infty)$ is an increasing function which
fulfills the inequality $$
\int_{1}^{\infty}(yh^{\frac{1}{n}}(y))^{-1}dy<\infty.  $$ If the
sets $U(s)=\{u-s<v\}\cap U''$ are non-empty and relatively compact
in $U''\subset U' \subset\subset U$ for $s\in [S, S+D]$ then
$\inf_{U''}u $ is bounded from below by a constant depending on $A,
C, C', D, h, U', U$, but independent of $u, v, U''$.
\end{theorem}

The key argument of this theorem can be formulated into the following technical lemma that we will need later.

\begin{lm}
\label{2.1}
Assume that $a(s)$ is increasing, $t^na(s) \leq A a(s+t)/h(a(s+t)^{-\frac{1}{n}})$ for any $[s,s+t] \subset [S, S+D]$ and $\int_1^\infty (yh^{\frac{1}{n}} (y))^{-1} dy < \infty$. Then there exists $C>0$ independent of $S$ such that $a(S+D) \geq C$.
\end{lm}
\noindent{\bf Proof:} The condition on $a(s)$ can be rewritten as $t \leq \displaystyle \frac{A^{\frac{1}{n}} a(s)^{-\frac{1}{n}}}{a(s+t)^{-\frac{1}{n}} h^{\frac{1}{n}} (a(s+t)^{-\frac{1}{n}})}$. For $S = t_0 < \cdots < t_N =S+D$ such that $a(t_i)^{-\frac{1}{n}} = 2a(t_{i+1})^{-\frac{1}{n}}$ when $i\geq 1$ and $a(t_0)^{-\frac{1}{n}} \leq 2a(t_1)^{-\frac{1}{n}}$,
\[
(t_{i+1} - t_i) \leq  \frac{A^{\frac{1}{n}} a(t_i)^{-\frac{1}{n}}}{a(t_{i+1})^{-\frac{1}{n}} h^{\frac{1}{n}} (a(t_{i+1})^{-\frac{1}{n}})}.
\]
\[
0<D = \sum_{i=0}^{N-1} (t_{i+1} - t_i) \leq \sum_{i=0}^{N-1} \frac{A^{\frac{1}{n}} a(t_i)^{-\frac{1}{n}}}{a(t_{i+1})^{-\frac{1}{n}}h^{\frac{1}{n}} (a(t_{i+1})^{-\frac{1}{n}})},
\]
\[
\leq \frac{A^{\frac{1}{n}} a(t_{N-1})^{-\frac{1}{n}}}{a(S+D))^{-\frac{1}{n}}h^{\frac{1}{n}} (a(S+D)^{-\frac{1}{n}})} + \sum_{i=0}^{N-2} \frac{A^{\frac{1}{n}} a(t_i)^{-\frac{1}{n}}}{a(t_{i+1})^{-\frac{1}{n}} - a(t_{i+2})^{-\frac{1}{n}}} \frac{a(t_{i+1})^{-\frac{1}{n}} - a(t_{i+2})^{-\frac{1}{n}}}{a(t_{i+1})^{-\frac{1}{n}}h^{\frac{1}{n}} (a(t_{i+1})^{-\frac{1}{n}})},
\]
\[
\leq \frac{2A^{\frac{1}{n}}}{h^{\frac{1}{n}} (a(S+D)^{-\frac{1}{n}})} + \sum_{i=0}^{N-2} 4A^{\frac{1}{n}} \int^{a(t_{i+1})^{-\frac{1}{n}}}_{a(t_{i+2})^{-\frac{1}{n}}} \frac{dy}{yh^{\frac{1}{n}} (y)},
\]
\[
\leq \frac{2A^{\frac{1}{n}}}{h^{\frac{1}{n}} (a(S+D)^{-\frac{1}{n}})} + 4A^{\frac{1}{n}} \int^{+\infty}_{a(S+D)^{-\frac{1}{n}}} \frac{dy}{yh^{\frac{1}{n}} (y)} =: L(a(S+D)),
\]
where $\displaystyle \lim_{s\rightarrow 0} L(s)=0$. Hence there exists $C>0$ independent of $S$ such that $a(S+D) \geq C$.
\hfill $\Box$\\

By Section 2.5 in  \cite{Ko}, if there is a function
$f\in L^{p}(d\mu)$, $p>1$, such that $(\sqrt{-1}\partial
\overline{\partial}u)^{n}=fd\mu$, where $d\mu$ is the standard
Lebesgue measure, then Condition \ref{2.3.2} is satisfied. In this
case, we can choose $h(y)=(1+\log (1+y))^{2n}$.

In \cite{GZ1}, the notion of relative capacity was generalized to global capacity on a compact K\"{a}hler manifold $(M, \omega)$ of dimension $n$. For any compact subset $K\subset M$,  the global capacity of $K$ is
\[
{\rm Cap}_{\omega}(K)=\sup\left\{\int_{K}(\omega
+\sqrt{-1}\partial\overline{\partial}\psi)^{n}| \omega
+\sqrt{-1}\partial\overline{\partial}\psi \geq 0, \ \ \  0\leq \psi
\leq 1\right\}.
\]
The following properties will be used in the proof of
Theorem \ref{t3}.

\begin{prop}[Proposition 2.5 and 2.6  in \cite{GZ1}]\label{t2.3.3} Let  $(M, \omega)$ be a compact K\"{a}hler manifold of dimension $n$.
\begin{itemize}\label{00} \item[(i)] If $K\subset K'\subset M$, then
 ${\rm Cap}_{\omega}(K)\leq {\rm Cap}_{\omega}(K') $. \item[(ii)] For all $A>1$, ${\rm Cap}_{\omega}(\cdot)\leq
 {\rm Cap}_{A\omega}(\cdot)\leq A^{n} {\rm Cap}_{\omega}(\cdot)$. \item[(iii)]
 If $\psi$ is a function on $M$  satisfying that $\omega
+\sqrt{-1}\partial\overline{\partial}\psi \geq 0$, and $\psi <0$,
then $${\rm Cap}_{\omega}(\{\psi <-s\})\leq \frac{1}{s}\left(-\int_{M}\psi
\omega^{n}+n {\rm Vol}_{\omega}(M)\right), \quad \mbox{ for all } s>0.$$
\end{itemize}
\end{prop}

\begin{lm}
\label{5.3}
Fix $\chi \in C^{\infty}(M) \cap {\rm PSH}_{C_1\omega} (M)$ such that $-1\leq \chi \leq 0$, $\chi=0$ outside of the open subset $V \subset M$. For any compact subset $K \subset V$ such that $\chi=-1$ on $K$, we have
\[
{\rm Cap_{BT}} (K, V) \leq C_1^n{\rm Cap}_\omega (K).
\]
\end{lm}
\noindent{\bf Proof:} Let $u\in {\rm PSH}(V)$ with $-1\leq u <0$. $\phi = \max(u, \chi)$ is well defined on $M$ and is in ${\rm PSH}_{C_1\omega} (M)$. Clearly, $\phi = u$ on $K$.
\[
\int_{K}(\sqrt{-1}\partial\overline{\partial}u)^{n} = \int_{K} (\sqrt{-1} \partial\overline{\partial} \phi)^{n}
\]
\[
\leq \int_{K} (C_1\omega +\sqrt{-1}\partial\overline{\partial} \phi)^{n} \leq {\rm Cap}_{C_1\omega}(K) \leq C_1^n {\rm Cap}_{\omega}(K).
\]
Thus, by the definition of relative  capacity,
\[
{\rm Cap_{BT}} (K, V)\leq C_1^n {\rm Cap}_{\omega}(K).
\]
\hfill $\Box$\\

\begin{lm}
\label{5.2}
There exists $A>0$ (depending on $c_1, C_1>0$) such that for any branched covering map $\mathfrak{p}: V \rightarrow B_1 \subset \mathbb{C}^n$ of degree $\leq m$ satisfying
\[
\int_V |f|^{-2c_1} (-1)^{\frac{n^2}{2}}\Omega\wedge
\overline{\Omega} \leq C_1, \mbox{ where } f\Omega = \mathfrak{p}^* \Omega_{\mathbb{C}^n},
\]
and compact subset $K \subset V$, where $V$ is an open subset in a stein manifold $X$ with a Calabi-Yau form $\Omega$, we have
\[
\int_{K}(-1)^{\frac{n^2}{2}}\Omega\wedge \overline{\Omega}\leq Am \frac{{\rm Cap_{BT}}(K,V)}{h({\rm Cap_{BT}}(K,V)^{-\frac{1}{n}})} .
\]
\end{lm}

\noindent{\bf Proof:} Let $d\mu = (-1)^{\frac{n^2}{2}} \Omega \wedge
\overline{\Omega}$ and $d\mu_{\mathbb{C}^n} = (-1)^{\frac{n^2}{2}}
\Omega_{\mathbb{C}^n} \wedge \overline{\Omega}_{\mathbb{C}^n}$.
\[
\int_{K}d\mu  \leq \int_{\mathfrak{p}^{-1}(\mathfrak{p}(K))\cap V} d\mu \leq  m
\int_{\mathfrak{p}(K)}\frac{d\mu_{\mathbb{C}^n}}{ \min_{V\cap \mathfrak{p}^{-1}(z)}|f|^{2}}.
\]
Since
\[
\int_{\mathfrak{p}(K)}\frac{d\mu_{\mathbb{C}^n}}{ \min_{V\cap \mathfrak{p}^{-1}(z)}|f|^{2(1+\epsilon)}} \leq \int_V |f|^{-2\epsilon} d\mu \leq C_1,
\]
according to section 2.5 in \cite{Ko}, and (\ref{2.3.02}),
\[
\int_{\mathfrak{p}(K)}\frac{d\mu_{\mathbb{C}^n}}{ \min_{V\cap \mathfrak{p}^{-1}(z)}|f|^{2}} \leq A \frac{{\rm Cap_{BT}}(\mathfrak{p}(K), B_1)}{h({\rm Cap_{BT}}(\mathfrak{p}(K), B_1)^{-\frac{1}{n}})} \leq A \frac{{\rm Cap_{BT}}(K,V)}{h({\rm Cap_{BT}}(K,V)^{-\frac{1}{n}})}.
\]
\hfill $\Box$\\

\setcounter{equation}{0}\section{A priori  estimates}

\subsection{A priori estimate for diameters of Ricci-flat K\"{a}hler manifolds}
In this section, we give a priori  estimate for diameters of
Ricci-flat  K\"{a}hler manifolds, which is used in the proof of
Theorem 1.1.

\begin{theorem}\label{t3.1}   Let $(M,  \omega, g)$ be a compact
K\"{a}hler n-manifold with $c_{1}(M)=0$, and $\{g_{k}\}$ be a family
of Ricci-flat  K\"{a}hler metrics with   K\"{a}hler forms
$\omega_{k}$.
 Then there exists a  constant
 $C$
in-dependent  of $k$   such that the diameters ${\rm diam}_{g_{_{k}}}(M)$
of $(M, g_{_{k}})$  satisfy that
\begin{equation} {\rm diam}_{g_{k}}(M)\leq 32n+    C
(\int_{M}\omega_{k}\wedge \omega^{n-1})^{n}.\end{equation}
\end{theorem}

This result  is from  the second  author's thesis \cite{Z}.
 In a recent paper  \cite{To},  it  is also  obtained by  Tosatti  independently.  However,
for the completeness, we present the details of the  proof  here.
  For proving this theorem, we need a reformulation of Lemma 1.3 in
 \cite{DPS}.

\begin{lm}\label{t3.2}  Let $(M, \omega, g)$ be a compact K\"{a}hler
n-manifold, and $\{g_{k}\}$ be a family of K\"{a}hler metrics with
K\"{a}hler forms $\omega_{k}$. Then, for any $0<\delta <
{\rm Vol}_{g}(M)$, there are
 open subsets $U_{k,\delta}$ of $M$ such that \begin{equation}
{\rm Vol}_{g}(U_{k,\delta})\geq {\rm Vol}_{g}(M)-\delta, \ {\rm and \ } \
{\rm diam}_{g_{k}}^{2}(U_{k,\delta}) \leq C \delta^{-
1}\int_{M}\omega_{k}\wedge \omega^{n-1},\end{equation} where $C$ is
a constant independent of $k$. \end{lm}

The only difference between Lemma \ref{t3.2}   and  Lemma 1.3 in
\cite{DPS} is that we use the quantity   $\int_{M}\omega_{k}\wedge
\omega^{n-1} $ instead of assuming  $\omega_{k}\in [\omega]$. The
proof of the lemma is the same as the proof of Lemma 1.3 in
\cite{DPS}.  For readers' convenience, we present the sketched proof
here.\\

\noindent{\bf Proof:}  First, suppose that $K$ is a compact convex set in
some coordinate open set of $M$. On $K$,  $\omega= \frac{
\sqrt{-1}}{2}\sum g_{\alpha \overline{\beta}}dz^{\alpha}\wedge
d\overline{z}^{\beta}$, $ g_{0}={\rm Re}\sum dz^{\alpha}
d\overline{z}^{\alpha}$ and $ \omega_{0}=\frac{ \sqrt{-1}}{2}\sum
dz^{\alpha}\wedge d\overline{z}^{\alpha}$. We join $x_{1}\in K$,
$x_{2}\in K$ by the segment $[x_{1},x_{2}]\subset K$, and denote
$d\mu= \frac{(-1)^{ \frac{n}{2}}}{2^{n}n!}dz^{1}\wedge d
\overline{z}^{1}\wedge \cdots \wedge dz^{n}\wedge d
\overline{z}^{n}$ the Lebesgue measure of $K$. Note that, for any
$v\in TK$, $g_{k}(v,v)\leq tr_{g_{0}}g_{k} g_{0}(v,v)$.  By Fubini
Theorem and  Canchy-Schwarz inequality, we get \begin{eqnarray*}& &
\ \int_{K\times K}{\rm length}_{g_{k}}([x_{1},x_{2}])^{2}d\mu (x_{1})d\mu
(x_{2}) \\ & =& \ \int_{K\times K}(\int^{1}_{0}
\sqrt{g_{k}((1-t)x_{1}+tx_{2})(x_{2}-x_{1})}dt)^{2}d\mu (x_{1})d\mu
(x_{2}) \\ & \leq & \ |x_{2}-x_{1} |_{g_{0}}^{2}\int^{1}_{0}dt
\int_{K\times K} tr_{g_{0}}g_{k}((1-t)x_{1}+tx_{2})d\mu (x_{1})d\mu
(x_{2})
\\ & \leq & 2^{2n} {\rm diam}^{2}_{g_{0}}K \cdot {\rm Vol}_{g_{0}}(K)\cdot
\int_{K}\omega_{k}\wedge \omega_{0}^{n-1}
\\ & \leq & C_{K}  \int_{M}\omega_{k}\wedge \omega^{n-1},
\end{eqnarray*} where $C_{K}$ is a constant  independent of $k$. The
second inequality is obtained by   integrating first with respect to
$y=(1-t)x_{1}$ when $t\leq \frac{1}{2}$, resp. $y=tx_{2}$ when
$t\geq \frac{1}{2}$ (Note that $d\mu (x_{i})\leq 2^{2n} d\mu (y)$).
 If
\begin{equation}S=\{ (x_{1},x_{2}) \in K\times K |
{\rm length}^{2}_{g_{k}}([x_{1},x_{2}])>
\frac{C_{K}}{\delta}\int_{M}\omega_{k}\wedge \omega^{n-1}
\},\end{equation} then
$${\rm Vol}_{g_{0}\times g_{0}}(S)<\delta.$$ Let
\begin{equation}S(x_{1})=\{x_{2}\in K | (x_{1},x_{2}) \in S \}, \ \ {\rm and \ } \ Q=\{
x_{1}\in K | {\rm Vol}_{g_{0}}(S(x_{1}))\geq
\frac{1}{2}{\rm Vol}_{g_{_{0}}}(K)\}.
\end{equation}  By Fubini
Theorem, we obtain that  $$ {\rm Vol}_{g_{0}}(Q)< \frac{2\delta}
{{\rm Vol}_{g_{_{0}}}(K) }.$$ For any $x_{1},x_{2}\in K\backslash Q$, we
have ${\rm Vol}_{g_{0}}(S(x_{j}))< \frac{1}{2}{\rm Vol}_{g_{0}}(K)$. Thus
$K\backslash S(x_{1}) \bigcap K\backslash S(x_{2})$ is not empty. If
$y\in K\backslash S(x_{1}) \bigcap K\backslash S(x_{2})$, then $
(x_{1},y), (x_{2},y)\in (K\times K) \backslash S$, and
\begin{equation} {\rm length}^{2}_{g_{k}}([x_{1},y]\cup [y,x_{2}])\leq 2
\frac{C_{K}}{\delta}\int_{M}\omega_{k}\wedge
\omega^{n-1}.\end{equation} By continuity, a similar estimate still
holds for any two points $x_{1},x_{2}\in \overline{K\backslash Q}$,
with some $y\in K$. Let $K_{k,\delta}=\overline{K\backslash Q}$.
Then \begin{equation} {\rm Vol}_{g}(K\backslash K_{k,\delta})\leq
{\rm Vol}_{g}(Q) \leq C_{1}{\rm Vol}_{g_{0}}(Q)<C_{1}\frac{2\delta}
{{\rm Vol}_{g_{0}}(K) }=C_{2, K}\delta.\end{equation}

Now we cover $M$ with finitely many compact convex coordinate
patches $K_{i}$, $i=1, \cdots, N$, such that $int K_{i}\cap int
K_{i+1}$ are not empty. Then, by above arguments, there exist
$K_{i,\delta}\subset K_{i}$ with ${\rm Vol}_{g}(K_{i}\backslash
K_{i,\delta})<C_{2, K_{i}}\delta$ such that any pair of points in
$K_{i,\delta}$ can be joined by a path of length $\leq C_{
K_{i}}\delta ^{- \frac{1}{2}}(\int_{M}\omega_{k}\wedge
\omega^{n-1})^{ \frac{1}{2}}$. If we take $C_{2, K_{i}}\delta <
\frac{1}{2}{\rm Vol}_{g}(K_{i}\cap K_{i+1})$ for every $i$, then
$(K_{i}\backslash K_{i,\delta})\cup (K_{i+1}\backslash
K_{i+1,\delta})$ can not contain $K_{i}\cap K_{i+1}$ and therefore
$K_{i,\delta}\cap K_{i+1,\delta}$ are  not empty. This implies that
any $x\in K_{i,\delta}$ can be joined to any $y\in K_{j,\delta}$ by
a piecewise smooth path of length $\leq NC_{3}\delta ^{-
\frac{1}{2}}(\int_{M}\omega_{k}\wedge \omega^{n-1})^{ \frac{1}{2}}$,
where $C_{3}=\max \{C_{ K_{i}}\}$. Then we obtain the conclusion by
taking $U_{k,\delta}= \bigcup\limits_{i} K_{i,\delta}$.
\hfill $\Box$\\

\noindent{\bf Proof of Theorem \ref{t3.1}:} First, we can assume that $g$ is
a Ricci-flat K\"{a}hler metric. Then $$0={\rm Ric}_{g_{k}}-{\rm Ric}_{g}=-
\sqrt{-1}\partial \overline{\partial}log
\frac{\omega_{k}^{n}}{\omega^{n}}.$$  By Hodge Theory, there exist
constants $A_{k}$ such that
\begin{equation}\omega_{k}^{n}= e^{A_{k}}\omega^{n}.\end{equation} Then,
 we have \begin{equation}e^{A_{k}}=
\frac{\int_{M}\omega_{k}^{n}}{\int_{M}\omega^{n}}=
\frac{{\rm Vol}_{g_{k}}(M)}{{\rm Vol}_{g}(M)}.\end{equation} By Lemma
\ref{t3.2},  for any $\delta >0$,  there are
 open subsets $U_{k,\delta}$ of $M$ such that \begin{equation}
{\rm Vol}_{g}(U_{k,\delta})\geq {\rm Vol}_{g}(M)-\delta, \ {\rm and \ } \
{\rm diam}_{g_{k}}^{2}(U_{k,\delta}) \leq C \delta^{-
1}\int_{M}\omega_{k}\wedge \omega^{n-1},\end{equation} where $C$ is
a constant in-dependent of $k$. Let $p_{k}\in U_{k,\delta}$,
$\delta=\frac{1}{2}{\rm Vol}_{g}(M)$  and \begin{equation}\label{3010}
r=\max \{ 1, 2C \delta^{- \frac{1}{2}}(\int_{M}\omega_{k}\wedge
\omega^{n-1})^{\frac{1}{2}}\}.\end{equation} Thus
\begin{equation} {\rm Vol}_{g}(B_{g_{k}}(p_{k},r)) \geq
{\rm Vol}_{g}(U_{k,\delta})= \frac{1}{2}{\rm Vol}_{g}(M).\end{equation}
Therefore,
\begin{equation} {\rm Vol}_{g_{k}}(B_{g_{k}}(p_{k},r))=
\frac{1}{n!}\int_{B_{g_{k}}(p_{k},r)}\omega_{k}^{n}=e^{A_{k}}{\rm Vol}_{g}(B_{g_{k}}(p_{k},r))
\geq \frac{e^{A_{k}}}{2}{\rm Vol}_{g}(M).\end{equation} By Bishop-Gromov
theorem, \begin{equation}
\frac{{\rm Vol}_{g_{k}}(B_{g_{k}}(p_{k},1))}{{\rm Vol}_{g_{k}}(B_{g_{k}}(p_{k},r))}\geq
\frac{1}{r^{2n}}.\end{equation} Hence
\begin{equation}{\rm Vol}_{g_{k}}(B_{g_{k}}(p_{k},1))\geq
\frac{e^{A_{k}}}{2r^{2n}}{\rm Vol}_{g}(M)=
\frac{1}{2r^{2n}}{\rm Vol}_{g_{k}}(M).\end{equation} Now  we need:

\begin{lm} [Lemma 2.3 in \cite{Pa}]\label{t3.03}  \it Let $(M,g)$ be  a $2n$-dimensional  compact
Riemannian manifold with nonnegative  Ricci curvature. Then, for all
points $p\in M$ and all radiuses  $1< R< {\rm diam}_{g}(M)$, we have $$
\frac{{\rm Vol}_{g}(B_{p}(2R+2))}{{\rm Vol}_{g}(B_{p}(1))}\geq
\frac{R-1}{2n}.$$\end{lm}

See  \cite{Pa} or   Theorem 4.1 of Chapter  in \cite{SY} for the
proof. By letting $R=\frac{1}{2}{\rm diam}_{g_{k}}(M)$, we obtain
\begin{equation} {\rm diam}_{g_{k}}(M)\leq 2+ 8n
\frac{{\rm Vol}_{g_{k}}(M)}{{\rm Vol}_{g_{k}}(B_{g_{k}}(p_{k},1))}<2+
16nr^{2n}.\end{equation} Thus, by (\ref{3010}), we obtain that
$${\rm diam}_{g_{k}}(M)\leq 32n+ C (\int_{M}\omega_{k}\wedge
\omega^{n-1})^{n},
$$ where $C$ is a constant independent of $k$.
\hfill $\Box$\\

The following corollary is a direct consequence of Theorem
\ref{t3.1}  and Gromov's precompactness theorem (c.f. \cite{G1}).

\begin{co}\label{t3.4}  Let $(M,  \omega, g)$ be a compact
K\"{a}hler n-manifold with $c_{1}(M)=0$, and $\{g_{k}\}$ be a family
of Ricci-flat  K\"{a}hler metrics with   K\"{a}hler forms
$\omega_{k}$.  If $$\int_{M}\omega_{k}\wedge \omega^{n-1}\leq C, $$
for a constant $C$ independent of $k$, then a subsequence of $\{(M,
g_{k})\} $ converges to a compact metric space $(Y, d_{Y})$ in the
Gromov-Hausdorff topology.
\end{co}

\subsection{An   estimate for the first   eigenvalue}
Let $M_{0}$ be a projective  variety of  dimension $n$, $S$ be the
singular set of $M_{0}$,  and $\pi: \mathcal{M}\rightarrow \Delta$
be a smoothing  of $M_{0}$  in $\mathbb{CP}^{N}$ over the unit disc
$\Delta=\{t\in \mathbb{C}| |t|<1 \}$ as defined in the introduction. Our definition implies that $\pi: \mathcal{M}\backslash S \rightarrow \Delta$ is a smooth fibration. Then since only the central fibre is singular by definition, we have that $\mathcal{M}$
is a complex subvariety in $\mathbb{CP}^{N}\times \Delta $ of dimension $n+1$ with singular set $\mathcal{S}\subset S \subset M_0$.

We denote $g_{FS} $ the Fubini-Study metric on $\mathbb{CP}^{N}$.
Let $\bar{g}=(g_{FS}+\sqrt{-1}
\partial\overline{\partial}|t|^{2})|_{\mathcal{M}}$, and $\bar{g}_{t}=\bar{g}|_{M_{t}}$. By using  Li-Tian's  estimate
 on heat kernels  (\cite{PT}) and Davis' result (\cite{D}),   there is a
uniform Sobolev constant
  on all $(M_{t}, \bar{g}_{t})$, i.e. there is a constant $\bar{C}_{S}>0$ independent
of $t$ such that,   for any $t\neq 0$, and
 any  smooth function $\chi$ on $M_{t}$, $$ \|\chi\|_{L^{\frac{4n}{2n-2}}(\bar{g}_{t})}\leq \bar{C}_{S}( \|d\chi\|_{L^{2}(\bar{g}_{t})}
+ \|\chi\|_{L^{2}(\bar{g}_{t})}),$$ (c.f. \cite{Yo}).

\begin{prop}\label{t3.500} If   $g$ is   a smooth  K\"ahler  metric on
the normal analytic variety $\mathcal{M}$, and  $g_{t}=g|_{M_{t}}$, then for any $c \in (0,1)$, there is a uniform Sobolev constant $C_{S}>0$ on  $(M_{t}, g_{t})$ independent of $t$ satisfying $0< |t| \leq c$, i.e. for any such $t$, and
 any  smooth function $\chi$ on $M_{t}$, $$ \|\chi\|_{L^{\frac{4n}{2n-2}}(g_{t})}\leq C_{S}( \|d\chi\|_{L^{2}(g_{t})}
+ \|\chi\|_{L^{2}(g_{t})}).$$
\end{prop}
\begin{proof}
 Since $g, \bar{g}$ are smooth, $\mathcal{M}$ is normal and $\mathcal{M} \cap \{|t|\leq c\}$ is compact, Proposition \ref{2.2} implies that there is a constant $C>0$ such that $C^{-1}\bar{g} \leq g \leq C \bar{g}$ on $\mathcal{M} \cap \{|t|\leq c\}$. Then $C^{-1}\bar{g}_{t} \leq g_{t} \leq C \bar{g}_{t}$ for $|t| \leq c$. As consequence, we obtain a uniform Sobolev constant $C_{S}>0$ on  $(M_{t}, g_{t})$ independent of $t$ satisfying $0< |t| \leq c$.
\end{proof}

\begin{prop}\label{t3.5}
Let $M_{0}$ be an irreducible  projective  variety of  dimension
$n$, and $\pi: \mathcal{M}\rightarrow \Delta$ be a smoothing  of
$M_{0}$ in $\mathbb{CP}^{N}$ over the unit disc $\Delta=\{t\in
\mathbb{C}| |t|<1 \}$. If   $g$ is  a smooth   K\"ahler  metric on
$\mathcal{M}$, and  $g_{t}=g|_{M_{t}}$, then there is a constant
$C>0$ independent of $t$ such that $$ \lambda_{1,t}>C,$$ where $
\lambda_{1,t}$ is  the first  eigenvalue  of the Laplacian
$\Delta_{t} $ on $(M_{t}, g_{t})$.
\end{prop}

 This result can  be obtained from  the main theorem in
 \cite{Yo}.  However,  for the
completeness, we give an independent   proof here.\\

\noindent{\bf Proof:}  If it is not true, then there exists $t_k (\in \Delta) \rightarrow 0$ such that $\lambda_{1,k} = \lambda_{1,t_k} \rightarrow 0$ with eigenfunctions $\phi_k$ satisfying $\Delta_{t_k} \phi_k = - \lambda_{1,t_k} \phi_k$.

\[
\int_{M_{t_k}} \phi_k =0,\quad \int_{M_{t_k}} |\phi_k|^2 = 1.
\]
\[
\|\phi_k\|_{L^{\frac{2n}{n-1}} (g_{t_k})} \leq C_S (\|d\phi_k\|_{L^2 (g_{t_k})} + \|\phi_k\|_{L^2 (g_{t_k})})
\]
\[
= C_S (1 + \lambda_{1,k}^{\frac{1}{2}}) \|\phi_k\|_{L^2 (g_{t_k})}=C_S (1 + \lambda_{1,k}^{\frac{1}{2}}).
\]
By Proposition \ref{t3.500}, the above Sobolev constant $C_S$ is independent of $k$. For any compact set $K \subset M_0 \setminus S$, $F_{t_k}^* g_{t_k}$ $C^\infty$ converges to $g_0$ on $K$. $\{F_{t_k}^* \phi_k\}$ is bounded in $W^{1,2} (K)$, therefore is weakly relative compact by Banach-Alaoglu theorem. May assume it weakly converges to $\phi_0\in W^{1,2} (K, g_0)$, and the convergence is strong in $L^2(K, g_0)$. By lower semi-continuity of norm under weak limit,

\[
0 \leq \|d\phi_0\|_{L^2 (K, g_0)} \leq \lim_{k\rightarrow \infty}
\|dF_{t_k}^* \phi_k\|_{L^2 (K, g_0)} \leq \lim_{k\rightarrow \infty}
\|d \phi_k\|_{L^2 (M_{t_k}, g_{t_k})} = \lim_{k\rightarrow \infty}
\lambda_{1,k}^{\frac{1}{2}} =0
\]
Hence $\phi_0$ is locally constant on $K$. Since $M_0$ is irreducible, we may assume $K$ is connected. Then $\phi_0$ is a constant on $K$.

\[
\left| \int_{M_{t_k} \setminus F_{t_k} (K)} \phi_k \right| \leq \|\phi_k\|_{L^2 (g_{t_k})} |{\rm Vol}_{g_{t_k}} (M_{t_k} \setminus F_{t_k} (K))|^{\frac{1}{2}} = |{\rm Vol}_{g_{t_k}} (M_{t_k} \setminus F_{t_k} (K))|^{\frac{1}{2}}
\]
\[
\int_{M_{t_k} \setminus F_{t_k} (K)} |\phi_k|^2 \leq \|\phi_k\|^2_{L^{\frac{2n}{n-1}} (g_{t_k})} |{\rm Vol}_{g_{t_k}} (M_{t_k} \setminus F_{t_k} (K))|^{\frac{1}{n}}
\]
\[
\leq C_S (1 + \lambda_{1,k}^{\frac{1}{2}}) |{\rm Vol}_{g_{t_k}} (M_{t_k} \setminus F_{t_k} (K))|^{\frac{1}{n}}
\]
\[
0 = \lim_{k\rightarrow \infty} \int_{M_{t_k}} \phi_k = \lim_{k\rightarrow \infty} \int_{M_{t_k} \setminus F_{t_k} (K)} \phi_k + \lim_{k\rightarrow \infty} \int_K F_{t_k}^* \phi_k
\]
\[
{\rm Vol}_{g_0} (K) |\phi_0| = \lim_{k\rightarrow \infty} \left|\int_K F_{t_k}^* \phi_k\right| = \lim_{k\rightarrow \infty} \left|\int_{M_{t_k} \setminus F_{t_k} (K)} \phi_k\right| \leq |{\rm Vol}_{g_0} (M_0 \setminus K)|^{\frac{1}{2}}.
\]
\[
 1= \lim_{k\rightarrow \infty} \int_{M_{t_k}} |\phi_k|^2 = \lim_{k\rightarrow \infty} \int_{M_{t_k} \setminus F_{t_k} (K)} |\phi_k|^2 + \lim_{k\rightarrow \infty} \int_K |F_{t_k}^* \phi_k|^2
\]
\[
\leq C_S |{\rm Vol}_{g_0} (M_0 \setminus K)|^{\frac{1}{n}} + {\rm Vol}_{g_0} (K) |\phi_0|^2.
\]
\[
\leq C_S |{\rm Vol}_{g_0} (M_0 \setminus K)|^{\frac{1}{n}} + {\rm Vol}_{g_0} (M_0 \setminus K)/{\rm Vol}_{g_0} (K).
\]
This is a contradiction when $K$ is chosen large enough.
\hfill $\Box$\\

\begin{remark}
If we remove the hypothesis that  $M_{0}$ is irreducible in the
above proposition, we obtain
\[
\lim_{t\rightarrow 0}\lambda_{m-1,t}=0, \quad \mbox{ and } \quad \lambda_{m,t}>C,
\]
for a constant $C>0$ independent of $t$, where $m\geq 1$ is the number
of irreducible components of $M_{0}$ by the main theorem in
\cite{Yo}.
\end{remark}

The following lemma will be used in the proof of Theorem \ref{t3}.

\begin{lm}\label{t3.6}
 Let $M_{0}$ be an irreducible  projective  variety of  dimension
$n$, which admits a smoothing  $\pi: \mathcal{M}\rightarrow \Delta$
in $\mathbb{CP}^{N}$ over the unit disc $\Delta=\{t\in \mathbb{C}|
|t|<1 \}$. Let   $g$ be   a smooth  K\"ahler metric on
$\mathcal{M}$, $g_{t}=g|_{M_{t}}$, and $\omega_{t}$ be the K\"ahler
form of $ g_{t}$. For any $t\neq 0$, if $\varphi_{t}$ is a smooth
function satisfying that $\omega_{t}+\sqrt{-1}\partial
\overline{\partial}\varphi_{t}  $ is a K\"ahler form  on
$M_{t}=\pi^{-1}(t)$, and $\sup_{M_{t}}\varphi_{t}=0 $, then   there
is a constant $C>0$ independent of $t$ such that
 $$ \int_{M_{t}}\varphi_{t}\omega_{t}^{n} \geq -C.$$
\end{lm}

\noindent{\bf Proof:} We assume  ${\rm Vol}_{\omega_{t}}(M_{t})=\frac{1}{n!}\int_{M_{t}}\omega_{t}^{n}
=\frac{1}{n!}$ for convenience.
  If
$H_{t}(x,y,s)$ denotes the heat kernel on $(M_{t}, \omega_{t})$, and
$K_{t}(x,y,s)=H_{t}(x,y,s)-n! $, then the Green function on $(M_{t},
\omega_{t})$ is $G_{t}(x,y)=\int_{1}^{\infty}K_{t}(x,y,s)ds$. Note
that
\begin{equation}\label{34.2}K_{t}(x,y,s)\geq
-K_{t}^{\frac{1}{2}}(x,x,s)K_{t}^{\frac{1}{2}}(y,y,s),\end{equation}
and

\begin{equation}\label{34.3}K_{t}(x,x,s)\leq K_{t}(x,x,1) e^{-\lambda_{1,t}(s-1)},\end{equation}
where
$\lambda_{1,t}>0$ is the first eigenvalue of the Laplacian on
$(M_{t}, \omega_{t})$  (c.f. Lemma 3.1 in \cite{Lu1} and \cite{CL}).

Since $M_{0}$ has only one irreducible component, there is a
constant $C>0$ independent of $t$
 such that $$\lambda_{1,t} \geq C,
$$ by Proposition  \ref{t3.5}.    For  any smooth  function $\chi$ on $M_{t}$ with
$\int_{M_{t}}\chi\omega_{t}^{n}=0$, we have
$$\int_{M_{t}}|d\chi|^{2}\omega_{t}^{n}\geq
\lambda_{1,t}\int_{M_{t}}\chi^{2}\omega_{t}^{n}\geq
C\int_{M_{t}}\chi^{2}\omega_{t}^{n}.$$ Then, by Proposition
\ref{t3.500}, we have a uniformly Sobolev inequality
$$ \|\chi\|_{L^{\frac{4n}{2n-2}}(\omega_{t})}\leq C_{S}( \|d\chi\|_{L^{2}(\omega_{t})}
+ \|\chi\|_{L^{2}(\omega_{t})})=C_{S}(1+
\lambda_{1,t}^{-\frac{1}{2}})\|d\chi\|_{L^{2}(\omega_{t})}\leq
\bar{C}_{S} \|d\chi\|_{L^{2}(\omega_{t})},$$ for a constant
$\bar{C}_{S}>0$ independent of $t$.  Since
$\int_{M_{t}}K_{t}(x,y,s)\omega_{t}^{n}(y)=0$, by the same arguments
as the proof of Equation (3.12) in \cite{Y2},
$$K_{t}(x,x,1)\leq n^{n}\bar{C}_{S}^{n}.$$ Thus, by (\ref{34.2}) and (\ref{34.3}),  there is a constant $\bar{C}>0$
 independent of  $t$  such that $$G_{t}(x,y)=\int_{1}^{\infty}K_{t}(x,y,s)ds\geq
-n^{n}\bar{C}_{S}^{n}\int_{1}^{\infty}e^{-\lambda_{1,t}(s-1)}ds=-n^{n}\bar{C}_{S}^{n}
\frac{1}{\lambda_{1,t}}\geq -\bar{C}.
$$

If  $\tilde{G}_{t}(x,y)$ is  the normalized Green function such that
$\inf_{M_{t}}\tilde{G}_{t}(x,y)=0$, then
$$\int_{M_{t}}\tilde{G}_{t}(x,y)\omega_{t}^{n}\leq C, $$ for a
constant $C>0$ independent of $t$.  Note that
$n+\Delta_{t}\varphi_{t}\geq 0$ where $\Delta_{t}$ is the Laplacian
of $(M_{t}, \omega_{t})$. By Green's formula, we obtain
$$\varphi_{t}(x)-\int_{M_{t}}\varphi_{t}\omega_{t}^{n}=-\frac{1}{n!}\int_{M_{t}}\tilde{G}_{t}(x,y)
\Delta_{t}\varphi_{t}\omega_{t}^{n}\leq nC.$$ By letting
$\varphi_{t}(x)=\sup_{M_{t}}\varphi_{t}=0$, we obtain the
conclusion.
\hfill $\Box$\\

\subsection{Estimates concerning the condition (\ref{1.1})}
Recall that $d\mu = d\mu_{\mathcal{M}} = (-1)^{\frac{(n+1)^2}{2}} \Omega \wedge \overline{\Omega}$
and $d\mu_t = d\mu_{M_t} = (-1)^{\frac{n^2}{2}} \Omega_t \wedge \overline{\Omega}_t$.\\

\begin{lm}
\label{3.1}
For any $c\geq 0$ and holomorphic function $f$ on $\mathcal{M}$,
\[
b(t) := \int_{M_t} |f|^{-2c} d\mu_t
\]
is lower semi-continuous on $\Delta$. In particular, there exists $C>0$ such that
\[
\mathcal{V}_t := \int_{M_t} d\mu_t \geq C \mbox{ for } t\in \Delta.
\]
\end{lm}
\noindent{\bf Proof:} For any $t_0 \in \Delta$, first assume that $f|_{M_{t_0}}$ is not identically zero. Then there exist compact subsets $K_1 \subset \cdots \subset K_i \subset \cdots \subset M_{t_0}$ such that the integrant of $b(t_0)$ is finite and continuous on $K_i \cap M_{t_0}$ and
\[
b(t_0) = \sup_i \int_{K_i \cap M_{t_0}} |f|^{-2c}d\mu_{t_0}.
\]

The integrant of $b(t)$ is continuous on an open neighborhood of $K_i \cap M_{t_0} \subset \mathcal{M}$. Hence for fixed $i$,
\[
\int_{K_i \cap M_{t_0}} |f|^{-2c} d\mu_{t_0} = \lim_{t \rightarrow t_0} \int_{K_i \cap M_{t}} |f|^{-2c} d\mu_{t} \leq \liminf_{t \rightarrow t_0} b(t).
\]
Then
\[
b(t_0) = \sup_i \int_{K_i \cap M_{t_0}} |f|^{-2c} d\mu_{t_0} \leq \liminf_{t \rightarrow t_0} b(t).
\]
Namely, $b(t)$ is lower semi-continuous on $\Delta$. Since $b(t)>0$ for any $t$, there exists $C>0$ such that $b(t)>C$ for $t\in \Delta$. In particular, for $c=0$, $\mathcal{V}_t>C$ for $t\in \Delta$.

For the case $c>0$ and $f|_{M_{t_0}} \equiv 0$, $b(t_0)=+\infty$. By $\mathcal{V}_t>C>0$, it is easy to see that $\displaystyle\lim_{t\rightarrow t_0} b(t)=+\infty$.
\hfill $\Box$\\

Assume that $M \subset \mathbb{C}^m$ is a closed analytic subvariety of $B_R \subset\mathbb{C}^m$ for sufficiently large $R>0$. ($M$ would be a local neighborhood of our $\mathcal{M}$. In this subsection, $M$ is considered a metric subspace of $\mathbb{C}^m$ with the standard metric. The proof of Proposition \ref{2.2} implies that such metrics on $M$ would be mutually quasi-isometric for different embeddings.) For any closed subset $D \subset \mathbb{C}^m$, define $\|f\|_D := \sup_{z\in D} |f(z)|$. A conic family of holomorphic functions on $D$ is called {\sf\em projectively compact} (resp. {\sf\em pre-compact}), if (resp. the closure of) any closed subset of the family, bounded under $\|\cdot\|_D$, is compact.

Let $M$ be a normal analytic variety, then it is locally irreducible. There is a canonical stratification $M = \displaystyle\bigcup_{i=0}^n M^{(i)}$, $M^{(i)}$ is a $i$-dimensional open manifold. ${\rm Sing} (M) = \displaystyle\bigcup_{i=0}^{n-1} M^{(i)}$ is the singular part of $M$, $M^{(n)}$ is the smooth part of $M$, and $M = M^{(n)} \cup {\rm Sing} (M)$. We say $M$ is {\sf\em locally homogeneous}, if for any $p\in M^{(i)}$, there is an open neighborhood $U$ of $p$ in $M$ and an isomorphism $U \cong (U \cap M^{(i)}) \times U_i$, where $U_i \subset \mathbb{C}^{m_i}$ is a homogeneous subvariety. For example, a homogeneous subvariety in $\mathbb{C}^m$ with isolated singularity at the origin is a locally homogeneous variety.\\

\begin{lm}
\label{cy3}
Let $M \subset \mathbb{C}^m$ be a homogeneous subvariety, and $P$ be a projectively pre-compact family of holomorphic function on $B_1 \cap M$, then $\tilde{P}$ consists of $\tilde{f} (z) = f(rz)$ for $f\in P$ and $0 \leq r \leq 1$, is a projectively pre-compact family of holomorphic function on $B_1 \cap M$.
\end{lm}
\noindent{\bf Proof:} Assume $(r_k,f_k) \rightarrow (r_0,f_0)$, and $l_0\geq 0$ is the smallest integer such that the degree $l_0$ term $f_0^{[l_0]} \not=0$. Let $\tilde{f}_k (z) = c_kf_k(r_kz)$ so that $\|\tilde{f}_k\|_{M \cap B_1} =1$. $\{\tilde{f}_k\}$ is clearly projectively pre-compact when $r_0\not=0$. When $r_0=0$, one has $\tilde{f}_k^{[>l_0]} \rightarrow 0$. On the other hand, $\{\tilde{f}_k^{[\leq l_0]}\}$ being a subset of a finite dimensional vector space is clearly projectively pre-compact.
\hfill $\Box$\\

\begin{lm}
\label{cy4}
Let $P$ be a projectively pre-compact family of holomorphic function on $B_1 \cap M$, which is irreducible. Then for any $D \subset M \cap B_1$ with non-empty interior, $P_D$ consisting of $f|_D$ for $f\in P$ is a projectively pre-compact family of holomorphic functions on $D$.
\end{lm}
\noindent{\bf Proof:} For any $\{f_k\} \subset P$, by taking subsequence and scaling, we may assume $f_k \rightarrow f_0 \not\equiv 0$. Then $f_k|_D \rightarrow f_0|_D$. Since $M$ is irreducible and $D$ has nonempty interior, we have $f_0|_D \not=0$. Hence $P_D$ is projectively pre-compact.
\hfill $\Box$\\

\begin{lm}
\label{cy5}
Let $M \subset \mathbb{C}^m$ be a homogeneous subvariety,  $(B'_a,0) \subset (\mathbb{C}^{m'},0)$ be a ball, and $P$ be a projectively pre-compact family of holomorphic function on $B'_a \times (B_1 \cap M) \subset \mathbb{C}^{m'+m}$, then $\tilde{P}$ consists of $\tilde{f} (z) = f(p+rz)$ for $f\in P$, $p\in B'_a$ and $0 \leq r + |p|/a \leq 1/2$, is a projectively pre-compact family of holomorphic function on $B'_{a/2} \times (B_1 \cap M)$.
\end{lm}
\noindent{\bf Proof:} Assume $(p_k, r_k,f_k) \rightarrow (p_0, r_0,f_0)$ with $r_k + |p_k|/a \leq 1/2$. Let $\hat{f}_k (z) = f_k(z +p_k)$. Since $\{f_k\}$ is uniformly continuous, $(0, r_k, \hat{f}_k) \rightarrow (0, r_0,\hat{f}_0)$, $\{\hat{f}_k\}$ is pre-compact on $B'_{a/2} \times (B_1 \cap M)$, and $\mathbb{C}^{m'} \times M$ is homogeneous in $\mathbb{C}^{m'+m}$, lemma \ref{cy3} implies that $\{\tilde{f}_k\}$ is also projectively pre-compact on $B'_{a/2} \times (B_1 \cap M)$, where $\tilde{f}_k (z) = c_kf_k(p_k + r_kz) = c_k\hat{f}_k ( r_kz)$ so that $\|\tilde{f}_k\|_{B'_{a/2} \times (B_1 \cap M)} =1$.
\hfill $\Box$\\

\begin{lm}
\label{cy6}
Let $M \subset \mathbb{C}^m$ be a subvariety. For $R>1$ and a projectively pre-compact family $P$ of holomorphic function on $M \cap D$, where $B_R \subset D$ and $M \cap D$ is irreducible, there exists $C_1 >0$, such that for any $f \in P$ satisfying $f\not=0$ on $M \cap B_1$, $|f(0)|=1$, we have $\|f\|_{M \cap B_R} \leq C_1$.
\end{lm}
\noindent{\bf Proof:} If the assertion is not true, then there exists $f_k \in P$, such that $\|f_k\|_{M \cap B_R} \rightarrow +\infty$. Consequently, $m_k = \|f_k\|_{M \cap D} \rightarrow +\infty$. Let $g_k = f_k/m_k$. Then  $\|g_k\|_{M \cap D} =1$. Since $P$ is a projectively pre-compact family, we may assume $g_k \rightarrow g_0$, then $g_0(0)=0$, $g_0 \not\equiv 0$. Since $M \cap D$ is irreducible, $g_0 \not\equiv 0$ in a neighborhood of $0\in M$. Take a smooth curve $Y$ in $M$ passing through $0$ such that $g_0|_Y$ is not identically zero near $0$. Since $g_0(0)=0$, by residue theorem, $B_1 \cap Y \cap g_k^{-1} (0)$ is non-empty for $k$ large enough, which is a contradiction.
\hfill $\Box$\\

\begin{co}
\label{cy7}
Assume that $M \subset \mathbb{C}^m$ is a homogeneous subvariety that is locally homogeneous. Let $f$ be holomorphic function on $M$ such that $f|_{{\rm Sing} (M)} \equiv 0$. Then for $R>1$ there exists $C>0$ such that for any $p\in M$, $\|f\|_{M \cap B_{Rr_p} (p)} \leq C|f(p)|$, where $r_p = {\rm Dist} (p, M_0)$, $M_0 = f^{-1}(0)$.
\end{co}
\noindent{\bf Proof:} If the corollary is not true, then there exists $p_k \in M$ such that $\|f\|_{ B_{Rr_{p_k}} (p_k)} \geq k|f (p_k)|$. For induction purpose, let $f_k =f$, then $r_{p_k} = {\rm Dist} (p_k, M \cap f_k^{-1} (0))$. Clearly, $\{f_k\}$ is projectively pre-compact on $M$.

By possibly taking subsequence, we may assume that $p_k \rightarrow p_0$ such that $p_0\in M^{(j)}$. Take the local homogenous neighborhood $U \cong (U\cap M^{(j)}) \times U_j$ of $p_0\in M^{(j)}$ with the embedding $(x,y): U \rightarrow (U\cap M^{(j)}) \times \mathbb{C}^m$ with coordinates $x$ on $U\cap M^{(j)}$ and $y$ on $\mathbb{C}^m$ such that $x(p_0)=0$. One may assume $B'_a \times (B_1 \cap U_j) \subset U$. By lemma \ref{cy4}, $\{f_k\}$ is projectively pre-compact on $B'_a \times (B_1 \cap U_j)$. $|y(p_k)| + |x(p_k)|/a \rightarrow 0$, hence $\leq 1/2$ for $k$ large. Let $\tilde{f}_k (x,y) = c_kf_k(x(p_k) + |y(p_k)|x, |y(p_k)|y)$ so that $|\tilde{f}_k (\tilde{p}_k)| =1$, where $(x,y) (\tilde{p}_k) = (0, y(p_k)/|y(p_k)|)$, which implies that ${\rm Dist} (\tilde{p}_k, M^{(j)}) =1$. By lemma \ref{cy5}, $\{\tilde{f}_k\}$ is a projectively pre-compact family on $B'_{a/2} \times (B_1 \cap U_j)$. $f|_{{\rm Sing} (M)} \equiv 0$ implies that $r_{p_k} \leq {\rm Dist} (p_k, M^{(j)}) \rightarrow 0$. Hence for $k$ large, $B_{Rr_{p_k}} (p_k) \subset D:=B'_{a/2} \times B_1$. Without lost of generality, we may assume $\tilde{p}_k \rightarrow \tilde{p}_0 \in M^{(j')}$ for $j'>j$. We now replace $(f_k, p_k,j)$ with $(\tilde{f}_k, \tilde{p}_k,j')$. We still have $\|f_k\|_{ B_{Rr_{p_k}} (p_k)} \geq k|f_k (p_k)|$.

This process can be repeated. Since $j'>j$, the precess has to stop when $j'=n$. Then we have $p_k \rightarrow p_0 \in M^{(n)}$, $|f_k(p_k)|=1$, ${\rm Dist} (p_k, M^{(j)})=1$, $\{f_k\}$ is a projectively pre-compact family on $B'_{a/2} \times (B_1 \cap U_j) = M \cap D$ and $B_{Rr_{p_k}} (p_k) \subset D$ for $k$ large. By lemma \ref{cy5}, $\{\tilde{f}_k(z) = f_k(p_k + r_{p_k} z)\}$ is a projectively pre-compact family on $M \cap B_R$. Apply lemma \ref{cy6}, to the family $\{\tilde{f}_k(z)\}$, we have $\|f_k\|_{ B_{Rr_{p_k}} (p_k)} \leq C_1|f_k (p_k)|$ for certain $C_1>0$, which is a contradiction.
\hfill $\Box$\\

For $\rho \geq 0$, let $M_{\Delta_\rho} = f^{-1} (\Delta_\rho)$, where $\Delta_\rho = \{t\in \mathbb{C}: |t| \leq \rho\}$.\\

\begin{lm}
\label{cy8}
There exists constants $N, C>0$ such that for any $\rho>0$, one can find a locally finite cover $\{B_{2r_i}(p_i)\}_{i\in I}$ of $M_{\Delta_\rho} \setminus M_0$ with the property that for any $p\in M_{\Delta_\rho}$, the number of $i$ such that $p\in M_{\Delta_\rho} \cap B_{2r_i}(p_i)$ is less than $N$. Furthermore, if $M \subset \mathbb{C}^m$ is a homogeneous subvariety that is locally homogeneous, we have $\displaystyle \sup_{B_{2r_i}(p_i)} |f| \leq C\rho$ for all $i$.
 \end{lm}
\noindent{\bf Proof:} For $p\in M_{\Delta_\rho}$, let $r_p := d(p, M_0)$. Find $p_1\in M_{\Delta_\rho}$ such that $r_{p_1} = \max_{p\in M_{\Delta_\rho}} r_p$. By induction, we can find $p_i\in M_{\Delta_\rho}$ such that
\[
r_{p_i} = \max_{p\in M_{\rho,i}} r_p, \quad \mbox{ where } M_{\rho,i} = M_{\Delta_\rho} \left\backslash \bigcup_{j=1}^{i-1} B_{2r_{p_j}} (p_j). \right.
\]

For any $p\in M_{\Delta_\rho}$, let $I_p$ denote the set of $k$ such that $p\in B_{2r_{p_k}} (p_k)$. For any $i,j \in I_p$, assume $j<i$, then $p_i \not\in B_{2r_{p_j}} (p_j)$, $d(p_i,p_j) \geq 2r_j \geq \max(d(p,p_i), d(p,p_j))$. Hence $\angle p_ipp_j \geq \pi/3$. This implies that $|I_p| \leq N(n)$.

If there is $p'\in M_{\Delta_\rho}\setminus M_0$ such that $p' \not\in B_{2r_{p_i}} (p_i)$ for all $i$, then by our construction, $\{p_i\}$ is an infinite set and $r_{p_i} \geq r_{p'} = d(p', M_0)>0$ for all $i$. Notice that $\{B_{r_{p_i}} (p_i)\}$ are disjoint. These last 3 statements form a contradiction. Hence $\{B_{2r_{p_i}} (p_i)\}$ covers $M_t$ for $0< |t| <\rho$.

If $M \subset \mathbb{C}^m$ is a homogeneous subvariety that is locally homogeneous, by corollary \ref{cy7}, there exists $C>0$ such that $\displaystyle \sup_{B_{2r_i}(p_i)} |f| \leq C\rho$ for all $i$.
\hfill $\Box$\\

\begin{theorem}
\label{cy2}
Assume that $M \subset \mathbb{C}^m$ is a homogeneous subvariety that is locally homogeneous, and $M_0$ is irreducible with only canonical singularities, $\psi: M \rightarrow \mathbb{C}$ is holomorphic and is not identically zero on $M_0$. Then for $\epsilon>0$ small enough, there exists $C >0$ such that for any $\rho \geq 0$,
\[
\int_{M_{\Delta_\rho}} \frac{d\mu}{|\psi|^{2\epsilon}} \leq C \rho^2, \mbox{ where } M_{\Delta_\rho} = \pi^{-1} (\Delta_\rho).
\]
\end{theorem}
\noindent{\bf Proof:} By Proposition \ref{2.3}, we only need to prove that there exists $C>0$ such that
\[
\int_{M_{\Delta_\rho}} \frac{d\mu}{|\psi|^{2\epsilon}} \leq C \rho^2 \int_{M_0} \frac{d\mu_0}{|\psi|^{2\epsilon}}.
\]
If it is not true, then there exist $\rho_k$ such that
\[
\int_{M_{\Delta_{\rho_k}}} \frac{d\mu}{|\psi|^{2\epsilon}} \geq k\rho_k^2 \int_{M_0} \frac{d\mu_0}{|\psi|^{2\epsilon}}.
\]
By lemma \ref{cy8}, one may find locally finite cover $\{B_{2r_{k,i}}(p_{k,i})\}_{i\in I_k}$ of $M_{\Delta_{\rho_k}} \setminus M_0$ with $r_{k,i} = {\rm Dist} (p_{k,i}, M_0)$, and constants $N, C>0$ (independent of $k$) such that for any $p\in M_0$, the number of $i$ such that $p\in M_0 \cap B_{2r_{k,i}}(p_{k,i})$ is less than $N$, and $\displaystyle \sup_{B_{2r_{k,i}}(p_{k,i})} |f| \leq C\rho_k$. Then there exists $i_k \in I_k$ (with $r_k := r_{k,i_k}$ and $p_k := p_{k,i_k}$) such that

\[
\int_{M_{\Delta_{\rho_k}} \cap B_{2r_k} (p_k)} \frac{d\mu}{|\psi|^{2\epsilon}} \geq \frac{k \rho_k^2}{N} \int_{M_0 \cap B_{2r_k} (p_k)} \frac{d\mu_0}{|\psi|^{2\epsilon}}.
\]
Normalize $B_{2r_k} (p_k)$ to $B_2(0)$, $(f,\psi,\rho_k)$ is accordingly normalized to $(f_k,\psi_k, \tilde{\rho}_k)$ so that $\displaystyle \sup_{B_2(0)} |f_k| = \sup_{B_2(0)} |\psi_k| =1$, and $\tilde{\rho}_k \geq 1/C$. Let $X^k_t = f_k^{-1} (t)$ and $X^k_{\Delta_\rho} = f_k^{-1} (\Delta_\rho)$. By our construction, $X^k_0 \cap B_1(0) \not= \emptyset$, hence ${\rm Vol} (X^k_0 \cap B_2(0)) \geq C >0$. (For simplicity, we still use $\Omega_t$ and $\Omega$ to denote the corresponding normalized Calabi-Yau forms.) Then
\[
\int_{B_2(0)} \frac{d\mu}{|\psi_k|^{2\epsilon}} \geq \int_{X^k_{\Delta_{\tilde{\rho}_k}} \cap B_2(0)} \frac{d\mu}{|\psi_k|^{2\epsilon}} \geq \frac{k}{NC^2} \int_{X^k_0 \cap B_2(0)} \frac{d\mu_0}{|\psi_k|^{2\epsilon}}
\]
\[
\geq \frac{k}{NC^2} \int_{X^k_0 \cap B_2(0)} \frac{1}{|df_k|^2|\psi_k|^{2\epsilon}}\geq Ck{\rm Vol} (X^k_0 \cap B_2(0)) \geq Ck.
\]
Since $f_k$ and $\psi_k$ are polynomials with bounded degree. By taking subsequence, we may assume $(f_k, \psi_k, \tilde{\rho}_k) \rightarrow (f_0, \psi_0, \tilde{\rho}_0)$. (Notice that $\tilde{\rho}_0 \not=0$.) Then
\[
\int_{B_2(0)} \frac{d\mu}{|\psi_0|^{2\epsilon}} = \lim_{k \rightarrow +\infty}\int_{B_2(0)} \frac{d\mu}{|\psi_k|^{2\epsilon}} = + \infty.
\]
This is a contradiction.
\hfill $\Box$\\

Consider $\mathbb{C}^m$ with the weighted $\mathbb{C}^*$-action $\rho_t(z) = (t^{w_1} z_1, \cdots, t^{w_m} z_m)$ with the weight vector $w = (w_1, \cdots, w_m) \in \mathbb{Z}_+^m$. (For convenience, we would use $\mathbb{C}^m_w$ to denote $\mathbb{C}^m$ with the weighted action, and $\mathbb{C}^m$ to refer to the usual action, where all $w_i=1$.) There is a natural $w$-homogeneous branched covering $\phi_w: \mathbb{C}^m_w \rightarrow \mathbb{C}^m$ of weight $[w]$ (the smallest common multiple of all $w_i$) defined as $\phi_{w,i} (z) = z_i^{[w]/w_i}$. For $M \subset \mathbb{C}^m$ that is $w$-homogeneous, $\phi_w (M) \subset \mathbb{C}^m$ is homogeneous. In the rest of this section, we mainly concern $M \subset \mathbb{C}^m$ that is $w$-homogeneous. (Without lost of generality, we may assume that $\rho_t M \subset M$ for $|t|\leq 1$.) When $M$ is normal, a $w$-homogeneous holomorphic function on smooth part of $M$ can be extended to a $w$-homogeneous holomorphic function on $\mathbb{C}^m$, hence defines a $w$-homogeneous holomorphic function on $M$.\\

\begin{prop}
\label{cy9}
Assume that $M \subset \mathbb{C}^m$ is $w$-homogeneous, $f$ and $\psi$ are $w$-homogeneous holomorphic functions on $M$ of weight $w_f$ and $w_\psi$, and for any $0 \leq r \leq 1$,
\[
\int_{M_{\Delta(r)}} \frac{d\mu}{|\psi|^{2\epsilon}} = \int_{\Delta(r)} dtd\bar{t} \int_{M_t} \frac{d\mu_t}{|\psi|^{2\epsilon}} \leq Cr^2.
\]
Then there exists $C>0$ such that for $|t|\leq 1$,
\[
I_t = \int_{M_t} \frac{d\mu_t}{|\psi|^{2\epsilon}} \leq C.
\]
\end{prop}
\noindent{\bf Proof:} For $|t| \leq 1$, $\rho_{t/r} (M_r) \subset M_t$.
\[
I_t = \int_{M_t} \frac{d\mu_t}{|\psi|^{2\epsilon}} \geq \int_{\rho_{t/r} (M_r)} \frac{d\mu_t}{|\psi|^{2\epsilon}} = (\frac{|t|}{r})^a \int_{M_r} \frac{d\mu_r}{|\psi|^{2\epsilon}} = (\frac{|t|}{r})^a I_r,
\]
where $a = 2(w_0 + \cdots + w_m - w_f - \epsilon w_\psi)$. Then
\[
Cr^2 \geq \int_{\Delta(r)} I_t d\mu_{\mathbb{C}} \geq I_r \int_{\Delta(r)} (\frac{|t|}{r})^a d\mu_{\mathbb{C}} = \frac{2\pi}{2+a}r^2 I_r,\quad (a>-2).
\]
Hence $I_r \leq (2+a)C/2\pi$.
\hfill $\Box$\\

\begin{prop}
\label{cy13}
Assume $M\subset \mathbb{C}^m$ and $f$ a holomorphic function on $\mathbb{C}^m$ that is not identically zero on each connected component of $M$, then there exists $R>0$ and a $w$-homogeneous map $\mathbb{C}^m \rightarrow \mathbb{C}^n$ (which can be made a linear projection in the homogeneous case) such that for any $t$, the induced map $M_t \cap B_R(0) \rightarrow \mathbb{C}^n$ is a branched covering.
\end{prop}
\noindent{\bf Proof:} There is a natural equivariant branched covering $\phi_w: \mathbb{C}^m_w \rightarrow \mathbb{C}^m$. $\phi_w(M)$ is a closed subscheme of $\mathbb{C}^m$. By Noether normalization theorem, there exists a linear projection $\mathbb{C}^m \rightarrow \mathbb{C}^{n+1}$ that induces a branched covering $\phi_w(M) \rightarrow \mathbb{C}^{n+1}$. The composition $\tilde{\mathfrak{p}}: M \rightarrow \mathbb{C}^{n+1}$ is also a branched covering.

Assume $f$ satisfies $p(f) := f^l + a_{l-1} f^{l-1} + \cdots + a_0=0$, then $\tilde{\mathfrak{p}} (M_t)$ is contained in the divisor $D_{p(t)}$ (in particular, $\tilde{\mathfrak{p}} (M_0)$ is contained in the divisor $D_{a_0}$). By Weierstrass preparation theorem, there exists a projection $\mathbb{C}^{n+1} \rightarrow \mathbb{C}^n$ that restricts to branched coverings $\{p(t)=0\} \rightarrow \mathbb{C}^n$. The composition gives the desired $w$-homogeneous map $M \rightarrow \mathbb{C}^n$.
\hfill $\Box$\\

\begin{co}
\label{cy14}
Assume $M\subset \mathbb{C}^m_w$ is a quasi-homogeneous normal variety with weight $w$, $f$ a holomorphic function on $M$, then there exists $R>0$ and a linear projection $M \rightarrow \mathbb{C}^n$ such that for any $t$, the induced map $M_t \cap B_R(0) \rightarrow \mathbb{C}^n$ is a branched covering.
\end{co}
\noindent{\bf Proof:} Under the condition, $f$ can be extended to a holomorphic function on $\mathbb{C}^m$. Then we are in the situation of proposition \ref{cy13}.
\hfill $\Box$\\

\begin{co}
\label{cy15}
Assume $M \subset \mathbb{CP}^N_w \times \Delta$ is a closed subvariety. There exists a smaller disk $\Delta' \subset \Delta$ and a map $\mathbb{CP}^N_w \rightarrow \mathbb{CP}^n$ such that for any $t\in \Delta'$, the induced map $M_t \rightarrow \mathbb{CP}^n$ is a branched covering.
\end{co}
\noindent{\bf Proof:} Consider $\tilde{M} \subset \mathbb{C}^{N+1} \times \Delta \subset \mathbb{C}^{N+2}$ that projectivizes to $M \subset \mathbb{CP}^N_w \times \Delta$. Use the branched covering $\tilde{\phi}_w = \phi_w \times {\rm id}_\Delta: \mathbb{C}^{N+1}_w \times \Delta \rightarrow \mathbb{C}^{N+1} \times \Delta$, $\tilde{\phi}_w (\tilde{M}) \subset \mathbb{C}^{N+1} \times \Delta$ is a closed subscheme that is homogeneous on $\mathbb{C}^{N+1}$-direction. Apply proposition \ref{cy13}, there exists $R>0$ and a linear projection $\tilde{\phi} (\tilde{M}) \rightarrow \mathbb{C}^{n+1}$ such that for any $t$, the induced map $\tilde{\phi} (\tilde{M}_t) \cap B_R(0) \rightarrow \mathbb{C}^{n+1}$ is a branched covering. Notice that $\tilde{M}_t \cap B_R(0) \rightarrow \mathbb{C}^{n+1}$ is $w$-homogeneous, hence can be homogeneously extended to a branch covering $\tilde{M}_t \rightarrow \mathbb{C}^{n+1}$ if $(0,t) \in \tilde{M}_t \cap B_R(0)$. For $t\in \Delta' := \{t\in \Delta: |t| <R\}$, $(0,t) \in \tilde{M}_t \cap B_R(0)$, which implies that $M_t \rightarrow \mathbb{CP}^n_w$ is a branched covering.
\hfill $\Box$\\

The following lemma indicates that being quasi-homogeneous is not as restrictive as it seems.\\

\begin{lm}
\label{cy11}
Consider $(M,0)$ with a $\mathbb{C}^*$-action fixing 0 and $f = hg$, where $h$ is a nowhere zero holomorphic function and $g$ is a $\mathbb{C}^*$-equivariant function with degree $d$. There exists a map $F: M \rightarrow M$ that is biholomorphic near 0, and $f\circ F = g$.
\end{lm}
\noindent{\bf Proof:} $f(z) = h(z)g(z) = g(h^{1/d}(z)z)$. $F^{-1}(z) = h^{1/d}(z)z: M \rightarrow M$ is biholomorphic near $z=0$. Hence $f\circ F = g$.
\hfill $\Box$\\

\setcounter{equation}{0}\section{Gromov-Hausdorff convergence of
Calabi-Yau manifolds }
 In this section, we prove Theorem \ref{t1} and Corollary \ref{t2}.
 First, we prove  the
 following general result.

\begin{theorem}\label{t4.1} Let $(M_{k}, g_{k})$ be a family of Riemannian
 $n$-manifolds, and $(N, d_{N})$ be a compact path metric space.
 Assume that
 \begin{itemize}\label{00}
  \item[(i)]  There are two constants $C>0$ and $\kappa >0$ independent of $k$  such
  that  $${\rm Ric}(g_{k})\geq  -Cg_{k}, \ \ \ \ {\rm and }  \ \ \ \
  {\rm Vol}_{g_{k}}(B_{g_{k}}(p, r))\geq \kappa r^{n}, $$ for any metric
  ball $B_{g_{k}}(p, r)$.
   \item[(ii)]  $$0< \lim_{k\rightarrow \infty}
   {\rm Vol}_{g_{k}}(M_{k})= \mathcal{H}^{n}(N)< \infty,$$ where $
   \mathcal{H}^{n}(N)$ is the $n$-Hausdorff measure of $(N, d_{N})$.
   \item[(iii)]  There is a dense open subset $N_{0}\subset N$ such
   that $\dim_{\mathcal{H}}N\backslash N_{0}\leq n-2$, and $N_{0}$
   is a smooth manifold.  There  is
     a $C^{1,\alpha}$-Riemannian metric $g$ on $N_{0}$ such that,
     for any $x$ and $y\in N_{0}$, there is a
   minimal geodesic $\gamma$ in $N_{0}$ connecting $x$ and $y$ satisfying  $d_{N}(x,y)={\rm length}_{g}(\gamma)$.
    \item[(iv)] There are smooth embeddings $F_{k}: N_{0}\rightarrow M_{k}
    $ such that, for any compact subset $K\subset N_{0}$,
    $F_{k}^{*}g_{k}$ $C^{1,\alpha}$-converges to $g$ on $K$.
 \end{itemize} Then $$\lim_{k\rightarrow \infty}d_{GH}((M_{k}, g_{k}), (N,
 d_{N}))=0,$$ where $d_{GH}$ denotes the Gromov-Hausdorff distance.

\end{theorem}
Note that the assumptions (i) and (ii) imply that the diameters of
$(M_{k}, g_{k})$ are uniformly bounded from above.  By
  Gromov's precompactness   theorem (c.f.  \cite{G1}), a
  subsequence of  $\{(M_{k}, g_{k})\}$ converges to a compact length  metric
  space $(Y, d_{Y})$ in the Gromov-Hausdorff topology.  Since
  $F_{k}$ are diffeomorphisms from $N_{0}$ to their images
  $F_{k}(N_{0})$,  we do not distinguish between  $N_{0}$ and
  $F_{k}(N_{0})$ in this section.

 \begin{lm}\label{t4.2} There exists an
  embedding $f: N_{0}\rightarrow Y$ which is a locally isometry,
  i.e. for any compact subset $K\subset\subset N_{0}$, there is a
  $\delta >0$ such that, for any $p_{1}$,  $p_{2}\in K$ with
  $d_{N}(p_{1},p_{2})< \delta$, we have
  $d_{N}(p_{1},p_{2})=d_{Y}(f(p_{1}),f(p_{2})).$
    \end{lm}

    \noindent{\bf Proof:} For any $i >0$, let $W_{i}=\{x\in N_{0}| d_{N}(x, N\backslash  N_{0})
   \geq \frac{1}{i}\}$. Since,  when $k\rightarrow \infty$, $g_{k}$
  converges to $g$ in the $C^{1,\alpha}$-sense on a fixed
  $W_{i}$, by passing to a subsequence,  we can assume that  \begin{equation}\label{3.3}
  \|g_{k}-g\|_{C^{1}(g)}\leq \frac{1}{k},\end{equation} on $W_{i}$.

Since $\{(M_{k}, g_{k})\}$ converges to  $(Y, d_{Y})$ in the
Gromov-Hausdorff topology, by passing to a subsequence, we
   assume that $d_{GH}((M_{k}, g_{k}), (Y, d_{Y}))< \frac{1}{2k}$.
 There are $\frac{1}{k}$-Hausdorff approximations
  $\psi_{k}: M_{k}\rightarrow Y$ for each $k$, i.e.   $Y\subset \{y| d_{Y}(y, \psi_{k}(M_{k}))<
  \frac{1}{k}\}$ and  \begin{equation}\label{3.4}|d_{M_{k}}(q_{1}, q_{2})-d_{Y}(\psi_{k}(q_{1}), \psi_{k}( q_{2}))|<
  \frac{1}{k},  \end{equation} for any $q_{1}, q_{2} \in M_{k}$, where
  $d_{M_{k}}$ is the distance function induced by $g_{k}$.

  Let $A$ be a countable dense subset of $N$. Then, for any $i$,
  $A\cap W_{i}$ is a countable dense subset of $ W_{i}$. Now, we
  define a map $f_{i}$ from $A\cap W_{i}=\{a_{1}, a_{2}, \cdots\}$ to
  $Y$. For $a_{1}$,  a subsequence $\{\psi_{k_{1}}(a_{1})\}$ of $\{\psi_{k}(a_{1})\}$
  converges to a point  $b_{1}\in Y$ since $Y$ is compact. Let
  $f_{i}(a_{1})=b_{1}$. For $a_{2}$ and $(A\cap W_{i},
  d_{M_{k_{1}}})$, by repeating the above procedure, we obtain
  that a subsequence $\{\psi_{k_{2}}(a_{j})\}$, $j=1,2$,
  converges to  $b_{j}\in Y$, $j=1,2$, respectively. Define $f_{i}(a_{2})=b_{2}$. By
  repeating this procedure and the standard  diagonal argument, we can find a
  subsequence of $(M_{k}, g_{k})$, denoted by $(M_{k}, g_{k})$ also, such
  that $d_{GH}((M_{k}, g_{k}), (Y, d_{Y}))< \frac{1}{2k}$, and $\psi_{k}(a_{j})$
  converges to  $b_{j}\in Y$, i.e. $d_{Y}(\psi_{k}(a_{j}), b_{j})\rightarrow 0$ when $k\rightarrow\infty$. For any
  $a_{j} \in A\cap W_{i}$, define $f_{i}(a_{j})=b_{j}$.

  Now, we prove that  $f_{i}: W_{i-1}\cap A
\rightarrow Y$ is injective.  If it is not true, there are $x$,
$y\in W_{i-1}\cap A$ such that $f_{i}(x)=f_{i}(y)$.  By (\ref{3.4}),
and passing to a subsequence,
$${\rm length}_{g_{k}}(\gamma_{k})=d_{M_{k}}(x, y)< \frac{1}{k}+
d_{Y}(\psi_{k}(x), \psi_{k}( y))< \frac{3}{k},$$ where, for any $k$,
$\gamma_{k}$
 is  the minimal geodesic connecting $x$ and $y$ in $(M_{k}, g_{k})$.  By
(\ref{3.3}),  we have
$$\sqrt{1-\frac{1}{k}}{\rm length}_{g}(\gamma_{k}\cap W_{i-1}) \leq {\rm length}_{g_{k}}(\gamma_{k}\cap
W_{i-1})\leq {\rm length}_{g_{k}}(\gamma_{k})< \frac{3}{k}.$$ If there is
a subsequence of ${k}$ such that $\gamma_{k}\cap (M_{k}\backslash
W_{i-1})$ are not empty, $$\sqrt{1-\frac{1}{k}}(d_{N}(x, \partial
W_{i-1})+d_{N}(y, \partial W_{i-1}))\leq
\sqrt{1-\frac{1}{k}}{\rm length}_{g}(\gamma_{k}\cap W_{i-1}) <
\frac{3}{k}.$$  By taking $k\gg 1$, it is a contradiction.  Thus
$\gamma_{k}\subset W_{i-1}$ for $k\gg 1$, and,
$$\sqrt{1-\frac{1}{k}}d_{N}(x, y)\leq
\sqrt{1-\frac{1}{k}}{\rm length}_{g}(\gamma_{k})< \frac{3}{k},$$  which is
also a contradiction. Hence $f_{i}: W_{i-1}\cap A \rightarrow Y$
is injective.

  Note that   there is a $r_{i}>0$ such that, for any $q\in W_{i}$,
  the metric ball $B_{g}(q, r_{i})$ is a geodesic  convex set
  (\cite{Pe}). By taking $r_{i}<\frac{1}{i(i-1)}$, for any $q_{1}, q_{2} \in W_{i-1}$
  with $d_{N}(q_{1}, q_{2})\leq r_{i}$,
    there is a unique minimal geodesic $\gamma_{s}\subset
    W_{i}$ connecting $q_{1}$ and $ q_{2}$ such that $d_{N}(q_{1},
    q_{2})={\rm length}_{g}(\gamma_{s})$. Thus, by (\ref{3.3}),  we obtain
    that $$d_{M_{k}}(q_{1},
    q_{2})\leq {\rm length}_{g_{k}}(\gamma_{s}) \leq
    \sqrt{1+\frac{1}{k}}{\rm length}_{g}(\gamma_{s})=\sqrt{1+\frac{1}{k}} d_{N}(q_{1},
    q_{2}).$$ By reversing the roles of $g$ and $g_{k}$,  and the same argument as above, we have $$d_{N}(q_{1},
    q_{2})\leq
    \sqrt{1+\frac{1}{k}} d_{M_{k}}(q_{1},
    q_{2}).$$ Note that, for any $a_{1}, a_{2}\in A\cap W_{i-1} $ with $d_{N}(a_{1}, a_{2})\leq r_{i}$,
    $$d_{Y}(b_{1}, b_{2})\leq d_{Y}(b_{1}, \psi_{k}(a_{1}))+d_{Y}(\psi_{k}(a_{1}),
    \psi_{k}(a_{2}))+d_{Y}(\psi_{k}(a_{2}), b_{2}), \ \ \ {\rm
    and}$$ $$d_{Y}(b_{1}, b_{2})\geq d_{Y}(\psi_{k}(a_{1}),
    \psi_{k}(a_{2}))-d_{Y}(b_{1}, \psi_{k}(a_{1}))-d_{Y}(\psi_{k}(a_{2}), b_{2}).
    $$ Thus, by (\ref{3.4}),
     $$d_{Y}(b_{1}, b_{2})\leq d_{Y}(b_{1}, \psi_{k}(a_{1}))+\sqrt{1+\frac{1}{k}} d_{N}(a_{1},
    a_{2})+d_{Y}(\psi_{k}(a_{2}), b_{2})+ \frac{1}{k}, \ \ \ {\rm
    and}$$ $$d_{Y}(b_{1}, b_{2})\geq (1+\frac{1}{k})^{-\frac{1}{2}} d_{N}(a_{1},
    a_{2})-d_{Y}(b_{1}, \psi_{k}(a_{1}))-d_{Y}(\psi_{k}(a_{2}), b_{2})-\frac{1}{k}.
    $$ By letting $k\rightarrow \infty$, we obtain that $$d_{Y}(b_{1},
    b_{2})=d_{N}(a_{1},
    a_{2}).$$ Hence we can extend $f_{i}$ uniquely to a continuous
    map $f_{i}: W_{i-1}\rightarrow Y$, which is injective, and  satisfies that $$ d_{Y}(f_{i}(q_{1}),
    f_{i}(q_{2}))=d_{N}(q_{1},
    q_{2}), $$  for any $q_{1}, q_{2} \in W_{i-1}$
  with $d_{N}(q_{1}, q_{2})\leq r_{i}$.

  By the same arguments as above, we can find a $r_{i+1}>0$, and a continuous
    map $f_{i+1}: W_{i}\rightarrow Y$, which  is injective,  satisfies that $$ d_{Y}(f_{i+1}(q_{1}),
    f_{i+1}(q_{2}))=d_{N}(q_{1},
    q_{2}), $$  for any $q_{1}, q_{2} \in W_{i}$
  with $d_{N}(q_{1}, q_{2})\leq r_{i+1}$. Furthermore, from the
  construction, we can assume that $f_{i+1}|_{W_{i-1}}=f_{i}$. Thus
  we get a family of maps $f_{i+1}: W_{i}\rightarrow Y$.
  Define $f: N_{0}\rightarrow Y$ by $f(q)=f_{i}(q)$ if $q\in W_{i-1}$.  We obtain the conclusion.
  \hfill $\Box$\\

 This lemma implies that
   \begin{equation}\label{3.5} {\rm length}_{ g}(\gamma)={\rm length}_{d_{Y}}(f(\gamma)) ,
   \end{equation} if $\gamma$ is a smooth  curve in $N_{0}$.

     \begin{lm}\label{t4.3} There is a continuous surjective
    map $\tilde{f}: N \rightarrow Y$ such that   $\tilde{f}|_{
    N_{0}}=f$.
      \end{lm}
   \begin{proof}
     Note that $ N_{0}$ is dense in $N$.
       Let $x\in N$, and $\{x_{j}\}\subset N_{0}$ be a
      sequence of points  converging to $x$. For any $x_{j}, x_{j+l}\in  \{x_{j}\}$, there is
      a  minimal geodesic $\gamma_{j,j+l}\subset N_{0}$ connecting $x_{j}$ and $
      x_{j+l}$ with ${\rm length}_{
      g}(\gamma_{j,j+l})=d_{N}(x_{j}, x_{j+l})$ from the assumption.   By (\ref{3.5}),
 $$d_{Y}(f(x_{j}), f(x_{j+l}))\leq
      {\rm length}_{d_{Y}}(f(\gamma_{j,j+l}))= {\rm length}_{
      g}(\gamma_{j,j+l})=d_{N}(x_{j}, x_{j+l}).$$ Hence
      $\{f(x_{j})\}$ is a Canchy sequence, and we denote the limit
      as $y$. If $\{x'_{j}\}\subset N_{0}$ is  another
      sequence of points  converging to $x$, and $\gamma_{j}$ are minimal  geodesics connecting
      $x_{j}$ and  $x'_{j}$ in $N_{0}$,  then $$d_{Y}(f(x_{j}), f(x'_{j}))\leq
      {\rm length}_{d_{Y}}(f(\gamma_{j}))= {\rm length}_{
      g}(\gamma_{j})=
      d_{N}(x_{j}, x'_{j})\rightarrow 0,$$ when $j\rightarrow\infty$. Thus $\{f(x'_{j})\}$ converges
      to $y$ too.   Define $\tilde{f}(x)=y$, and, clearly,
      $\tilde{f}$ is a continuous
    map from $ N $ to $ Y$ from the construction.

    We claim that $\tilde{f}(N)$ is closed in $Y$. Let
    $\{y_{j}\}\subset \tilde{f}(N)$ be a sequence of points
    converging to $y$ in $Y$.  From the construction above, for any
    $j$, there is a sequence of points $\{x_{j,i}\}\subset N_{0}$
    such that $d_{Y}(y_{j}, f(x_{j,i}))\rightarrow 0$ when
    $i\rightarrow\infty$. By the standard diagonal argument,
   we can find a   sequence   of points $\{x_{j,i_{j}}\}\subset
   N_{0}$, and a point $x\in N$ such that $$d_{N}(x_{j,i_{j}}, x)\rightarrow
   0, \ \ \ \ {\rm and} \ \ \ d_{Y}(y, f(x_{j,i_{j}}))\rightarrow 0,$$ when
    $j\rightarrow\infty$. By the construction of $\tilde{f}$,
    $y=  \tilde{f}(x)$, and, thus,  $\tilde{f}(N)$ is closed in $Y$.

 Now, we prove that $\tilde{f}$ is surjective. If $\tilde{f}$ is not surjective, there is a point
   $y\in Y\backslash
    \tilde{f}(N)$, and a $\delta >0$ such that
     the intersection of the metric ball   $B_{d_{Y}}(y, \delta)$ and $\tilde{f}(N)$
    is empty. Let
    $B_{g_{k}}(y_{k}, \delta)$ be    metric $\delta$-balls  of $(M_{k}, g_{k})$ such that
    $B_{g_{k}}(y_{k}, \delta)$ converges to $B_{d_{Y}}(y, \delta)$ under the convergence of
    $(M_{k}, g_{k})$ to $(Y, d_{Y})$.  Now we need the volume
    convergence theorem duel to Colding and Cheeger:

\begin{theorem}[\cite{C} \cite{CC1}]
\label{t4.05}
Let $(M_{k}, g_{k}, y_{k})$ be a family of Riemannian
 $n$-manifolds, which converges to a compact path metric space $(Y, d_{Y},
 y)$.  If  there are two constants $C>0$ and $\kappa >0$ independent of $k$  such
  that  $${\rm Ric}(g_{k})\geq  -Cg_{k}, \quad \mbox{ and } \quad
  {\rm Vol}_{g_{k}}(B_{g_{k}}(p,  \delta))\geq \kappa  \delta^{n}, $$ for any metric
  ball $B_{g_{k}}(p,  \delta)$, then
  \begin{equation}
  \lim\limits_{k\rightarrow\infty} {\rm Vol}_{g_{k}}(M)= \mathcal{H}^{n}(Y) \mbox{ and } \lim\limits_{k\rightarrow\infty}  {\rm Vol}_{g_{k}}(B_{g_{k}}(y_{k}, \delta))= \mathcal{H}^{n}(B_{d_{Y}}(y, \delta)),
  \end{equation}
  where $\mathcal{H}^{n}$ denotes the Hausdorff measure.
\end{theorem}

   By this theorem and  the assumptions,  we obtain that
   \begin{equation}
   \mathcal{H}^{n}(Y)=\mathcal{H}^{n}(N)
   \quad \mbox{ and } \quad \mathcal{H}^{n}(B_{d_{Y}}(y, \delta))\geq \kappa \delta^{n} .  \end{equation}
    Since
  $\dim_{\mathcal{H}}N\backslash N_{0} \leq n-2$,   the $n$-dimensional Hausdorff measure of
  $N\backslash N_{0}$ is zero, i.e.  $\mathcal{H}^{n}(N\backslash N_{0})=0$, and $$\mathcal{H}^{n}(N)={\rm Vol}_{g}(N_{0}).$$   From Lemma \ref{t4.2}, $f$  is a locally isometry,
  i.e. for any compact subset $K\subset\subset N_{0}$, there is a
  $\delta' >0$ such that, for any $p_{1}$,  $p_{2}\in K$ with
  $d_{N}(p_{1},p_{2})< \delta'$, we have
  $d_{N}(p_{1},p_{2})=d_{Y}(f(p_{1}),f(p_{2}))$. Thus,   for any $y\in f(N_{0})$, the tangent
  cone $Y_{y}$ is $\mathbb{R}^{n}$, and the $n$-Hausdorff measure $\mathcal{H}^{n}$ is the Riemannian
  measure induced by $g$ on $f(N_{0})$.
   Hence $${\rm Vol}_{g}(N_{0})=\mathcal{H}^{n}(Y)\geq
   \mathcal{H}^{n}(f(N_{0}))+\mathcal{H}^{n}(B_{d_{Y}}(y,
   \delta))\geq {\rm Vol}_{g}(N_{0})+ \kappa \delta^{n}>{\rm Vol}_{g}(N_{0}).$$ It is
   a contradiction.
\end{proof}

\begin{lm}\label{t4.5} $\tilde{f}: (N, g)  \rightarrow (Y, d_{Y})$  is an isometry ,
  i.e. for any $p_{1}, p_{2}\in N$, $$ d_{N}(p_{1}, p_{2})=d_{Y}(\tilde{f}(p_{1}),
  \tilde{f}(p_{2})).$$
  \end{lm}

  \noindent{\bf Proof:} Note that $\tilde{f}$ is a uniformly  continuous map,
  since $N$ is compact.   For any $p_{1}, p_{2}\in N$, there are
  sequences of points $\{p_{j,i}\}\subset N_{0}$, $j=1,2$,  such that $d_{N}(p_{j,i}, p_{j})\rightarrow 0$ when $i\rightarrow\infty$.
  Thus $d_{Y}(f(p_{j,i}), \tilde{f}( p_{j}))\rightarrow 0$, $j=1,2$,  when
  $i\rightarrow\infty$. From the assumption, there is a minimal geodesic $\gamma_{i}$
  connecting $p_{1,i}$ and $p_{2,i}$ in $N_{0}$, which satisfies that
  ${\rm length}_{g}(\gamma_{i})=d_{N}(p_{1,i}, p_{2,i})$.  By
  (\ref{3.5}),
  $$d_{Y}(f(p_{1,i}), f(p_{2,i}))\leq {\rm length}_{d_{Y}}(f(\gamma_{i}))={\rm length}_{g}(\gamma_{i})=d_{N}(p_{1,i}, p_{2,i}).$$
  Thus \begin{eqnarray*}d_{Y}(\tilde{f}(p_{1}),\tilde{f}(p_{2})) & \leq & d_{Y}(f(p_{1,i}), f(p_{2,i})) +
  d_{Y}(f(p_{1,i}), \tilde{f}( p_{1}))+d_{Y}(f(p_{2,i}), \tilde{f}( p_{2}))\\ & \leq & d_{Y}(f(p_{2,i}), \tilde{f}( p_{2}))+
  d_{Y}(f(p_{1,i}), \tilde{f}( p_{1}))+d_{N}(p_{1,i}, p_{2,i}) \\  & \leq &
  d_{N}(p_{1}, p_{2}) +\sum (d_{Y}(f(p_{j,i}), \tilde{f}( p_{j}))+d_{N}(p_{j,i}, p_{j})). \end{eqnarray*} By letting
  $i\rightarrow\infty$, we obtain that \begin{equation}\label{3.8} d_{Y}(\tilde{f}(p_{1}),\tilde{f}(p_{2}))  \leq
  d_{N}(p_{1}, p_{2}).\end{equation}

  If $S_{N}=N\backslash N_{0}$ and $S_{Y}=Y\backslash f(N_{0})$,
  then $\tilde{f}(S_{N})\supset S_{Y}$ since $\tilde{f}$ is
  surjective. Since
  $\dim_{\mathcal{H}}S_{N} \leq n-2$,   the $n-1$-dimensional Hausdorff measure of
  $S_{N}$ is zero, i.e.  $\mathcal{H}^{n-1}(S_{N})=0$. For any
  $\eta >0$, and  any collection of countable coverings, $\{B_{g}(q_{\nu},
  r_{\nu})\}$, of $S_{N}$ with $r_{\nu}\leq \eta$, by (\ref{3.8}), $ \tilde{f}(B_{g}(q_{\nu},
  r_{\nu}))\subset B_{d_{Y}}(\tilde{f}(q_{\nu}),
  r_{\nu})$, and $\{B_{d_{Y}}(\tilde{f}(q_{\nu}),
  r_{\nu})\}$ is a covering of $S_{Y}$. Thus
  $$\mathcal{H}_{\eta}^{n-1}(S_{Y})\leq
  \varpi_{n-1}\sum\limits_{\nu}r_{\nu}^{n-1},$$ where $\varpi_{n-1}$
  is the volume of 1-ball in Euclidean space $\mathbb{R}^{n-1}$. We have $$\mathcal{H}_{\eta}^{n-1}(S_{Y})\leq
 \inf\limits_{\{B_{g}(q_{\nu},
  r_{\nu})\}}
  \varpi_{n-1}\sum\limits_{\nu}r_{\nu}^{n-1}=\mathcal{H}_{\eta}^{n-1}(S_{N}),$$
  and $$\mathcal{H}^{n-1}(S_{Y})=\lim\limits_{\eta \rightarrow 0}\mathcal{H}_{\eta}^{n-1}(S_{Y})\leq
 \lim\limits_{\eta \rightarrow 0}\mathcal{H}_{\eta}^{n-1}(S_{N})=\mathcal{H}^{n-1}(S_{N})=0.
 $$ Hence the $n-1$-dimensional Hausdorff measure of
  $S_{Y}$ is zero, i.e.  $\mathcal{H}^{n-1}(S_{Y})=0$. We need the
  following theorem:

  \begin{theorem}[Theorem 3.7 in \cite{CC2}] \label{t4.06} Let  $(M_{k}, g_{k}, y_{k})$ and  $(Y, d_{Y},
 y)$ be the same as in Theorem \ref{t4.05}, and $B$ be a closed
 subset of $Y$ with $\mathcal{H}^{n-1}(B)=0 $. If $x_{1}\in Y\backslash B$,
 then,
 for $\mathcal{H}^{n}$-almost all $ x_{2}\in Y\backslash B$, there
 is a minimal geodesic connecting $x_{1}$ and $x_{2}$ which lies in
 $Y\backslash B$.
  \end{theorem}
  This theorem implies that, for any $x_{1}, x_{2}\in Y\backslash
  S_{Y}$, any $\varepsilon >0$, there is an  $x_{2}'\in Y\backslash
  S_{Y}$  such that there is a minimal geodesic connecting $x_{2}$ and $x_{2}'$ in $Y\backslash
  S_{Y}$, and   $d_{Y}(x_{2},x_{2}') <  \varepsilon$.
  Hence we can find a  curve $\bar{\gamma}$ connecting $x_{1}$ and $
  x_{2}$ in  $Y\backslash
  S_{Y}$ such that $$ {\rm length}_{d_{Y}}(\bar{\gamma})\leq d_{Y}(x_{2},x_{2}')+d_{Y}(x_{1},x_{2})\leq
   \varepsilon
  +d_{Y}(x_{1},x_{2}).$$

   If there is an  $i$ such that  $d_{Y}(f(p_{1,i}), f(p_{2,i}))<
  {\rm length}_{d_{Y}}(f(\gamma_{i}))$,   there
  is a curve $\bar{\gamma}$ connecting $f(p_{1,i}), f(p_{2,i})$ such
  that
    $\bar{\gamma}\subset f(N_{0})$, and $$
    {\rm length}_{d_{Y}}(\bar{\gamma})\leq d_{Y}(f(p_{1,i}), f(p_{2,i}))+
    \frac{1}{2}\varrho <{\rm length}_{d_{Y}}(f(\gamma_{i})),$$ where $\varrho= {\rm length}_{d_{Y}}(f(\gamma_{i}))-d_{Y}(f(p_{1,i}),
    f(p_{2,i}))$.  It contradicts to
     that $f(\gamma_{i})$ is the minimal geodesic in
     $(f(N_{0}), d_{Y})$. Thus, for any $i$,  $$d_{Y}(f(p_{1,i}), f(p_{2,i}))=
  {\rm length}_{d_{Y}}(f(\gamma_{i}))=d_{N}(p_{1,i}, p_{2,i}),$$ and
   \begin{eqnarray*}d_{Y}(\tilde{f}(p_{1}),\tilde{f}(p_{2})) & \geq &
   d_{Y}(f(p_{1,i}), f(p_{2,i}))-d_{Y}(f(p_{2,i}), \tilde{f}( p_{2})) -
  d_{Y}(f(p_{1,i}), \tilde{f}( p_{1}))\\ & \geq  & d_{N}(p_{1,i}, p_{2,i}) -d_{Y}(f(p_{2,i}), \tilde{f}( p_{2}))
  -d_{Y}(f(p_{1,i}), \tilde{f}( p_{1}))\\  & \geq &
  d_{N}(p_{1}, p_{2}) -\sum (d_{Y}(f(p_{j,i}), \tilde{f}( p_{j}))+d_{N}(p_{j,i}, p_{j})). \end{eqnarray*} By letting
  $i\rightarrow\infty$, we obtain that $$d_{Y}(\tilde{f}(p_{1}),\tilde{f}(p_{2}))
  \geq
  d_{N}(p_{1}, p_{2}),$$ and, thus,  $$d_{Y}(\tilde{f}(p_{1}),\tilde{f}(p_{2}))
  =
  d_{N}(p_{1}, p_{2}).$$ We obtain that $\tilde{f}$ is injective,
  and is an  isometry.
  \hfill $\Box$\\

\noindent{\bf Proof of Theorem \ref{t4.1}:}  If it is not true, there is a  subsequence  of  $(M_{k}, g_{k})$, denoted by  $(M_{k}, g_{k})$
     also,  such that  $d_{GH}((M_{k},g_{k}), (N, d_{N}))>C, $ for
 a constant $C>0$.  By
  Gromov's precompactness   theorem (c.f.  \cite{G1}), a
  subsequence of  $\{(M_{k}, g_{k})\}$ converges to a compact length  metric
  space $(Y, d_{Y})$ in the Gromov-Hausdorff topology, which satisfies $d_{GH}((Y, d_{Y}), (N, d_{N}))>C $.  It  contradicts to  Lemma
  \ref{t4.5}.
  \hfill $\Box$\\

  Now we prove Theorem \ref{t1}.\\

 \noindent{\bf Proof of Theorem \ref{t1}:} Let $N$ be a  Calabi-Yau $n$-variety, which admits
 a crepant  resolution $(M, \pi)$,  $\alpha \in H^{1}(N,
 \mathcal{PH}_{N})$ be a class  represented by a smooth K\"{a}hler form on $N$,   and $g$ be  the unique  singular Ricci-flat
 K\"{a}hler metric with K\"{a}hler form $\omega\in \alpha$.  Assume that the path metric structure of  $(N\backslash S, g)$  extends    to a path metric structure $d_{N}$ on
 $N$ such that the Hausdorff dimension of $S$
 satisfies $\dim_{\mathcal{H}}S\leq 2n-4$, and
 $N\backslash S$ is geodesic convex in $(N, d_{N})$, where $S$ is   the singular set  of $N$,  i.e.  for any $x, y\in
 N\backslash S$, there is a minimal geodesic $\gamma \subset N\backslash S$ connecting $x$ and $y$
  satisfying ${\rm length}_{g}(\gamma)=d_{N}(x,y)$.  Let  $g_{k}$ be  a family
  of Ricci-flat
 K\"{a}hler metrics on $M$ with K\"{a}hler forms $\omega_{k}$ such
 that $[\omega_{k}]\rightarrow \pi^{*}\alpha$ in $H^{1,1}(M,
 \mathbb{R})$ when $k\rightarrow \infty$.

   Note that
  \begin{equation}\label{4.2} \lim\limits_{k\rightarrow \infty}{\rm Vol}_{g_{k}}(M)={\rm Vol}_{g}(N\backslash S),
  \end{equation} $${\rm and} \ \ \   \lim\limits_{k\rightarrow \infty}\int_{M}\omega_{k}\wedge
  \omega_{1}^{n-1}=\langle  \pi^{*}[\omega]  \wedge
  [\omega_{1}]^{n-1}, [M]\rangle . $$
     By Theorem  \ref{t3.1}   and
  Bishop-Gromov comparison theorem, we obtain that  $${\rm diam}_{g_{k}}(M)\leq C_{1}, \ \ \
   \ $$  and,  for any metric ball
  $B_{g_{k}}(r)$, \begin{equation}\label{4.3}{\rm Vol}_{g_{k}}(B_{g_{k}}(r))\geq
  \frac{{\rm Vol}_{g_{k}}(M)}{{\rm diam}_{g_{k}}^{2n}(M)}r^{2n}\geq  C_{2}r^{2n},  \end{equation}
  where $C_{1}$ and $C_{2}$ are two constants independent of $k$.
 Since $\dim_{\mathcal{H}}S\leq 2n-4$, ${\rm Vol}_{g}(N\backslash
 S)=\mathcal{H}^{2n}(N)$.
      By
   Theorem  \ref{202} (Theorem 1.1 in \cite{To}),   $\{g_{k}\}$ converges to $ \pi^{*}g$ on any
  compact subset $K\subset\subset \pi^{-1}(N\backslash S)$ in the $C^{\infty}$
  sense.  Thus the
  conclusion is a directly consequence of Theorem  \ref{t4.1}.
  \hfill $\Box$\\

 \noindent{\bf Proof of Corollary  \ref{t2}:} Let $N$ be   a compact Calabi-Yau
  $n$-orbifold  with $H^{2}(N,
\mathcal{O}_{N})=\{0\}$,  $g$ be  a Ricci-flat  K\"{a}hler metric on
$N$,
   $\omega$ be the K\"{a}hler form of $g$. Assume that $N$ admits  a  crepant resolution $(M, \pi)$.
   By Lemma \ref{203}, there is a
     smooth $(1,1)$-form $\omega_{0}$ in the sense of orbifold forms,
     which  is a smooth    K\"{a}hler form in the sense of Section 5.2 in
    \cite{EGZ}.  By the uniqueness part of Theorem 7.5 of
    \cite{EGZ},  $g$ is the unique Ricci-flat  K\"{a}hler metric on
$N$ with K\"{a}hler  form $\omega=\omega_{0}+ \sqrt{-1}\partial
\overline{\partial} \varphi_{0}$ for a continuous function
$\varphi_{0}$  on $N$. Note that $(N, g)$ is a compact metric space,
  the smooth part  $N_{0}$ of $N$  is geodesic convex in $N$  (c.f.  \cite{Bo}), and
$\dim_{\mathbb{R}}N\backslash N_{0}\leq 2n-4$ since $N\backslash
N_{0}$ is a subvariety of $N$. Hence we obtain  Corollary \ref{t2}
from Theorem \ref{t1}.
\hfill $\Box$\\

\setcounter{equation}{0}\section{Convergence of Calabi-Yau manifolds under smoothing}
Let $M_{0}$ be a projective Calabi-Yau  $n$-variety,  and $S$ be the set of  singular points of $M_{0}$. Assume that $M_{0}$ admits a smoothing  $\pi: \mathcal{M}\rightarrow \Delta$ in $\mathbb{CP}^{N}$ over the unit disc $\Delta=\{t\in \mathbb{C}|
|t|<1 \}$. (See section 1 for precise definition.) Recall that we assumed further that the
 canonical bundle $\mathcal{K}_{\mathcal{M}} \cong \mathcal{O}_{\mathcal{M}}$. Let $\Omega$ denote the corresponding trivializing section of $\mathcal{K}_{\mathcal{M}}$. By the adjunction formula (c.f. \cite{GH}), we have $\mathcal{K}_{M_{t}} =\mathcal{K}_{\mathcal{M}} \otimes [M_{t}]|_{M_{t}} \cong \mathcal{O}_{M_{t}}$. The corresponding trivializing section can be expressed locally as $\Omega_t = (\imath_{\frac{\partial}{\partial t}} \Omega)|_{M_t}$. In the following, by a local embedding $i: (\mathcal{M},x_0) \hookrightarrow (\mathbb{C}^{n'},0)$, we means an isomorphism of an open neighborhood of $x_0$ in $\mathcal{M}$ with a closed analytic subvariety in $B'_R := B_R(0, \mathbb{C}^{n'})$ for sufficiently large $R>0$ that maps $x_0$ to $0$.\\

\begin{lm}
\label{5.1}
For any $x_0\in M_0$, there are $m, C_1>0$ and a local embedding $i: (\mathcal{M},x_0) \hookrightarrow (\mathbb{C}^{n'},0)$ such that:
\begin{itemize}
\item[(i)] For $U':= \mathcal{M} \cap i^{-1} B'_1$ and $U := \mathcal{M} \cap i^{-1} B'_2$, there is $v \in C^\infty (U)$ so that $\omega = \sqrt{-1} \partial \bar{\partial} v$ and $\displaystyle\inf_{\partial U} v \geq C_1 + \sup_{U'} v$.
\item[(ii)] There is a holomorphic map $\mathfrak{p}: U \rightarrow B_1(0) \subset \mathbb{C}^n$ that restricts to a finite branched covering $\mathfrak{p}: M_t \cap U \rightarrow B_1(0)$ of degree $\leq m$ for all $t\in \Delta$. (In particular, when $x_0 \not\in S$, $\mathfrak{p}|_{M_t \cap U}$ is an open embedding, such that $(\mathfrak{p}^* \Omega_{\mathbb{C}^n})|_{M_t} = c\Omega_t$ for a constant $c>0$ independent of $t\in \Delta$.)
\end{itemize}
\end{lm}
\begin{proof}
(i) is obvious when $\mathcal{M}$ is smooth. When $\mathcal{M}$ is not smooth, there is a local embedding $\mathcal{M} \subset \mathbb{C}^N$ such that $\omega = \tilde{\omega}|_{\mathcal{M}}$ for a smooth \k~form $\tilde{\omega}$ on $\mathbb{C}^N$. Then (i) is a consequence of the smooth case.

(ii) is a consequence of the local result Corollary \ref{cy14}, or the global result Corollary \ref{cy15} that restricts to $U$.
\end{proof}

Let  $g$ be a smooth  K\"ahler metric  with K\"ahler form $\omega$ on $\mathcal{M}$, $g_{t} = g|_{M_{t}}$, $\omega_{t}=\omega|_{M_{t}}$ for any $t$, and $\int_{M_{t}}\omega_{t}^{n} \equiv V$ for a constant $V$. By re-normalizing   $\omega$, we assume $V=1$ for convenience. By Yau's theorem on  the Calabi conjecture (\cite{Ya1}), for any $t\neq 0$, there is a unique $\varphi_{t} \in C^{\infty}(M_{t})$ such that

\begin{equation}
\label{4.1}
(\omega_{t}+ \sqrt{-1}\partial \overline{\partial}\varphi_{t})^{n} =\frac{(-1)^{\frac{n^2}{2}}}{\mathcal{V}_{t}}\Omega_{t}\wedge \overline{\Omega}_{t}, \mbox{ and } \sup_{M_{t}}\varphi_{t}=0.
\end{equation}

\begin{prop}
\label{5.7}
There are constant $m,\bar{c}>0$ and a finite collection $\{x_\alpha \in U'_\alpha \subset\subset U_\alpha, v_\alpha \in {\rm PSH}(U_{\alpha})\}$ with $\{U'_\alpha\}$ covering $M_0$ such that for each $\alpha$, $x_\alpha \in M_0$,  $\omega =\sqrt{-1} \partial\overline{\partial} v_{\alpha}$ on $U_{\alpha}$, $\displaystyle\inf_{\partial U_{\alpha}} v_{\alpha} \geq \bar{c} + \sup_{U'_{\alpha}} v_{\alpha}$, and there is a holomorphic map $\mathfrak{p}_\alpha: U_\alpha \rightarrow B_1(0) \subset \mathbb{C}^n$ that restricts to a finite branched covering $\mathfrak{p}_\alpha: M_t \cap U_\alpha \rightarrow B_1(0)$ of degree $\leq m$ for all $t\in \Delta$. (In particular, when $x_\alpha \not\in S$, $\mathfrak{p}|_{M_t \cap U_\alpha}$ is an open embedding, such that $(\mathfrak{p}_\alpha^* \Omega_{\mathbb{C}^n})|_{M_t} = C_\alpha\Omega_t$ for a constant $C_\alpha>0$ independent of $t\in \Delta$.)

For any $c_1, C_1>0$, let $\Lambda = \Lambda_{c_1,C_1}$ be the set of $t\in \Delta$ such that $M_t$ is covered by $\{U_\alpha\}$ and for each $\alpha$ with $x_\alpha \in S$,
\begin{equation}
\label{5.6}
\int_{U_\alpha \cap M_t} |f_\alpha|^{-2c_1} (-1)^{\frac{n^2}{2}}\Omega_t\wedge
\overline{\Omega}_t \leq C_1, \mbox{ where } f_\alpha \Omega_t = \mathfrak{p}_\alpha^* \Omega_{\mathbb{C}^n}.
\end{equation}
Then $\Lambda$ is closed and there exists $C_2>0$ such that for any $t\in \Lambda$, $\displaystyle \inf_{M_{t}}\varphi_{t} \geq -C_2$.
\end{prop}
\noindent{\bf Proof:} The first part of the proposition is a direct consequence of Lemma \ref{5.1} using the fact that $M_0$ is compact. Lemma \ref{3.1} implies that $\Lambda$ is closed.

If $\varphi_t$ is not uniformly bounded below for $t\in \Lambda$,  there is a sequence $t_{k} (\in \Lambda) \rightarrow 0$, and a sequence of points $x_{k} \in M_{t_{k}}$, such that $M_{t_k}$ satisfies the assumption (\ref{5.6}) and

\begin{equation}\label{4.6}\varphi_{k}(x_{k})=\inf_{M_{t_{k}}}\varphi_{k}\rightarrow
-\infty,\end{equation}
where $\varphi_{k}=\varphi_{t_{k}}$. By passing to a subsequence, we may assume that $x_{k}\rightarrow p_{\alpha} \in M_0 \cap U'_{\alpha}$. From now on, our discussions only involve this fixed $\alpha$.

By the first part of the proposition, there is a $v_{\alpha}\in {\rm PSH}(U_{\alpha})$ such that $\omega =\sqrt{-1} \partial\overline{\partial} v_{\alpha}$ on $U_{\alpha}$,
\[
\inf_{\partial U_\alpha} v_{\alpha} =0 \mbox{ and } v_{\alpha}(p_{\alpha})\leq - \bar{c}.
\]
Let $V_k = U_\alpha\cap M_{t_k}$. Then, by (\ref{4.6}), for $t_{k}\ll 1$,
\[
v_{\alpha}(x_{k})+\varphi_{k}(x_{k})\leq \inf_{\partial U_\alpha\cap M_{t_k}}(v_{\alpha}+\varphi_{k})
-\frac{2\bar{c}}{3}.
\]
Let $D=\frac{\bar{c}}{3}-2\epsilon$ and $Q_{k} = v_{\alpha}(x_{k}) + \varphi_{k}(x_{k}) +\epsilon$ with $\epsilon \ll \bar{c}$. $U(q)=\{y\in V_{k}|v_{\alpha}(y)+\varphi_{k}(y)<
q \} \subset U''_\alpha = \{y\in U_\alpha|v_{\alpha}(y) \leq -\bar{c}/3\} \subset\subset U_\alpha$ for any $q\in [Q_{k},Q_{k}+D]$. In particular, $U(q)$ is not empty and relatively compact in $V_{k}$. If $0< \rho<Q_{k}+D-q$, and $w\in {\rm PSH}(V_{k})$ with $-1\leq w
<0$, then $U(q)\subset \tilde{U}=\{\frac{v_{\alpha}+\varphi_{k}-q-\rho}{\rho}< w\}\cap
V_k \subset U(q+\rho)$. By
Theorem \ref{t2.3.1},
\begin{eqnarray*}\int_{U(q)}(-1)^{\frac{n}{2}}(\partial\overline{\partial}w)^{n}&\leq &  \int_{\tilde{U}}(-1)^{\frac{n}{2}}(\partial\overline{\partial}
w)^{n}\\ & \leq &
\rho^{-n}\int_{\tilde{U}}(-1)^{\frac{n}{2}}(\partial\overline{\partial}(v_{\alpha}
+\varphi_{k}))^{n}\\ & \leq &
\rho^{-n}\int_{U(q+\rho)}(-1)^{\frac{n}{2}}(\partial\overline{\partial}(v_{\alpha}
+\varphi_{k}))^{n},
\end{eqnarray*}
thus, for any $0< \rho<Q_{k}+D-q$, we obtain
\[
{\rm Cap_{BT}}(U(q), V_{k})\leq \frac{1}{\rho^{n}} \int_{U(q+\rho)}(-1)^{\frac{n}{2}}(\partial\overline{\partial}(v_{\alpha}
+\varphi_{k}))^{n} = \frac{1}{\rho^{n}\mathcal{V}_t} \int_{U(q+\rho)}d\mu_t.
\]
(Notice that by our construction, the assumption (\ref{5.6}) can be easily satisfied if $x_\alpha \not\in S$.) Under the assumption (\ref{5.6}), Lemma \ref{5.2} implies that
\[
{\rm Cap_{BT}}(U(q), V_{k}) \leq \frac{C}{\rho^{n}} \int_{U(q+\rho)}d\mu_t \leq \frac{C}{\rho^{n}} \frac{{\rm Cap_{BT}}(U(q+\rho), V_{k})}{h({\rm Cap_{BT}}(U(q+\rho), V_{k})^{-\frac{1}{n}})}.
\]
Lemma \ref{2.1} applies to $a(q) := {\rm Cap_{BT}}(U(q), V_{k})$ implies that

\begin{equation}
\label{5.5}
{\rm Cap_{BT}}(U(Q_k +D), V_{k}) \geq C>0.
\end{equation}
Since $U''_\alpha \subset\subset U_\alpha$, there exists $\chi \in C^{\infty}(\mathcal{M})$ such that $-1\leq \chi \leq 0$, $\chi=0$ outside of $U_\alpha \subset \mathcal{M}$ and $\chi=-1$ on $U''_\alpha$. Clearly, for $C_3>0$ large enough, $\chi \in {\rm PSH}_{C_1\omega} (\mathcal{M})$. Apply lemma \ref{5.3}, we have

\[
{\rm Cap_{BT}}(U(Q_k +D), V_{k}) \leq C_3^n {\rm Cap}_{\omega_{t_k}} (U(Q_k +D))
\]
Let $C_4 = -\inf_{U_\alpha} (v_\alpha)$. Then $U(Q_{k}+D)= \{x \in V_k| \varphi_{k}(x)+v_{\alpha}(x)\leq Q_{k}+D\} \subset \{x \in M_{t_k}|
\varphi_{k}(x)\leq Q_{k}+D+C_4\}=:\widetilde{U} $, by  Proposition \ref{t2.3.3},
\[
{\rm Cap_{BT}}(U(Q_k +D), V_{k}) \leq C_3^n {\rm Cap}_{\omega_{t_k}} (\tilde{U})
\]
\[
\leq \frac{C_3^n}{|Q_{k}+D+C_4|}\left(-\int_{M_{t_{k}}}\varphi_{k}\omega_{t_{k}}^{n}+
n V\right)< \frac{C}{|Q_{k}+D+C_4|}
\]
This estimate together with (\ref{5.5}) implies that $\varphi_k(x_k) >C$. This contradicts (\ref{4.6}), and finishes the proof of proposition \ref{5.7}.
\hfill $\Box$\\

\begin{lm}
\label{t5.4}
Under the same situation as in Proposition \ref{5.7}, let
$\tilde{\omega}_{t}=\omega_{t}+\sqrt{-1}\partial\overline{\partial}\varphi_{t}$.
For any compact subset $K\subset \mathcal{M}\backslash S$, there
exists a constant $C_{K}>0$ independent of $t \in \Lambda$ such that
$$C\omega_{t}\leq
\tilde{\omega}_{t}\leq C_{K}\omega_{t},$$ on $K\cap M_{t} $, where
$C>0$ is a constant independent of $t \in \Lambda$ and $K$.
\end{lm}
\noindent{\bf Proof:} Let $\psi_{t}: (M_{t}, \tilde{\omega}_{t}) \rightarrow (\mathcal{M},
\omega)$ be the inclusion maps, which are holomorphic.  Then   Yau's Schwarz lemma says
$$\Delta_{\tilde{\omega}_{t}}\log |\partial \psi_{t}|^{2}\geq \frac{{\rm Ric}_{\tilde{\omega}_{t}}(\partial \psi_{t},
\overline{ \partial \psi_{t}})}{|\partial
\psi_{t}|^{2}}-\frac{R_{\omega}(\partial \psi_{t}, \overline{
\partial \psi_{t}},
\partial \psi_{t},
\overline{ \partial \psi_{t}})}{|\partial \psi_{t}|^{2}},
$$  where $R_{\omega}$ is the holomorphic bi-sectional
curvature of $\omega$ (c.f. \cite{BK} or \cite{Y3}). Note that there
is a  finite covering $ \{U_{\alpha}\}$ of $ \mathcal{M}$ such that,
for any $\alpha$, there is an embedding $i_{\alpha}:
U_{\alpha}\hookrightarrow \mathbb{C}^{m_{\alpha}}$, and a smooth
strongly pluri-subharmonic function $u_{\alpha}$ on $i_{\alpha}(
U_{\alpha}) \subset \mathbb{C}^{m_{\alpha}}$ satisfying that
$\omega|_{U_{\alpha}}=\sqrt{-1}\partial\overline{\partial}u_{\alpha}\circ
i_{\alpha} $.  Thus there is a uniform upper bound for the
holomorphic bi-sectional curvature of $\omega$ on
$\mathcal{M}\backslash S$.

 Since
$|\partial
\psi_{t}|^{2}=tr_{\tilde{\omega}_{t}}\omega_{t}=n-\Delta_{\tilde{\omega}_{t}}\varphi_{t}$
and ${\rm Ric}_{\tilde{\omega}_{t}}\equiv 0$, we have
$$\Delta_{\tilde{\omega}_{t}}\log tr_{\tilde{\omega}_{t}}\omega_{t}\geq
-\overline{R}tr_{\tilde{\omega}_{t}}\omega_{t},
$$
where $\overline{R}=\max \{\sup_{\mathcal{M}\backslash S}R_{\omega},
1\}$. Then
 $$\Delta_{\tilde{\omega}_{t}}(\log tr_{\tilde{\omega}_{t}}\omega_{t}-2\overline{R}\varphi_{t})
\geq -2n\overline{R}+\overline{R}tr_{\tilde{\omega}_{t}}\omega_{t}.
$$ By the  maximum principle, there is a point $x\in M_{t}$ such that
$tr_{\tilde{\omega}_{t}}\omega_{t}(x)\leq 2n$, and
$$\log tr_{\tilde{\omega}_{t}}\omega_{t}-2\overline{R}\varphi_{t}\leq (\log tr_{\tilde{\omega}_{t}}\omega_{t}-2\overline{R}\varphi_{t})(x)\leq \log 2n-2\overline{R}\varphi_{t}(x).
$$
Hence
$$tr_{\tilde{\omega}_{t}}\omega_{t}\leq 2ne^{2\overline{R}(\varphi_{t}-\varphi_{t}(x))}\leq C, \ \ \
{\rm and} \ \ \  \omega_{t}\leq C \tilde{\omega}_{t},  $$
 for a constant $C>0$ independent of $t$ by Proposition
  \ref{5.7}. Note that, for any compact subset $K\subset \mathcal{M}\backslash
S$, there exists a constant $C'_{K}>0$ independent of $t$ such that
$$
\tilde{\omega}_{t}^{n}=\frac{(-1)^{\frac{n^2}{2}}}{\mathcal{V}_{t}}\Omega_{t}\wedge
\overline{\Omega}_{t} \leq C'_{K}\omega_{t}^{n},$$ on $K\cap M_{t}
$. We obtain that $C\omega_{t}\leq \tilde{\omega}_{t}\leq
C_{K}\omega_{t}.$
\hfill $\Box$\\

In \cite{EGZ}, it is proved that there is a unique
continues function $\hat{\varphi}_{0}$ on $M_{0}$, which is smooth
on $M_{0}\backslash S$, satisfying that
\begin{equation}\label{6.22}
 (\omega_{0}+
\sqrt{-1}\partial \overline{\partial}\hat{\varphi}_{0}
)^{n}=\frac{(-1)^{\frac{n^2}{2}}}{\mathcal{V}_{0}}\Omega_{0}\wedge
\overline{\Omega}_{0},  \ \ \ \ \sup \hat{\varphi}_{0}=0,
\end{equation} in the distribution sense on $M_{0}$, and as smooth forms on $M_{0}\setminus S$, i.e.
$\tilde{\omega}_{0}=\omega_{0}+\sqrt{-1}\partial
\overline{\partial}\hat{\varphi}_{0}$ is the unique  singular
Ricci-flat K\"ahler form  (See Section 2  for details).

Recall the smooth embedding $F:M_{0}\backslash S\times \Delta \rightarrow  \mathcal{M}$ constructed in the introduction. Let $F_t :=F|_{M_{0}\backslash S\times \{t\}}: M_{0}\backslash S \rightarrow M_{t}$. For any compact subset $K \subset M_0 \setminus S$, $F_{t}^{*}\omega_{t}$ $C^{\infty}$-converges to $\omega_{0}$, and $dF_{t}^{-1}J_{t}dF_{t}$ $C^{\infty}$-converges to $J_{0}$ on $K$ when $t\rightarrow 0$, where $J_{t}$ (resp.  $J_{0}$) is the complex structure on $M_{t}$ (resp.  $M_{0}$).

\begin{theorem}
\label{t5.5}
Under the same situation as in Proposition \ref{5.7}, on any compact subset $K \subset M_0 \setminus S$, $F_{t}^{*}\varphi_{t}$ converges to $\hat{\varphi}_{0}$ smoothly, when
$t (\in \Lambda)\rightarrow 0$. Furthermore, the diameters of $(M_t, \tilde{g}_t)$ have a  uniformly upper bound, i.e.
\begin{equation}
\label{6.29}
{\rm diam}_{\tilde{g}_t}(M_t)\leq \bar{C},
\end{equation}
for a constant $\bar{C}>0$  independent of $t\in \Lambda$.
\end{theorem}
\noindent{\bf Proof:} By Proposition \ref{5.7} and Lemma  \ref{t5.4}, for any compact subset $K\subset \mathcal{M}\backslash S$, there exist  constants $C>0$, $C_{K}>0$
independent of $t$ such that $\|\varphi_{t}\|_{C^{0}(M_{t})}\leq C$, and $C^{-1}\omega_{t}\leq \omega_{t}+\sqrt{-1}\partial \overline{\partial}\varphi_{t} \leq C_{K}\omega_{t} $.  By Theorem 17.14 in \cite{GT}, we have $\|\varphi_{t}\|_{C^{2,
\alpha}(M_{t}\cap K)}\leq C_{K}''$ for a constant $C_{K}''>0$, and, furthermore, for any $l>0$, $\|\varphi_{t}\|_{C^{l, \alpha}(M_{t}\cap K)}\leq C_{K,l}$ for constants $C_{K,l}>0$
independent of $t$ by the standard bootstrapping argument.   Thus, by passing to a subsequence, $F_{K_{i},k}^{*}\varphi_{t_{k}}$ $C^{\infty}$-converges to a smooth function  $\varphi_{0}$ on $K_{i}$ with $\|\varphi_{0}\|_{L^{\infty}}<C$. By the standard diagram
argument, we can extend  $\varphi_{0}$  to a smooth function on $M_{0}\backslash  S$, denoted by  $\varphi_{0}$ too, which satisfies the equation

\[
(\omega_{0}+ \sqrt{-1}\partial \overline{\partial}\varphi_{0})^{n}=\frac{(-1)^{\frac{n^2}{2}}}{\mathcal{V}_{0}}\Omega_{0}\wedge
\overline{\Omega}_{0}
\]
and $\|\varphi_{0}\|_{L^{\infty}}<C$, where $\mathcal{V}_{0}=\int_{M_{0}\backslash  S}
(-1)^{\frac{n^2}{2}}\Omega_{0}\wedge \overline{\Omega}_{0} $.
Hence $\tilde{\omega}_{0}=\omega_{0}+ \sqrt{-1}\partial
\overline{\partial}\varphi_{0}$ is a Ricci-flat K\"{a}hler form on
 $M_{0}\backslash S$.

Let  $\bar{\pi}:\bar{M}_{0}\rightarrow M_{0}$ be a  resolution
of $M_{0}$,  which exists by \cite{Hir}.     Note that
$\bar{\pi}^{*}\omega_{0}$ is a
 semi-positive $(1,1)$-form on
 $\bar{M}_{0}$, and $\bar{\pi}^{*}\varphi_{0}$ is a bounded
$\bar{\pi}^{*}\omega_{0}$-pluri-subharmonic function on
$\bar{M}_{0}\backslash \bar{\pi}^{-1}(S)$. We claim that
$\bar{\pi}^{*}\varphi_{0}$ can be extended to a bounded
$\bar{\pi}^{*}\omega_{0}$-pluri-subharmonic function $
\bar{\varphi}_{0}$ on $\bar{M}_{0} $. Let $\{U_{\gamma}\}$ be a
family of coordinate charts on $\bar{M}_{0} $ such that
$\bigcup_{\gamma}U_{\gamma}=\bar{M}_{0}$. For each $U_{\gamma}$,
there is a smooth pluri-subharmonic function $v_{\gamma}$ on
$U_{\gamma}$ such that $
\bar{\pi}^{*}\omega_{0}=\sqrt{-1}\partial\overline{\partial}v_{\gamma}$,
and, for any $E_{\alpha}$, there is a holomorphic function
$f_{\gamma,\alpha}$ with $f_{\gamma,\alpha}^{-1}(0)=E_{\alpha}\cap
U_{\gamma}$. Note that $\log |f_{\gamma,\alpha}| $ is a
pluri-subharmonic function, and $E_{\alpha}\cap U_{\gamma}$ is a
pluripolar set.  Since $v_{\gamma}+ \bar{\pi}^{*}\varphi_{0}$ is a
bounded pluri-subharmonic function on $ U_{\gamma}\backslash
E_{\alpha}$, $\bar{\pi}^{*}\varphi_{0}$ can be extended uniquely to
 a function $\bar{\varphi}_{0,\gamma}$ such that
 $v_{\gamma}+\bar{\varphi}_{0,\gamma}$ is a pluri-subharmonic function on $
 U_{\gamma}$ by Theorem 5.24 in \cite{De}. By the uniqueness, there
 is a $\bar{\pi}^{*}\omega_{0}$-pluri-subharmonic function
 $\bar{\varphi}_{0}$ on $\bar{M}_{0}$ satisfying that
 $\bar{\varphi}_{0}|_{U_{\gamma}}=\bar{\varphi}_{0,\gamma}$.

 Now we prove that $\bar{\varphi}_{0}\in L^{\infty}(\bar{M}_{0})$.  From
 the proof of Theorem 5.23 in \cite{De},
 $(v_{\gamma}+\bar{\varphi}_{0,\gamma})(x)=\nu^{*}(x)=\lim\limits_{\epsilon\rightarrow 0}
 \sup\limits_{B(x,\epsilon)}\nu$,  where
 $\nu(x)=\sup\limits_{\delta}\nu_{\delta}(x)$, $\nu_{\delta}=v_{\gamma}+
 \bar{\pi}^{*}\varphi_{0}+ \delta \log |f_{\gamma,\alpha}|$ on  $ U_{\gamma}\backslash
E_{\alpha}$, and  $\nu_{\delta}\equiv -\infty$ on $U_{\gamma}\cap
E_{\alpha} $. By assuming $|f_{\gamma,\alpha}|<1$, we have
$\nu=v_{\gamma}+
 \bar{\pi}^{*}\varphi_{0}$ on  $ U_{\gamma}\backslash
E_{\alpha}$, and  $\nu \equiv -\infty$ on $U_{\gamma}\cap E_{\alpha}
$. Thus $C_{1}< \inf_{U_{\gamma}\backslash E_{\alpha}}(v_{\gamma}+
 \bar{\pi}^{*}\varphi_{0})\leq v_{\gamma}+\bar{\varphi}_{0,\gamma}
 \leq  \sup_{U_{\gamma}\backslash E_{\alpha}}(v_{\gamma}+
 \bar{\pi}^{*}\varphi_{0})<C_{2}$, and $\bar{\varphi}_{0}\in
 L^{\infty}(\bar{M}_{0})$.  Thus $(\bar{\pi}^{*}\omega_{0}+ \sqrt{-1}\partial
\overline{\partial}\bar{\varphi}_{0}  )^{n} $ is a probability
measure (c.f. \cite{BT}), and $(\bar{\pi}^{*}\omega_{0}+
\sqrt{-1}\partial \overline{\partial}\bar{\varphi}_{0}  )^{n}
=\frac{(-1)^{\frac{n^2}{2}}}{\mathcal{V}_{0}}\bar{\pi}^{*}\Omega_{0}\wedge
\bar{\pi}^{*}\overline{\Omega}_{0}$ on $\bar{M}_{0}\backslash
\bar{\pi}^{-1}(S)$.

Now we prove that $\bar{\varphi}_{0}$ is the unique solution of

\begin{equation}\label{6.20}
 (\bar{\pi}^{*}\omega_{0}+ \sqrt{-1}\partial \overline{\partial}\bar{\varphi}_{0})^{n} =\frac{(-1)^{\frac{n^2}{2}}}{\mathcal{V}_{0}}\bar{\pi}^{*}\Omega_{0}\wedge
\bar{\pi}^{*}\overline{\Omega}_{0}.\end{equation}
By Lemma 6.4 in \cite{EGZ}, there is a function $f\in
L^{1+\varepsilon}((\bar{\pi}^{*}\omega_{0})^{n})$, for an $\varepsilon>0$, such that $d\mu=f (\bar{\pi}^{*}\omega_{0})^{n}$, where $d\mu = \frac{(-1)^{\frac{n^2}{2}}}{\mathcal{V}_{0}}\bar{\pi}^{*}\Omega_{0}\wedge
\bar{\pi}^{*}\overline{\Omega}_{0}$.  Note that, for any smooth function $\chi \geq 0$ on $\bar{M}_{0}$,

\[
0\leq \lim\limits_{\sigma \rightarrow  0}\int_{
\bar{\pi}^{-1}(B_{g_0}(S,\sigma))}\chi d\mu\leq
C\lim\limits_{\sigma \rightarrow  0}\int_{
\bar{\pi}^{-1}(B_{g_0}(S,\sigma))}f
(\bar{\pi}^{*}\omega_{0})^{n}
\]
\[
\leq C \lim\limits_{\sigma \rightarrow 0}{\rm Vol}_{g_{0}} (B_{g_{0}} (S,\sigma))^{\frac{\varepsilon}{1+\varepsilon}}=0,
\]
where $B_{g_{0}}(S,\sigma)=\{x\in M_{0}| d_{g_{0}}(x,
S)<\sigma\}$. Hence

\[
\int_{\bar{M}_{0}}\chi d\mu =
\lim\limits_{\sigma \rightarrow 0} \left(\int_{\bar{M}_{0}\backslash \bar{\pi}^{-1} (B_{g_0}(S,\sigma))}\chi d\mu + \int_{\bar{\pi}^{-1}(B_{g_0} (S,\sigma))}\chi d\mu\right)
\]
\[
= \int_{\bar{M}_{0}\backslash \bar{\pi}^{-1}(S)}\chi d\mu = \int_{\bar{M}_{0}\backslash \bar{\pi}^{-1}(S)}\chi
(\bar{\pi}^{*}\omega_{0}+ \sqrt{-1}\partial
\overline{\partial}\bar{\varphi}_{0}  )^{n} \leq
\int_{\bar{M}_{0}}\chi (\bar{\pi}^{*}\omega_{0}+ \sqrt{-1}\partial
\overline{\partial}\bar{\varphi}_{0}  )^{n}.
\]
Hence $d\mu\leq (\bar{\pi}^{*}\omega_{0}+ \sqrt{-1}\partial
\overline{\partial}\bar{\varphi}_{0}  )^{n}$ on $M_0$ in the distribution
sense. Since $$\int_{\bar{M}_{0}} (\bar{\pi}^{*}\omega_{0}+
\sqrt{-1}\partial \overline{\partial}\bar{\varphi}_{0}
)^{n}=\int_{\bar{M}_{0}} \bar{\pi}^{*}\omega_{0}^{n}=1=
\int_{\bar{M}_{0}} d\mu,$$ we obtain
\begin{equation}\label{6.20} \frac{(-1)^{\frac{n^2}{2}}}{\mathcal{V}_{0}}\bar{\pi}^{*}\Omega_{0}\wedge
\bar{\pi}^{*}\overline{\Omega}_{0}=d\mu= (\bar{\pi}^{*}\omega_{0}+
\sqrt{-1}\partial \overline{\partial}\bar{\varphi}_{0}
)^{n}\end{equation}in the distribution sense.   From  the following
  theorem,   $\bar{\varphi}_{0}$ is the unique solution of
(\ref{6.20}).

\begin{theorem}[Proposition 1.4 and Proposition 3.1 in \cite{EGZ}]\label{t5.06}
Let $\omega$ be a semi-positive $(1,1)$-form on a compact K\"{a}hler
 $n$-manifold $X$, and $f \in L^{1+\varepsilon}(\omega^{n})$,
$\varepsilon >0$.  Then there is a unique  function $\varphi \in
L^{\infty}(X)$ such that $$(\omega + \sqrt{-1}\partial
\overline{\partial}\varphi)^{n}= f \omega^{n}, \ \ \ \sup_{X}
\varphi=0.  $$
\end{theorem}

 Furthermore, from  \cite{EGZ},
$\bar{\varphi}_{0}$ is a continues function, and $\varphi_{0}$ can
be extended to a continues function on $M_{0}$, denoted by
$\varphi_{0}$ also, such that
$\bar{\varphi}_{0}=\bar{\pi}^{*}\varphi_{0}$.  Then $\varphi_{0}$ is
a solution of (\ref{6.22}). By the uniqueness of the solution of
(\ref{6.22}), $\varphi_{0}=\hat{\varphi}_{0}$, and
$F_{K_{i},k}^{*}\varphi_{t_{k}}$ $C^{\infty}$-converges to a smooth
function  $\hat{\varphi}_{0}$ on $K_{i}$, i.e. we do not need to
take  a subsequence of $F_{K_{i},k}^{*}\varphi_{t_{k}}$.   We obtain
the first part of the theorem.

It remains to show the uniform diameter bound. Note that, by Lemma
\ref{t5.4}, there are $C',C'_K >0$ independent of $t$ such that $C'g_t \leq \tilde{g}_t \leq (C'_K)^{-1} g_t$ on $K$. Then there is $0< r\leq 1$ independent of $t$ such that $B_{g_{t}}(p_{t}, C'_K r)\subset B_{\tilde{g}_{t}}(p_{t}, r)\subset K \subset\subset
\mathcal{M}\backslash S$ for certain $p_t \in K \cap M_t$. Thus
\[
{\rm Vol}_{\tilde{g}_{t}}(B_{\tilde{g}_{t}}(p_{t}, r))\geq
{\rm Vol}_{\tilde{g}_{t}}(B_{g_{t}}(p_{t}, C'_K r))\geq (C')^n
{\rm Vol}_{g_{t}}(B_{g_{t}}(p_{t}, C'_K r))> C
\]
for a constant $C>0$
independent of $t$. Thus
\[
{\rm Vol}_{\tilde{g}_{t}}(B_{\tilde{g}_{t}}(p_{t}, 1))\geq C, \quad \mbox{ and } \quad {\rm diam}_{\tilde{g}_{t}}(M_{t})< \bar{C}<\infty
\]
by Lemma \ref{t3.03} and the same arguments as in the proof of
Theorem \ref{t3.1}.
\hfill $\Box$\\

By (\ref{6.29}), and  Gromov's precompactness   theorem (c.f.  \cite{G1}), for any $t_{k}\rightarrow 0$ with $\{t_{k}\}\subset \Lambda $,  by passing to a
  subsequence, $\{(M, \tilde{g}_{t_{k}})\}$ converges to a compact length  metric
  space $(Y, d_{Y})$ in the Gromov-Hausdorff topology.   By the same arguments in
the proof of Lemma \ref{t4.2},  we obtain an embedding $f:
(M_{0}\backslash S, \tilde{g}_{0}) \rightarrow (Y, d_{Y})$,
which is a local isometry.

\begin{con}\label{t5.11} There is a homeomorphism $\tilde{f}:
M_{0}\rightarrow Y$ such that $\tilde{f}|_{M_{0}\backslash
S}=f$.
\end{con}

\begin{remark}
If $n=2$, this conjecture is true by the same arguments as in
Section 4, since $M_{0}$ is a K3 orbifold.
\end{remark}

For conifold singularity, locally $M_t = \{\pi(z) = z_0^2 + \cdots + z_n^2 =t\} \subset \mathbb{C}^{n+1}$, take $\mathfrak{p}(z) = (z_1, \cdots, z_n)$, $f = z_0$. condition (\ref{5.6}) can be verified directly, therefore, we have a direct proof of corollary \ref{t4}.\\

\noindent{\bf Direct proof of Corollary \ref{t4}:} $M_{0}$ has only finite many ordinary double points as singular points. Since the local smoothing of an ordinary double point is unique, when $x_{\alpha}\in S$ is an ordinary double point, by possibly taking $U_{\alpha}$ smaller at the beginning, there is coordinate $z = (z_0, \cdots, z_n)$ on the neighborhood $U_{\alpha}$ of $x_{\alpha}$ such that $x_{\alpha}=(0,\cdots , 0)$, and  $\pi (z)= z_0^{2}+ \cdots + z_n^{2}$.

\[
\int_{U_\alpha \cap M_t} |f_\alpha|^{-2c} d\mu_t = \int_{\mathfrak{p} (U_\alpha \cap M_t)} |f_\alpha|^{-2(1+c)} d\mu_{\mathbb{C}^n} \leq \int_{B_1} \frac{d\mu_{\mathbb{C}^n}}{|t-(z_1^2 + \cdots + z_n^2)|^{1+c}}
\]
It is straightforward to verify that this integral is bounded independent of $t \in \Delta$.\\

\[
\int_{B_1} \frac{d\mu_{\mathbb{C}^n}}{|t-(z_1^2 + \cdots + z_n^2)|^{1+c}} = \int_{B_{\frac{1}{\sqrt{|t|}}}} \frac{|t|^{n-1-c}d\mu_{\mathbb{C}^n}}{|1-(z_1^2 + \cdots + z_n^2)|^{1+c}}
\]
\[
\leq \left( \int_{B_R} + \sum_{i=1}^n \int_{D_i} \right) \frac{|t|^{n-1-c}d\mu_{\mathbb{C}^n}}{|1-(z_1^2 + \cdots + z_n^2)|^{1+c}} = I_0 + \sum_{i=1}^n I_i
\]
where $D_i = \{z'\in B_{\frac{1}{\sqrt{|t|}}} \setminus B_R: n|z_i| \geq |z'|\}$. Clearly, $I_0 \leq C$. On $D_i$, change the coordinate from $z' = (z_i, z'_i)$ to $(z_0, z'_i)$ by $\pi(z) =1$, we get $|z_0|^2 \leq 1+ |z'|^2 \leq 1 + 1/t$. For $c>0$ small,\\

\[
I_i \leq \int_{B_{\frac{1}{\sqrt{|t|}}}} \frac{|t|^{n-1-c}d\mu(z_0) d\mu(z'_i)}{|z_0|^{2c}\max(R^2, |z'|^2)} \leq |t|^{n-1-c} \int_{B_{\frac{2}{\sqrt{|t|}}}} \frac{d\mu(z_0)}{|z_0|^{2c}} \int_{B_{\frac{1}{\sqrt{|t|}}}} \frac{d\mu(z'_i)}{\max(R^2, |z'_i|^2)}
\]
\[
\leq C |t|^{n-1-c} |t|^{c-1} |t|^{-(n-2)} = C
\]
This verifies the condition (\ref{5.6}) for all $t\in \Delta$. Then Theorem \ref{t5.5} implies the Corollary \ref{t4}.
\hfill $\Box$\\

\noindent{\bf Proof of Theorem \ref{t1.1}:} It is straightforward to see that under the condition (\ref{1.1}) for $\Lambda = \Delta$, proposition \ref{5.7} can be proved with the condition (\ref{5.6}) satisfied for all $t\in \Delta$. Then Theorem \ref{t5.5} implies the Theorem \ref{t1.1}.
\hfill $\Box$\\

\begin{lm}
\label{5.8}
If $\mathcal{M}$ is locally homogeneous, proposition \ref{5.7} can be strengthened so that there exists $c_1, C_2>0$ such that for $c\in [0,c_1]$,
\[
\int_{U_\alpha \cap M_{\Delta (\sigma)}} \frac{d\mu}{|f_\alpha|^{2c}} \leq C_2 |\Delta (\sigma)|.
\]
Then for any $\epsilon>0$ and $c\in [0, c_1]$, there is $C_1>0$ such that $\Lambda = \Lambda (c,C_1)$ satisfies $|\Lambda \cap \Delta (\sigma)| \geq  (1-\epsilon) |\Delta (\sigma)|$ for $\sigma>0$ small. In particular, 0 is an accumulating point of $\Lambda$.
\end{lm}
\noindent{\bf Proof:} When $\mathcal{M}$ is locally homogeneous, by possibly taking $U_{\alpha}$ smaller at the beginning, Theorem \ref{cy2} can be applied to $M=U_\alpha$ and $\psi = f_\alpha$ to show that there exists $c_1, C_2>0$ such that for $c\in [0,c_1]$,

\[
\int_{\Delta (\sigma)} d\mu_{\mathbb{C}} \int_{U_\alpha \cap M_t} \frac{d\mu_t}{|f_\alpha|^{2c}} = \int_{U_\alpha \cap M_{\Delta (\sigma)}} \frac{d\mu}{|f_\alpha|^{2c}} \leq C_2 |\Delta (\sigma)|.
\]
According to the definition of $\Lambda$,

\[
C_1 |\Delta (\sigma) \setminus \Lambda| \leq \int_{\Delta (\sigma)} d\mu_{\mathbb{C}} \int_{U_\alpha \cap M_t} \frac{d\mu_t}{|f_\alpha|^{2c}} \leq C_2 |\Delta (\sigma)|.
\]
Hence, it is sufficient to take $C_1 = C_2/\epsilon$.
\hfill $\Box$\\

\noindent\noindent{\bf Proof of Theorem \ref{t3}:} By Lemma \ref{5.8}, $0$ is an accumulating point of $\Lambda$, there exists sequence $t_k \rightarrow 0$ in $\Lambda$. Then Theorem \ref{t5.5} implies the Theorem \ref{t3}.
\hfill $\Box$\\

\begin{lm}
\label{5.4}
If $(\mathcal{M}, \pi)$satisfies the condition (1.2), proposition \ref{5.7} can be strengthened so that there exists $c_1, C_1>0$ such that for $c\in [0,c_1]$ and $t\in \Delta$,
\[
\int_{U_\alpha \cap M_t} \frac{d\mu_t}{|f_\alpha|^{2c}} \leq C_1.
\]
In another word, $\Lambda = \Lambda_{c,C_1} = \Delta$.
\end{lm}
\noindent{\bf Proof:} When $(\mathcal{M}, \pi)$satisfies the condition (1.2), by possibly taking $U_{\alpha}$ smaller at the beginning, Theorem \ref{cy2} and Proposition \ref{cy9} can be applied to $M=U_\alpha$ and $\psi = f_\alpha$ to show that there exists $c_1, C_1>0$ such that for $c\in [0,c_1]$ and $t\in \Delta$,

\[
\int_{U_\alpha \cap M_t} \frac{d\mu_t}{|f_\alpha|^{2c}} \leq C_1.
\]
According to the definition of $\Lambda$, this means $\Lambda = \Lambda_{c,C_1} = \Delta$.
\hfill $\Box$\\

\noindent{\bf Proof of Theorem \ref{t6}:} Lemma \ref{5.4} and Theorem \ref{t5.5} implies the Theorem \ref{t6}.
\hfill $\Box$\\

\noindent\noindent{\bf Proof of Corollary   \ref{t5}:} Note that
\[
Vol_{\tilde{g}_{t}}(M_{t})=\frac{1}{n!}\int_{M_{t}}\tilde{\omega}_{t}^{n}=
\frac{(-1)^{\frac{n^2}{2}}}{n!\mathcal{V}_{t}}\int_{M_{t}}\Omega_{t}\wedge
\overline{\Omega}_{t}
\]
converges to
\[
\frac{(-1)^{\frac{n^2}{2}}}{n!\mathcal{V}_{0}}\int_{M_{0}\backslash
S}\Omega_{0}\wedge
\overline{\Omega}_{0}=Vol_{\tilde{g}_{0}}(M_{0}\backslash S),
 \]
 when $t\rightarrow 0$.
  By  (\ref{6.29}) and
  Bishop-Gromov comparison theorem, we obtain that,  for any metric ball
  $B_{\tilde{g}_{t}}(r)$, $t\neq 0$,  \begin{equation}\label{6.30}Vol_{\tilde{g}_{t}}(B_{\tilde{g}_{t}}(r))\geq
  \frac{Vol_{\tilde{g}_{t}}(M)}{diam_{\tilde{g}_{t}}^{2n}(M)}r^{2n}\geq  Cr^{2n},  \end{equation}
  where $C$ is a   constant independent of $t$.  Since    $\dim_{\mathcal{H}}S < 2n$,  $Vol_{\tilde{g}_{0}}(M_{0}\backslash
  S)=\mathcal{H}^{2n} (M_{0})$. We obtain the conclusion from
  Theorem \ref{t4.1},  Theorem \ref{t3} and Theorem \ref{t4}.
\hfill $\Box$\\

\setcounter{equation}{0}\section{Collapsing of a  Calabi-Yau
threefold  } The purpose of this section is to prove Theorem
\ref{t6.1}.\\

\noindent{\bf Proof of Theorem    \ref{t6.1}:} Let $W_{i}=\mathbb{CP}^{2}\times \mathbb{C}$,
$i=1,2$, and $W=W_{0}\cup W_{1}$ by identifying
$([x_{0},y_{0},z_{0}],u_{0})\in W_{0}$ with
$([x_{1},y_{1},z_{1}],u_{1})\in W_{1}$ if and only if
$u_{0}u_{1}=1$, $u^{4}_{0}x_{1}=x_{0}$, $u^{6}_{0}y_{1}=y_{0}$ and
$z_{1}=z_{0}$. Note that $\mathbb{CP}^{1}=\mathbb{C}\cup \mathbb{C}$
by identifying $u_{0}\in \mathbb{C}$ with $u_{1}\in \mathbb{C}$ if
and only if $u_{0}u_{1}=1$. There is a holomorphic map $\Psi:
W\rightarrow \mathbb{CP}^{1}$ given by $\Psi:
([x_{i},y_{i},z_{i}],u_{i})\mapsto u_{i}$. For a point
$\tau=(\tau_{1}, \cdots, \tau_{8}, \sigma_{1}, \cdots,
\sigma_{12})\in \mathbb{R}^{20}$, define $\mathfrak{g}(u)=
\prod_{\nu=1}^{8}(u-\tau_{\nu})$, and
$\mathfrak{h}(u)=\prod_{\nu=1}^{12}(u-\sigma_{\nu})$. Let $X_{\tau}$
be the
 algebraic surface given by
\begin{eqnarray*} & &
f_{0}=y_{0}^{2}z_{0}-4x_{0}^{3}+\mathfrak{g}(u_{0})x_{0}z_{0}^{2}+\mathfrak{h}(u_{0})z_{0}^{3}=0,\
\ \ \ \ {\rm and }
\\ & &f_{1}= y_{1}^{2}z_{1}-4x_{1}^{3}+u_{1}^{8}\mathfrak{g}(u_{1}^{-1})x_{1}z_{1}^{2}+
u_{1}^{12}\mathfrak{h}(u_{1}^{-1})z_{1}^{3}=0.
\end{eqnarray*}
 By Section 5 in  \cite{Kod}, $(X_{\tau}, \Psi|_{X_{\tau}})$ is an elliptic K3  surface, and there is
a holomorphic section $\sigma:
  \mathbb{CP}^{1}\rightarrow X_{\tau}$ given by  $u_{0}\mapsto ([0, u_{0}^{6},0],u_{0})\in
  W_{0}$ and  $u_{1}\mapsto ([0, 1,0],u_{1})\in
  W_{1}$.
  Note that conjugate maps $\iota_{1}: W_{i}\rightarrow W_{i}$
given by $([x_{i},y_{i},z_{i}],u_{i})\mapsto
([\bar{x}_{i},\bar{y}_{i},\bar{z}_{i}],\bar{u}_{i})$, and
$\iota_{2}: \mathbb{C}\rightarrow \mathbb{C}$ given by
$u_{i}\mapsto \bar{u}_{i}$ preserve $X_{\tau}$, $\Psi$ and $\sigma$.
Hence $\iota=( \iota_{1}, \iota_{2}) $ induces an anti-holomorphic
involution on $(X_{\tau}, \Psi|_{X_{\tau}})$. We denote  $I$  the
complex structure of $X_{\tau}$. There is a holomorphic  volume form
$$\Omega_{I}=du_{0}\wedge
(z_{0}dx_{0}-x_{0}dz_{0})/\partial_{y_{0}}f_{0}=du_{1}\wedge
(z_{1}dx_{1}-x_{1}dz_{1})/\partial_{y_{1}}f_{1},$$ on $X_{\tau}$,
which satisfies that $
\iota_{1}^{*}\Omega_{I}=\overline{\Omega}_{I}$ (c.f. Section 5 in
\cite{Kod}).

\begin{lm}\label{t6.2}  There is a sequence of Ricci-flat K\"{a}hler forms
$\omega_{k}$ on $X_{\tau}$ such that
$\iota_{1}^{*}\omega_{k}=-\omega_{k}$, $2\omega_{k}^{2}=
\Omega_{I}\wedge \overline{\Omega}_{I} $ and, for any $y\in
\mathbb{CP}^{1}$,
$$\epsilon_{k}=\int_{\Psi|_{X_{\tau}}^{-1}(y)}\omega_{k}\rightarrow
0,$$ when $k\rightarrow \infty$.
\end{lm}

\begin{proof} Note that $H^{2}(W_{i}, \mathbb{R})\cong H^{2}(\mathbb{CP}^{2},
\mathbb{R})$, $H^{1}(W_{0}\cap W_{1}, \mathbb{R})\cong
H^{1}(\mathbb{C}^{*}, \mathbb{R})$, and they  are  generated by the
Fubini-Study metric $\omega_{FS}$ on $\mathbb{CP}^{2}$ and
Im$\frac{dz}{z}$ on $\mathbb{C}^{*}=\mathbb{C}\backslash \{0\}$
respectively.
 Thus   $\iota_{1}^{*}: H^{j}(W_{i}, \mathbb{R})\rightarrow  H^{j}(W_{i}, \mathbb{R})$, $j=1,2$, is
 $\iota_{1}^{*}\gamma=-\gamma$, for any  $\gamma  \in H^{j}(W_{i}, \mathbb{R})$.  By Mayer-Vietoris exact
 sequence, the following diagram commutes
  $$\rightarrow H^{1}(W_{0}\cap W_{1},
 \mathbb{R})\stackrel{h_{1}}\rightarrow
 H^{2}(W, \mathbb{R})\stackrel{h_{2}}\rightarrow H^{2}(W_{0}, \mathbb{R})\oplus H^{2}(W_{1},
 \mathbb{R})\stackrel{h_{3}}\rightarrow H^{2}(W_{0}\cap W_{1}, \mathbb{R})
  $$ $$\iota_{1}^{*}=-{\rm id}\downarrow \ \ \ \ \ \ \ \ \ \ \ \ \ \ \ \ \ \ \ \iota_{1}^{*}\downarrow  \ \ \ \ \ \ \ \
   \ \  \ \ \   \ \ \iota_{1}^{*}=-{\rm id}\downarrow \ \ \
 \  \ \ \ \ \ \  \ \ \ \ \ \ \ \ \  \  \ \ \
  \iota_{1}^{*}\downarrow   \ \ \ \  \ \ \ \ \ \ \ $$  $$\rightarrow H^{1}(W_{0}\cap W_{1}, \mathbb{R})\stackrel{h_{1}}\rightarrow
 H^{2}(W, \mathbb{R})\stackrel{h_{2}}\rightarrow H^{2}(W_{0}, \mathbb{R})\oplus H^{2}(W_{1},
 \mathbb{R})\stackrel{h_{3}}\rightarrow H^{2}(W_{0}\cap W_{1}, \mathbb{R})
  .$$ Thus  we have that $\iota_{1}^{*}: H^{2}(W, \mathbb{R})\rightarrow  H^{2}(W, \mathbb{R})$
  is   given by
 $\iota_{1}^{*}=-{\rm id}$. Note that  $H^{1}(W_{i},
 \mathbb{R})=\{0\}$, and $h_{3}([\omega_{0}],
 [\omega_{1}])=[\omega_{0}-\omega_{1}]$. Thus,
   $\rm Im \it h_{2}= \rm Ker \it h_{3}= \mathbb{R}\cdot ([\omega_{FS}], [\omega_{FS}])
 $,   $h_{1}$ is injective, and $ H^{2}(W,
 \mathbb{R})\cong \rm Im \it h_{1} \oplus \rm Im \it h_{2}\cong
 \mathbb{R}^{2}$.  As $W$ admits K\"{a}hler metrics, we have $2=\dim
 H^{2}(W,
 \mathbb{R})= 2h^{2,0}+h^{1,1}$. Thus $h^{2,0}=0$, and   $H^{2}(W,
 \mathbb{R})=H^{1,1}(W,
 \mathbb{R}) $.  Furthermore, we have two generators of $H^{1,1}(W,
 \mathbb{R}) $, $\alpha= [\Psi^{*}\omega'_{FS}]$, where  $ \omega'_{FS}$ is the Fubini-Study metric on $\mathbb{CP}^{1} $,  and $\beta $, which
satisfies that, for any $y\in \mathbb{CP}^{1} $,
$i_{y}^{*}\beta=[\omega_{FS}]\in H^{2}(\mathbb{CP}^{2},
 \mathbb{R})$ where  $i_{y}: \mathbb{CP}^{2}=\Psi^{-1}(y)
\hookrightarrow W$ is the inclusion. Since $\Psi^{*}\omega'_{FS} $
is a semi-positive form, the K\"{a}hler cone of $W$ is
$\mathbb{K}_{W}=\{a\alpha+b\beta| b>0, a>k_{0}b\}$ for a constant
$k_{0}$.   By  $\alpha^{2}=0$,
$$C_{\alpha\beta}=\langle \alpha\wedge \beta, [X_{\tau}]\rangle=\langle \alpha\wedge (2k_{0}\alpha+\beta),
[X_{\tau}]\rangle=\int_{X_{\tau}}\Psi^{*}\omega'_{FS}\wedge
\omega'=\int_{X_{\tau}}|d\Psi|_{X_{\tau}}|^{2}\omega'^{2}>0,$$ where
$\omega'$ is a K\"{a}hler form representing  $2k_{0}\alpha+\beta$.

  If  $\omega_{s}$ are  the
K\"{a}hler forms such that $[\omega_{s}]=\alpha+s\beta$, $s\in (0,
\frac{1}{2|k_{0}|}]$, then we have
$$\mu(s)=\int_{X_{\tau}}\omega_{s}^{2}=2sC_{\alpha\beta}+s^{2}\langle\beta^{2}, [X_{\tau}] \rangle, \ \ \ {\rm and}
\ \ $$ $$ \ \int_{
\Psi|_{X_{\tau}}^{-1}(y)}\omega_{s}=s\langle\beta,
[\Psi|_{X_{\tau}}^{-1}(y)] \rangle =s\int_{
\Psi|_{X_{\tau}}^{-1}(y)}\omega_{FS}. $$ If
$\bar{\omega}_{s}=\mu(s)^{-\frac{1}{2}}\omega_{s}$, then
$\iota_{1}^{*}[\bar{\omega}_{s}]=-[\bar{\omega}_{s}]$,
\begin{equation}\label{6.1}\int_{X_{\tau}}\bar{\omega}_{s}^{2}=1, \ \ \  {\rm and} \ \ \ \int_{
\Psi|_{X_{\tau}}^{-1}(y)}\bar{\omega}_{s} =\mu(s)^{-\frac{1}{2}}
s\int_{ \Psi|_{X_{\tau}}^{-1}(y)}\omega_{FS}\rightarrow 0,
\end{equation}when $s\rightarrow 0$. Hence
$\iota_{1}^{*}[\bar{\omega}_{s}|_{X_{\tau}}]=-[\bar{\omega}_{s}|_{X_{\tau}}]$
in $H^{1,1}(X_{\tau}, \mathbb{R}) $.  Let $s_{k}\rightarrow 0$,
and $\omega_{k}$ be the Ricci-flat K\"{a}hler forms representing
$[\bar{\omega}_{s_{k}}|_{X_{\tau}}]$.  By the uniqueness of the
Ricci-flat K\"{a}hler form in a K\"{a}hler class, we obtain that
$\iota_{1}^{*}\omega_{k}=-\omega_{k}$. By  (\ref{6.1}), and
re-scaling $\omega_{k}$ if necessary,  we obtain the conclusion.
\end{proof}

Note that, for any $k$, $(X_{\tau}, \omega_{k},  \Omega_{I})$ is a
hyper-K\"{a}hler manifold. By re-scaling $ \Omega_{I}$ if necessary,
$ \omega_{k}^{2}= ({\rm Re}\Omega_{I})^{2}=({\rm
Im}\Omega_{I})^{2}$. By using  hyper-K\"{a}hler rotation,  we can
find a new complex structure $J_{k}$ with a holomorphic volume form
$$\Omega_{J_{k}}={\rm Im}\Omega_{I}+\sqrt{-1}\omega_{k}, \ \ \ {\rm
and \ \ a} \ \ {\rm  K\ddot{a}hler \ \ form } \ \
\omega_{J_{k}}={\rm Re}\Omega_{I}. $$  Since
$\iota_{1}^{*}\omega_{J_{k}}=\omega_{J_{k}}$ and
$\iota_{1}^{*}\Omega_{J_{k}}=-\Omega_{J_{k}}$, $\iota_{1}$ is a
holomorphic involution of $(X_{\tau}, J_{k})$.  Let
$T^{2}_{k}=\mathbb{C}/(\epsilon_{k}^{-\frac{1}{2}} \mathbb{Z}+
\sqrt{-1}\epsilon_{k}^{\frac{1}{2}}\mathbb{Z})$, and
 $\iota_{3} $ be  the  holomorphic involution on $ T^{2}_{k}$ given by
 $z\mapsto -z$.  The  holomorphic involution $\iota=(\iota_{1},\iota_{3}) $ on $X_{\tau}\times  T^{2}_{k} $
  preserves  the K\"{a}hler form $\hat{\omega}_{k}=\omega_{J_{k}}+\sqrt{-1}dz\wedge d\bar{z} $ and the  holomorphic volume
  form $\hat{\Omega}_{k}=\Omega_{J_{k}}\wedge dz $, i.e.
  $$\iota^{*}\hat{\omega}_{k}=\hat{\omega}_{k}, \ \ {\rm and}  \ \ \iota^{*}\hat{\Omega}_{k}=\hat{\Omega}_{k}.
  $$ Hence $(X_{\tau}\times  T^{2}_{k}) /\langle \iota \rangle$ is a
  Calabi-Yau orbifold with  $H^{2,0}((X_{\tau}\times  T^{2}_{k}) /\langle \iota \rangle)={0}
  $, the K\"{a}hler form $\hat{\omega}_{k} $ (resp. the  holomorphic volume
  form $\hat{\Omega}_{k}$) induces an orbifold K\"{a}hler  form $\hat{\omega}_{k} $ (resp. a  holomorphic volume
  form $\hat{\Omega}_{k}$) on $(X_{\tau}\times  T^{2}_{k}) /\langle \iota
  \rangle$, denoted still by $\hat{\omega}_{k} $ and
  $\hat{\Omega}_{k}$. For any $k$, let   $M_{k}$ be  a crepant resolution of
 $(X_{\tau}\times  T^{2}_{k}) /\langle
\iota
  \rangle $. Note that the homeomorphism type of   $M_{k}$ is
  indpendent of $k$, however, the complex structures on  $M_{k}$
  are different for different $k$.

 Now we follow the arguments in Section 5 of \cite{Kod}, and take
$(\tau_{1}, \cdots, \tau_{8}, \sigma_{1}, \cdots, \sigma_{12})$
satisfy  that $\tau_{\lambda}\neq\tau_{\nu}$, $\tau_{\lambda}\neq
\sigma_{\nu}$, and $\sigma_{\lambda}\neq \sigma_{\nu}$,
$\mathfrak{f}(u)=\frac{\mathfrak{g}(u)^{3}}{\mathfrak{g}(u)^{3}-27\mathfrak{h}(u)^{2}}
$ has no multiple pole, where
$\mathfrak{g}(u)=\prod_{\nu=1}^{8}(u-\tau_{\nu})$ and
$\mathfrak{h}(u)=\prod_{\nu=1}^{12}(u-\sigma_{\nu})$. Then
 all singular fibers of $\Psi|_{X_{\tau}}: X_{\tau}\rightarrow \mathbb{CP}^{1}$ are
  type $I_{1}$ (c.f. Section 5 in \cite{Kod}), which implies that
 $(X_{\tau}, \Psi|_{X_{\tau}})$ is an elliptic K3 surface with all singular fibers
 of   type $I_{1}$, and a holomorphic section $\sigma$.

  Let $\omega_{k}$ be a sequence of Ricci-flat
K\"{a}hler forms on $X_{\tau}$ given in Lemma \ref{t6.1}, and
$\hat{g}_{k}$ be the corresponding K\"{a}hler metrics.   By
\cite{GW2},  a subsequence of   $(X_{\tau}, \epsilon_{k}\hat{g}_{k})
$ converges to $(\mathbb{CP}^{1}, h)$ in the Gromov-Hausdorff
topology, where $h$ is   a singular Riemannian metric $h$  on
$\mathbb{CP}^{1}$ with $24$ singular points $\{q_{i}, i=1, \cdots ,
24\} $.  Furthermore, $ \Psi|_{X_{\tau}}$ and $\sigma$ are Hausdorff
approximations from the proof of Theorem 6.4 in \cite{GW2}.  Since
$\iota_{1}^{*}\hat{g}_{k}=\hat{g}_{k}$, $ \Psi|_{X_{\tau}}\circ
\iota_{1}=\iota_{2} \circ \Psi|_{X_{\tau}}$ and $\sigma \circ
\iota_{2}=\iota_{1} \circ \sigma$, we obtain $\iota_{2}^{*}h=h$.
  Note that, under  the hyperK\"{a}hler rotation, for any $k$,
 $\hat{g}_{k}$ is still a K\"{a}hler metric corresponding to the complex
 structure $J_{k}$, whose  K\"{a}hler form is $\omega_{J_{k}} $.  Thus $(X_{\tau}\times T^{2}_{k},
  \epsilon_{k}(\hat{g}_{k}+dz \otimes d\overline{z})) $ converges
to $(\mathbb{CP}^{1}\times S^{1},h+d\theta^{2})$ in the
$\mathbb{Z}_{2}$-equivariant Gromov-Hausdorff topology, where
$S^{1}=\mathbb{R}/\mathbb{Z}$,  $\mathbb{Z}_{2}$ acts on
$X_{\tau}\times T^{2}_{k}$ by the involution $\iota=(\iota_{1},
\iota_{3})$,   acts on $\mathbb{CP}^{1}\times S^{1}$ by the
involution $\iota'=(\iota_{1}, \iota_{4})$, and
$\iota_{4}:S^{1}\rightarrow S^{1}$ is given  by $\theta \mapsto
-\theta$. If $\check{g}_{k}$ (resp. $\check{h}$) is the induced
Ricci-flat orbifold  K\"{a}hler metrics on $X_{\tau}\times
T^{2}_{k}/\langle \iota \rangle$ (resp. $\mathbb{CP}^{1}\times
S^{1}/\langle \iota' \rangle$ ) by $\epsilon_{k}(\hat{g}_{k}+dz
\otimes d\overline{z})$ (resp. $h+d\theta^{2}$), then
$(X_{\tau}\times T^{2}_{k}/\langle \iota \rangle, \check{g}_{k})$
converges to $(B, d_{B})$ in the Gromov-Hausdorff topology, where
$B=\mathbb{CP}^{1}\times S^{1}/\langle\iota' \rangle$, and $d_{B}$
is the distance function induced by $\check{h}$.  Let $ \Pi$ be the
union  of the singularity set of the
 orbifold $B $, and the    image of $\{q_{i}, i=1, \cdots , 24\} \times S^{1} $
  under the quotient map $\mathbb{CP}^{1}\times S^{1}\rightarrow B $.
 We denote $g_{B}=\check{h}|_{B\backslash \Pi}$ on
$B\backslash \Pi  $. By \cite{GW1}, $B$ is homeomorphic to $S^{3}$.
By Corollary  \ref{t2}, for any $k$, we have a Ricci-flat K\"{a}hler
metric $g_{k}$ on  $M_{k}$ such that
$$d_{GH}((X_{\tau}\times T^{2}_{k}/\langle \iota  \rangle,
\check{g}_{k} ), (M_{k}, g_{k} ))< \frac{1}{k}.$$ We obtain the
conclusion by the diagonal   arguments. $\Box$\\

\end{document}